\documentclass[reqno]{amsart}
\usepackage{amsfonts,amsmath,amsthm,amssymb,mathrsfs}
\usepackage{color}
\usepackage{amsmath,amssymb,amsfonts,bm}
\usepackage{graphicx}
\usepackage{newunicodechar}
\usepackage{xcolor}
\numberwithin{equation}{section}
\usepackage[pdfstartview=FitH,colorlinks=true]{hyperref}
\hypersetup{urlcolor=red, linkcolor=blue,citecolor=red}
\allowdisplaybreaks

\theoremstyle{plain}

\newtheorem{thm}{Theorem}[section]
\newtheorem{definition}[thm]{Definition}
\newtheorem{prop}[thm]{Proposition}

\newtheorem{lemma}[thm]{Lemma}
\newtheorem{remark}[thm]{Remark}

\begin{document}
\title[Kramers-Fokker-Planck equation]
{Analytic Gelfand-Shilov smoothing effect of fractional Kramers-Fokker-Planck equation
}

\author[C.-J. Xu \& Y. Xu]
{Chao-Jiang Xu and Yan Xu}

\address{Chao-Jiang Xu and Yan Xu
\newline\indent
School of Mathematics and Key Laboratory of Mathematical MIIT,
\newline\indent
Nanjing University of Aeronautics and Astronautics, Nanjing 210016, China
}
\email{xuchaojiang@nuaa.edu.cn; xuyan1@nuaa.edu.cn}

\date{\today}



\begin{abstract}
We study the Cauchy problem of the fractional Kramers-Fokker-Planck equation and show that the solution to the Cauchy problem enjoys an analytic Gelfand-Shilov regularizing effect for positive time.
\end{abstract}

\maketitle

\section{Introduction}

\par The Cauchy problem of fractional Kramers-Fokker-Planck equation reads
\begin{equation}\label{1-1}
\left\{
\begin{aligned}
 &\partial_t u+v\cdot\nabla_x u+(1+|v|^{2})^{\frac\gamma2}(1-\Delta_{v})^su+(1+|v|^{2})^{\frac\gamma2+s}u=f(t, x, v),\\
 &u|_{t=0}=u_0,
\end{aligned}
\right.
\end{equation}
where $u=u(t, x, v)\ge0$ is the density distribution function of particles depending on time $t\ge0$, the position variable $x\in\mathbb R^3$ and the velocity variable $v\in\mathbb R^3$ with $\gamma>-3$, $0<s\le1$. 

The Fokker-Planck equation was introduced by Fokker and Planck, to describe the evolution of the density of particles under the Brownian motion.  Recently, the results of the global hypoelliptic estimates were motivated by applications to the kinetic theory of gases. The Kramers-Fokker-Planck operator has been studied in~\cite{H-2, H-3}.

It is a linear model of kinetic equations such as the non-cutoff Boltzmann equation for $0<s<1$ and the Landau equation for $s=1$ with moderate soft potential. For $\gamma=0, s=1$, the above equation is the classical Kramers-Fokker-Planck equation. The $C^{\infty}$ regularity of weak solution for the spatially homogeneous Boltzmann equation without angular cutoff has been discussed in~\cite{D-1, D-2, G-2, V-1}. The Gevrey regularity of the local solution for the initial datum with the same Gevrey regularity has been established in~\cite{U-1}. The Gevrey regularity results for the spatially homogeneous Boltzmann equation have given in~\cite{ B-1, M-2}. Regularity results for the spatially homogeneous Landau equation have been studied in~\cite{L-1, M-1, M-3}.

The regularity problems for the solutions of the kinetic equation have been studied in~\cite{M-5, B-2}. For the spatially inhomogeneous problems,~\cite{A-1, A-2, L-2} concerned the $C^{\infty}$ regularity problems for linear spatially inhomogeneous Boltzmann equation without angular cutoff. The analytic smoothing effect of a solution for the spatially inhomogeneous Landau equation has been given in~\cite{M-4}. And~\cite{C-1, G-1, H-1} also concern the regularity problems for the spatially homogeneous Landau equation.
%In the Cauchy problem of \eqref{1-1}, the operator $(1-\Delta_{v})^s$, where for any $s\in]0,1[$, $(1-\Delta_{v})^s: \mathcal S(\mathbb R^n)\to L^{2}(\mathbb R^n)$ is defined via the Fourier transform
%\begin{equation*}
    %\begin{split}
        %(1-\Delta_{v})^su=\mathcal F^{-1}\left((1+|\xi|^{2})^{s}\mathcal Fu\right), \quad\xi\in\mathbb R^{n},
    %\end{split}
%\end{equation*}
%where $\mathcal F$ and $\mathcal F^{-1}$ are Fourier transform and its inverse,
%$$\mathcal F(\xi)=\int_{\mathbb R^{n}}e^{-ix\cdot\xi}u(x)dx, \quad u\in\mathcal S(\mathbb R^n).$$
\begin{definition}\label{definition1}
    Let $\sigma>0$, the Gevrey spaces $\mathcal G^{\sigma}(\mathbb R^{n})$ is the space of $C^{\infty}$ functions $u$ satisfying for any $k\in\mathbb N$, there exists a constant $C>0$ such that
$$\|(1-\Delta)^{\frac k{2\sigma}}u\|_{L^{2}(\mathbb R^{n})}\le C^{k+1}k!.$$
%equivalently, there exists a constant $c_{0}>0$ such that
%$$e^{c_{0}|\xi|^{\frac1\sigma}}\hat u\in L^{2}(\mathbb R^{n}).$$
Let $\mu, \nu>0$ and $\mu+\nu\ge 1$, the Gelfand-Shilov space $S^{\mu}_{\nu}(\mathbb{R}^n)$ is the space of $C^{\infty}$ functions $u$ satisfying for any $k\in\mathbb N$ there exist the constants $C>0, \tilde C>0$ such that
$$\|(1-\Delta)^{\frac{k}{2\mu}}u\|_{L^{2}(\mathbb R^{n})}\le C^{k+1}k!    \qquad {\rm and}\qquad \left\|(1+|\cdot|^{2})^{\frac{k}{2\nu}}u\right\|_{L^{2}(\mathbb R^{n})}\le\tilde C^{k+1}k!.$$
\end{definition}

In this work, we consider the Cauchy problem \eqref{1-1} with $0<s<1$, $\gamma+2s>0$. Set $\tilde s=\min\{1/2, s\}$, we would show the Cauchy problem of \eqref{1-1} admits a solution that enjoys the Gevrey regularity and decay estimation. The main result reads as follows.

\begin{thm}\label{thm1}
    Let $0<s<1$, $\gamma+2s>0$ and $T>0$. Assume that $u_0\in L^{2}(\mathbb R^6_{x, v})$, $f\in C^{\infty}([0, T]; \mathcal G^{\frac{1}{2\tilde s}}(\mathbb R_{x}^{3}; S^{\frac{1}{2\tilde s}}_{\frac{1}{\gamma/2+s}}(\mathbb R^{3}_{v})))$,  then the Cauchy problem \eqref{1-1} admits a solution in $C^{\infty}(]0, T]; \mathcal G^{\frac{1}{2\tilde s}}(\mathbb R_{x}^{3}; S^{\frac{1}{2\tilde s}}_{\frac{1}{\gamma/2+s}}(\mathbb R^{3}_{v})))$. In fact, we prove that there exists $C>0$ such that for all $0<t\le T$
        $$
        \left\|(1-\Delta_{v}-t^{2}\Delta_{x})^{\tilde sk}u\right\|_{L^{2}(\mathbb R^{6}_{x, v})}+\left\|\langle v\rangle^{(\gamma/2+s)k}u\right\|_{L^{2}(\mathbb R^{6}_{x, v})}\le\frac{ C^{k+1}}{t^{k}}k!, \quad\forall k\in\mathbb N.
             $$
\end{thm}

In~\cite{X-1}, we show that the solution for the Cauchy problem of the spatially homogenous fractional Kramers-Fokker-Planck equation belongs to the Gelfand-Shilov space $S^{\frac{1}{2\tilde s}}_{\frac{1}{\gamma/2+s}}(\mathbb R^{3})$ for all $t>0$. However, the equation of \eqref{1-1} is a degenerate parabolic type equation, but in this work, we prove that the regularity of solution of the velocity variables transport to the spatial variables with same order $\frac{1}{2\tilde s}$ for all $t>0$, so that our results is much more strong then the classical hypoelliptic regularity results.

\section{Estimations of Commutator and Interpolation}

In the following, the notation $A\lesssim B$ means there exists a constant $C>0$ such that $A\le C B$, the notation $[T_1, T_2]=T_1T_2-T_2T_1$ is commutator. For simplicity, we denote $\langle D_{v}\rangle=\sqrt{1-\Delta_{v}}$, $\langle v\rangle=\sqrt{1+|v|^2}$, then \eqref{1-1} can be written as
\begin{equation*}
\left\{
\begin{aligned}
 &\partial_t u+v\cdot\nabla_x u+\langle v\rangle^{\gamma}\langle D_{v}\rangle^{2s}u+\langle v\rangle^{\gamma+2s}u=f,\\
 &u|_{t=0}=u_0.
\end{aligned}
\right.
\end{equation*}
For later use, we need the following estimation of commutator for the operator $\langle D_{v}\rangle^{r}$ with $r\in\mathbb R$.
\begin{lemma}\label{lemma2.2}
     Let $m, r\in\mathbb R$, then there exists a constant $C_{1}>0$ such that for all $u\in L^{2}(\mathbb R^{3}_{x}; \mathcal S(\mathbb R^{3}_{v}))$,
     \begin{equation}\label{commutator1}
          \left\|\left[\langle D_{v}\rangle^{r}, \langle v\rangle^{m}\right]u\right\|_{L^{2}(\mathbb R_{x, v}^{6})}\le C_{1}\left\|\langle D_{v}\rangle^{r-1}\left(\langle v\rangle^{m}u\right)\right\|_{L^{2}(\mathbb R^{6}_{x, v})},
      \end{equation}
      with the constant $C_{1}$ depends on $m$ and $r$,
\end{lemma}
\begin{proof}
     Let $\langle v\rangle^{m}u(x, v)=g(x, v)$, then
     $$\left[\langle D_{v}\rangle^{r}, \langle v\rangle^{m}\right]u=\langle D_{v}\rangle^{r}g-\langle v\rangle^{m}\langle D_{v}\rangle^{r}\left(\langle v\rangle^{-m}g\right).$$
     Since $\langle D_{v}\rangle^{r}$ can be viewed as a pseudo-differential operator of symbol $(1+|\xi|^{2})^{\frac{r}{2}}$, it follows that for any $N\in\mathbb N_{+}$,
     \begin{equation*}
         \langle v\rangle^{m}\langle D_{v}\rangle^{r}\langle v\rangle^{-m}=\sum_{0\le|\alpha|<N}\frac{1}{\alpha!}\langle v\rangle^{m}D_{v}^{\alpha}\langle v\rangle^{-m}a^{(\alpha)}(D_{v})+\langle v\rangle^{m}r_{N}(v,D_{v}),
     \end{equation*}
     where $a^{(\alpha)}(\xi)=\partial^{\alpha}_{\xi}\langle\xi\rangle^{r}$, and
     $$r_{N}(v,\xi)=N\sum_{|\alpha|=N}\int_{0}^{1}\frac{(1-\theta)^{N-1}}{\alpha!}r_{\theta, \alpha}(v, \xi)d\theta,$$
     with $r_{\theta, \alpha}(v, \xi)$ is the oscillating integral, defined via
     $$r_{\theta, \alpha}(v, \xi)=Os-\iint e^{-iv'\xi'}D^{\alpha}\langle v+v'\rangle^{-m}a^{(\alpha)}(\xi+\theta\xi')\frac{dv'd\xi'}{(2\pi)^{3}}.$$
     Using the identity
     \begin{equation}\label{identity}
         e^{-iv'\xi'}=\langle v'\rangle^{-2l '}\langle D_{\xi'}\rangle^{2l'}e^{-iv'\xi'}=\langle\xi'\rangle^{-2l}\langle D_{v'}\rangle^{2l}e^{-iv'\xi'},
     \end{equation}
     and the fact $2l'>|m|+N+|\beta|+3$, $2l>r+N+|\beta'|+3$, then from integration by parts and using the Leibniz formula, we can deduce that
     \begin{equation*}
        \begin{split}
             &D^{\beta}_{v}\partial^{\beta'}_{\xi}(\langle v\rangle^{m}r_{\theta,\alpha}(v,\xi))%&=Os-\iint e^{-iv'\eta}m^{(\alpha+\beta')}(\xi+\theta\eta)D_{v}^{\alpha+\beta}\langle v+v'\rangle^{\gamma/2}\frac{dv'd\eta}{(2\pi)^{3}}\\
             %&=\int\bigg(\int e^{-iv'\eta}\langle v'\rangle^{-2l'}(1-\Delta_{\eta})^{l'}\big(\langle\eta\rangle^{-2l}\\
             %&\quad\times(1-\Delta_{v'})^{l}m^{(\alpha+\beta')}(\xi+\theta\eta)D_{v}^{\alpha+\beta}\langle v+v'\rangle^{\gamma/2}\big)\frac{d\eta}{(2\pi)^{3}}\bigg)d v '\\
             %&=\sum_{\beta_{1}+\beta_{2}=\beta}C_{\beta}^{\beta_{1}}D^{\beta_{1}}_{v}\langle v\rangle^{-m}\int(1-\Delta_{v'})^{l}D^{\alpha+\beta_{2}}\langle v+v'\rangle^{\gamma}\bigg(\int e^{-iv'\eta}\\
             %&\quad\times(1-\Delta_{\eta})^{l'}\left(\langle\eta\rangle^{-2l}b^{(\alpha+\beta')}(\xi+\theta\eta)\bigg)\frac{d\eta}{(2\pi)^{3}}\right)\frac{dv'}{\langle v'\rangle^{2l'}}\\
             %&=\int(1-\Delta_{v'})^{l}D_{v}^{\alpha+\beta}\langle v+v'\rangle^{\gamma}\bigg(\int_{|\eta|\le\frac{\langle\xi\rangle}{2}}+\int_{|\eta|\ge\frac{\langle\xi\rangle}{2}}\\
             %&\quad e^{-iv'\eta}(1-\Delta_{\eta})^{l'}\left(\langle\eta\rangle^{-2l}a^{(\alpha+\beta')}(\xi+\theta\eta)\right)\frac{d\eta}{(2\pi)^{3}}\bigg)\frac{dv'}{\langle v'\rangle^{2l'}}\\
             =\sum_{\beta_{1}+\beta_{2}=\beta}C_{\beta}^{\beta_{1}}D^{\beta_{1}}_{v}\langle v\rangle^{m}\int_{\mathbb R^{3}}\langle D_{\xi'}\rangle^{2l'}a^{(\alpha+\beta')}(\xi+\theta\xi')\frac{G(v, \xi')d\xi'}{\langle \xi'\rangle^{2l}(2\pi)^{3}},
        \end{split}
    \end{equation*}
    where
    $$G(v, \xi')=\int_{\mathbb R_{3}}e^{-iv'\xi'}\langle D_{v'}\rangle^{2l}\left(\langle v'\rangle^{-2l'}D^{\alpha+\beta_{2}}\langle v+v'\rangle^{-m}\right)d v '.$$
    Now, we consider the term $G(v, \xi')$ first. By using the Leibniz formula and the fact $2l'>|m|+N+|\beta|+3$, one gets
    \begin{equation*}
        \begin{split}
            |G(v, \xi')|%&\le\int_{\mathbb R_{3}}\left|\langle D_{v'}\rangle^{2l}\left(\langle v'\rangle^{-2l'}D^{\alpha+\beta_{2}}\langle v+v'\rangle^{-m}\right)\right|dv'\\
            &\le C(l)\sum_{|\sigma|\le 2l}\int_{\mathbb R_{3}}\left|\partial^{\sigma}_{v'}\left(\langle v'\rangle^{-2l'}D^{\alpha+\beta_{2}}\langle v+v'\rangle^{-m}\right)\right|dv'\\
            &\le C(l, l', \alpha, \beta_{2})\int\langle v'\rangle^{-2l'}\langle v+v'\rangle^{-m-N-|\beta_{2}|}dv'\\
            %&\lesssim\int_{\mathbb R_{3}}\langle\eta\rangle^{-2l}\langle\xi+\theta\eta\rangle^{m-N-|\beta'|}dv'\\
            &\le C(l, l', \alpha, \beta_{2}, m)\langle v\rangle^{-m-N-|\beta_{2}|}\int_{\mathbb R_{3}}\langle v'\rangle^{-2l'+|m|+N+|\beta_{2}|}dv'\\
            &\le\tilde C(l, l', \alpha, \beta_{2}, m)\langle v\rangle^{-m-N-|\beta_{2}|}.
        \end{split}
    \end{equation*}
Substituting it into $D^{\beta}_{v}\partial^{\beta'}_{\xi}\left(\langle v\rangle^{m}r_{\theta, \alpha}(v, \xi)\right)$, since $\theta\in]0,1[$, by using Peetre's inequality, for all $v\in\mathbb R^{3}$
    \begin{equation*}%\label{r_theta alpha}
        \begin{split}
             &\left|D^{\beta}_{v}\partial^{\beta'}_{\xi}\left(\langle v\rangle^{m}r_{\theta,\alpha}(v,\xi)\right)\right|\\
             &\le C(l, l', \alpha, \beta, \beta', m)\langle v\rangle^{-N-|\beta|}\int_{\mathbb R^{3}}\langle \xi'\rangle^{-2l}\langle\xi+\theta\xi'\rangle^{r-N-|\beta'|}d\xi'\\
             &\le C(l, l', \alpha, \beta, \beta', m)\langle v\rangle^{-N-|\beta|}\langle\xi\rangle^{r-N-|\beta'|}2^{r+N+|\beta'|}\int_{\mathbb R^{3}}\langle \xi'\rangle^{-2l+|r|+N+|\beta'|}d\xi'\\
             &\le\tilde C(l, l', \alpha, \beta, \beta', m)\langle v\rangle^{-N-|\beta|}\langle\xi\rangle^{r-N-|\beta'|}\le\tilde C(l, l', \alpha, \beta, \beta', m)\langle\xi\rangle^{r-N-|\beta'|},
        \end{split}
    \end{equation*}
    here we use the fact $2l>r+N+|\beta'|+3$, hence
    $$\langle v\rangle^{m}r_{N}(v, D_{v})\in\Psi_{1,0}^{r-N},$$
    this implies $\langle v\rangle^{m}r_{N}(v, D_{v})\langle D_{v}\rangle^{N-r}\in\Psi_{1,0}^{0}$. Since $u(x, \cdot)\in\mathcal S(\mathbb R^{3}_{v})$, it follows that $g(x, \cdot)=\langle\cdot\rangle^{m}u(x, \cdot)\in L^{2}(\mathbb R^{3}_{v})$, then we have
    $$\left\|\langle \cdot\rangle^{m}r_{N}(\cdot, D_{v})\langle D_{v}\rangle^{N-r}\langle D_{v}\rangle^{r-N}g(x, \cdot)\right\|_{L^{2}(\mathbb R^{3}_{v})}\le C_{r, m}\left\|\langle D_{v}\rangle^{r-N}g(x, \cdot)\right\|_{L^{2}(\mathbb R^{3}_{v})}.$$
    Therefore, taking $N=1$, since $u\in L^{2}(\mathbb R^{3}_{x}; \mathcal S(\mathbb R^{3}_{v}))$, we can obtain that
    \begin{equation*}
        \begin{split}
            &\left\|\left[\langle D_{v}\rangle^{r}, \langle v\rangle^{m}\right]u\right\|_{L^{2}(\mathbb R^{6}_{x, v})}%=\left(\int_{\mathbb R^{3}_{x}}\left\|[\langle D_{v}\rangle^{r}, \langle v\rangle^{m}]u\right\|^{2}_{L^{2}(\mathbb R^{3}_{v})}dx\right)^{\frac12}\\
            %=\left\|\langle v\rangle^{m}r_{1}(v, D_{v})g\right\|_{L^{2}(\mathbb R^{6}_{x, v})}\\
            =\left\|\langle v\rangle^{m}r_{1}(v, D_{v})\langle D_{v}\rangle^{1-r}\langle D_{v}\rangle^{r-1}g\right\|_{L^{2}(\mathbb R^{6}_{x, v})}\\
            &=\left(\int_{\mathbb R^{3}_{x}}\left\|\langle \cdot\rangle^{m}r_{1}(\cdot, D_{v})\langle D_{v}\rangle^{1-r}\langle D_{v}\rangle^{r-1}g(x, \cdot)\right\|^{2}_{L^{2}(\mathbb R^{3}_{v})}dx\right)^{\frac12}\\
            %&\lesssim\sum_{j=1}^{N-1}\left\|\langle D_{v}\rangle^{r-j}g\right\|_{L^{2}(\mathbb R^{6}_{x, v})}+\left\|\langle D_{v}\rangle^{r-N}g\right\|_{L^{2}(\mathbb R^{6}_{x, v})}\\
            &\le C_{1}\left\|\langle D_{v}\rangle^{r-1}\left(\langle v\rangle^{m}u\right)\right\|_{L^{2}(\mathbb R^{6}_{x, v})},
        \end{split}
    \end{equation*}
    with $C_{1}$ depends on $m$ and $r$.
\end{proof}
\begin{remark}
Taking $r=2s$ in \eqref{commutator1}, if $0<s\le1/2$, then $2s-1\le0$, so that
\begin{equation}\label{0-s-1/2}
        \begin{split}
            \left\|\left[\langle D_{v}\rangle^{2s}, \langle v\rangle^{m}\right]u\right\|_{L^{2}(\mathbb R^{6}_{x, v})}&\le C_{1}\left\|\langle D_{v}\rangle^{2s-1}\left(\langle v\rangle^{m}u\right)\right\|_{L^{2}(\mathbb R^{6}_{x, v})}\\
            &\le C_{1}\left\|\langle v\rangle^{m}u\right\|_{L^{2}(\mathbb R^{6}_{x, v})},
        \end{split}
    \end{equation}
    if $1/2<s<1$, then from  \eqref{commutator1} we have
    \begin{equation}\label{1/2-s-1}
        \begin{split}
            \left\|\left[\langle D_{v}\rangle^{2s}, \langle v\rangle^{m}\right]u\right\|_{L^{2}(\mathbb R^{6}_{x, v})}\le C_{1}\left\|\langle D_{v}\rangle^{2s-1}(\langle v\rangle^{m}u)\right\|_{L^{2}(\mathbb R^{6}_{x ,v})}.
        \end{split}
    \end{equation}
\end{remark}
%Before giving the result of the commutator for the operator $M_{2s}^{k}(t, D_{x}, D_{v})$, we need the following result.

Now, recall the result from~\cite{M-4}, we have
\begin{lemma}(~\cite{M-4})\label{lemma  2.3}
For any $\alpha>0$, there exists $c_{\alpha}>0, C_{\alpha}>0$ such that $\forall t>0$
$$c_{\alpha}t\left(1+|\xi|^{2}+t^{2}|\eta|^{2}\right)^{\frac\alpha2}\le\int_{0}^{t}\left(1+|\xi+\rho\eta|^{2}\right)^{\frac\alpha2}d\rho\le C_{\alpha}t\left(1+|\xi|^{2}+t^{2}|\eta|^{2}\right)^{\frac\alpha2}.$$
\end{lemma}
For $0<s<1$, set $\tilde s=\min\{1/2, s\}$ and
\begin{equation*}%\label{def-M}
     M_{2\tilde s}(t, D_{x}, D_{v})=\int_{0}^{t}\left\langle D_{v}+\rho D_{x}\right\rangle^{2\tilde s}d\rho\sim t\left(1-\Delta_{v}-t^{2}\Delta_{x}\right)^{\tilde s},
\end{equation*}
we consider the commutator for the operator $M_{2\tilde s}(t, D_{x}, D_{v})$.
%\begin{equation}\label{M}
    %M_{2\tilde s}(t, D_{x}, D_{v})\sim t\left(1-\Delta_{v}-t^{2}\Delta_{x}\right)^{\tilde s}.
%\end{equation}

\begin{lemma}\label{commutator M}
     Let $0<s<1$ and $m\in\mathbb R$, then there exists a constant $C_{2}>0$, such that for any $k\in\mathbb N_{+}$
     \begin{equation}\label{2-3-1}
     \begin{split}
          &\left\|\langle v\rangle^{-\frac m2}\left[M^{k}_{2\tilde s}, \langle v\rangle^{m}\right]u\right\|_{L^{2}(\mathbb R^{6}_{x, v})}\\
          &\le\sum_{p=0}^{k-1}\left(C_{2}(t+1)\right)^{k-p}C_{k}^{p}\left\|\langle v\rangle^{\frac m2}M^{p}_{2\tilde s}u\right\|_{L^{2}(\mathbb R^{6}_{x, v})}, \quad\forall u\in H^{k}(\mathbb R^{3}_{x}; \mathcal S(\mathbb R^{3}_{v})).
     \end{split}
     \end{equation}
     %and
     %\begin{equation}\label{2-3-2}
          %\left\|\left[M^{k}_{2\tilde s}, \langle v\rangle^{\gamma}\right]u\right\|_{L^{2}(\mathbb R^{6}_{x, v})}\le C_{3}t\left\|\langle v\rangle^{\gamma}M^{k-1}_{2\tilde s}u\right\|_{L^{2}(\mathbb R^{6}_{x, v})}.
     %\end{equation}
     where the constant $C_{2}$ depends on $m, s$ and $C_{k}^{p}=\frac{k!}{p!(k-p)!}$.
\end{lemma}
\begin{proof}
    We show \eqref{2-3-1} holds by induction on the index $k$. For $k=1$, set $g(x, v)=\langle v\rangle^{\frac m2}u(x, v)$ it follows that
    \begin{equation*}
        \begin{split}
           \langle v\rangle^{-\frac m2}\left[M_{2\tilde s}, \langle v\rangle^{m}\right]u%&=\langle v\rangle^{-m/2}M_{2\tilde s}\left(\langle v\rangle^{m/2}g\right)-\langle v\rangle^{m/2}M_{2\tilde s}\left(\langle v\rangle^{-m/2}g\right)\\
           &=\langle v\rangle^{-\frac m2}\left[M_{2\tilde s}, \langle v\rangle^{\frac m2}\right]g-\langle v\rangle^{\frac m2}\left[M_{2\tilde s}, \langle v\rangle^{-\frac m2}\right]g\\
           &=\langle v\rangle^{-\frac m2}\sum_{j=1}^{3}r_{t, j}(v, D_{x}, D_{v})g-\langle v\rangle^{\frac m2}\sum_{j=1}^{3}\tilde r_{t, j}(v, D_{x}, D_{v})g,
          \end{split}
    \end{equation*}
    here
    \begin{equation*}
        \begin{split}
            r_{t, j}(v, \eta, \xi)=\int_{0}^{1}r_{\theta, t, j}(v, \eta, \xi)d\theta \quad {\rm and}\quad \tilde r_{t, j}(v, \eta, \xi)=\int_{0}^{1}\tilde r_{\theta, t, j}(v, \eta, \xi)d\theta,
            %&=\sum_{|\alpha|=1}\int_{0}^{1}\left(\iint e^{-iv'\xi'}D^{\alpha}\langle v+v'\rangle^{-m}\partial^{\alpha}_{\xi}M^{k}_{2\tilde s}(t, \eta, \xi+\theta\xi')\frac{dv'd\xi'}{(2\pi)^{3}}\right)d\theta\\
            %&=N\sum_{|\alpha|=N}\int_{0}^{1}\frac{(1-\theta)^{N-1}}{\alpha!}r_{\theta, \alpha}(v, \eta, \xi)d\theta,
        \end{split}
    \end{equation*}
    with the oscillating integrals $r_{\theta, t, j}(v, \eta, \xi)$ and $\tilde r_{\theta, t, j}(v, \eta, \xi)$, defined via
    $$r_{\theta, t, j}(v, \eta, \xi)=Os-\iint e^{-iv'\xi'}D_{j}\langle v+v'\rangle^{\frac m2}\partial_{\xi_{j}}M_{2\tilde s}(t, \eta, \xi+\theta\xi')\frac{dv'd\xi'}{(2\pi)^{3}}.$$
    $$\tilde r_{\theta, t, j}(v, \eta, \xi)=Os-\iint e^{-iv'\xi'}D_{j}\langle v+v'\rangle^{-\frac m2}\partial_{\xi_{j}}M_{2\tilde s}(t, \eta, \xi+\theta\xi')\frac{dv'd\xi'}{(2\pi)^{3}}.$$
   %Now, we consider $r_{\theta, \alpha}(v, \eta, \xi)$, and want to show that there exist $C(m, s)>0$ depends on $m$ and $s$ such that
   %$$\left|r_{\theta, \alpha}(v, \eta, \xi)\right|\le2^{k-1}C(m, s)kt\langle v\rangle^{m}M^{k-1}_{2\tilde s}(t, \eta, \xi).$$
   Set
   $$R_{\theta, t, j}(v, \eta, \xi)=\langle v\rangle^{-\frac m2}r_{\theta, t, j}(v, \eta, \xi) \quad {\rm and}\quad\tilde R_{\theta, t, j}(v, \eta, \xi)=\langle v\rangle^{\frac m2}\tilde r_{\theta, t, j}(v, \eta, \xi),$$
   $$R_{t, j}(v, \eta, \xi)=\int_{0}^{1}R_{\theta, t, j}(v, \eta, \xi)d\theta \quad {\rm and} \quad \tilde R_{t, j}(v, \eta, \xi)=\int_{0}^{1}\tilde R_{\theta, t, j}(v, \eta, \xi)d\theta,$$
   then $\langle v\rangle^{-\frac m2}\left[M_{2\tilde s}, \langle v\rangle^{m}\right]u$ can be rewritten as
    \begin{equation*}
        \begin{split}
           \langle v\rangle^{-\frac m2}\left[M_{2\tilde s}, \langle v\rangle^{m}\right]u%&=\langle v\rangle^{-m/2}M_{2\tilde s}\left(\langle v\rangle^{m/2}g\right)-\langle v\rangle^{m/2}M_{2\tilde s}\left(\langle v\rangle^{-m/2}g\right)\\
           %&=\langle v\rangle^{-m/2}\left[M_{2\tilde s}, \langle v\rangle^{m/2}\right]g-\langle v\rangle^{m/2}\left[M_{2\tilde s}, \langle v\rangle^{-m/2}\right]g\\
           &=\sum_{j=1}^{3}R_{t, j}(v, D_{x}, D_{v})g-\sum_{j=1}^{3}\tilde R_{t, j}(v, D_{x}, D_{v})g,
          \end{split}
    \end{equation*}
    we would show that $R_{t, j}(v, \eta, \xi)$ and $\tilde R_{t, j}(v, \eta, \xi)$ are the pseudo-differential operators of order $0$.

    For any $\alpha, \beta\in\mathbb N^{3}$, by using the identity \eqref{identity} and the fact $2l>4+|\beta|-2\tilde s$, $2l'>\frac{|m|}{2}+|\alpha|+4$, then from integration by part and the Leibniz formula,
   %\begin{equation*}\label{r-theta-j}
        %\begin{split}
            %r_{\theta, j}(v, \eta, \xi)=\int\langle D_{\xi'}\rangle^{2l'}\partial_{\xi_{j}}M_{2\tilde s}(t, \eta, \xi+\theta\xi')T(v, \xi')\frac{d\xi'}{\langle\xi'\rangle^{2l}(2\pi)^{3}},
        %\end{split}
    %\end{equation*}
    \begin{equation*}
        \begin{split}
            &\partial^{\alpha}_{v}\partial^{\beta}_{\xi}R_{\theta, t, j}(v, \eta, \xi)=\partial^{\alpha}_{v}\partial^{\beta}_{\xi}\left(\langle v\rangle^{-\frac m2}r_{\theta, t, j}(v, \eta, \xi)\right)\\
            &=\sum_{\alpha_{1}+\alpha_{2}=\alpha}C^{\alpha_{1}}_{\alpha}\int\langle D_{\xi'}\rangle^{2l'}\partial^{\beta}_{\xi}\partial_{\xi_{j}}M_{2\tilde s}(t, \eta, \xi+\theta\xi')\partial^{\alpha_{1}}_{v}\langle v\rangle^{-\frac m2}\partial^{\alpha_{2}}_{v}T(v, \xi')\frac{d\xi'}{\langle\xi'\rangle^{2l}(2\pi)^{3}},
        \end{split}
    \end{equation*}
    where
    $$T(v, \xi')=\int e^{-iv'\xi'}\langle D_{v'}\rangle^{2l}\left(\langle v'\rangle^{-2l'}D_{j}\langle v+v'\rangle^{\frac m2}\right)dv'.$$
    %$$\tilde T(v, \xi')=\int e^{-iv'\xi'}\langle D_{v'}\rangle^{2l}\left(\langle v'\rangle^{-2l'}D_{j}\langle v+v'\rangle^{-\frac m2}\right)dv'.$$
Now, we consider $\partial_{v}^{\alpha_{2}}T(v, \xi')$. For any $\alpha\in\mathbb N^{3}$, noting that
    $$\partial_{v}^{\alpha}\langle v+v'\rangle^{\frac m2}\le C(m, \alpha)\langle v+v'\rangle^{\frac m2-|\alpha|},$$
    then by using the Leibniz formula and Peetre's inequality, we have
    \begin{equation*}
        \begin{split}
            \left|\partial_{v}^{\alpha_{2}}T(v, \xi')\right|%&\le C(l)\sum_{|\sigma|\le 2l}\int\left|\partial^{\sigma}_{v'}\left(\langle v'\rangle^{-2l'}\partial_{v}^{\alpha_{2}}D_{j}\langle v+v'\rangle^{-m}\right)\right|dv'\\
            &\le C(l)\sum_{|\sigma|\le 2l}\sum_{\sigma_{1}+\sigma_{2}=\sigma}C_{\sigma}^{\sigma_{1}}\int\left|\partial^{\sigma_{1}}_{v'}\langle v'\rangle^{-2l'}\partial^{\sigma_{2}}_{v'}\partial_{v}^{\alpha_{2}}D_{j}\langle v+v'\rangle^{\frac m2}\right|dv'\\
            &\le C(l)\sum_{|\sigma|\le 2l}C(\sigma, \alpha_{2}, m, l')\int\langle v'\rangle^{-2l'}\langle v+v'\rangle^{\frac m2-1-|\alpha_{2}|}dv'\\
            &\le 2^{1+|\alpha_{2}|+\frac{|m|}{2}}C(l, l', m, \alpha_{2})\langle v\rangle^{\frac m2-1-|\alpha_{2}|}\int\langle v'\rangle^{-2l'+1+|\alpha_{2}|+\frac{|m|}{2}}dv'.
        \end{split}
    \end{equation*}
    %here $C(l)$ is the constant depending on $l$, and $\left[\cdot\right]$ is the the integer part.
    %For $m/2>1+|\alpha_{2}|$, using the Leibniz formula and Peetre's inequality, one gets
    %\begin{equation*}
        %\begin{split}
            %\left|\partial_{v}^{\alpha_{2}}T(v, \xi')\right|%&\le C(l)\sum_{|\sigma|\le 2l}\int\left|\partial^{\sigma}_{v'}\left(\langle v'\rangle^{-2l'}D_{j}\langle v+v'\rangle^{\gamma}\right)\right|dv'\\
            %&\le C(l)\sum_{|\sigma|\le 2l}\sum_{\sigma_{1}+\sigma_{2}=\sigma}C_{\sigma}^{\sigma_{1}}\int\left|\partial^{\sigma_{1}}_{v'}\langle v'\rangle^{-2l'}\partial^{\sigma_{2}}_{v'}D_{j}\langle v+v'\rangle^{\gamma}\right|dv'\\
            %&\le C(l)\sum_{|\sigma|\le 2l}\sum_{\sigma_{1}+\sigma_{2}=\sigma}C_{\sigma}^{\sigma_{1}}\left(2l'+|\sigma_{1}|\right)!\left(\left[m/2-1-|\alpha_{2}|\right]+1\right)!\\
            %&\quad\times\left(\left[|\sigma_{2}|-m/2+1+|\alpha_{2}|\right]+1\right)!\int\langle v'\rangle^{-2l'}\langle v+v'\rangle^{m/2-1-|\alpha_{2}|}dv'\\
            %&\le 2^{m/2-1-|\alpha_{2}|}C(l)\sum_{|\sigma|\le 2l}2^{|\sigma|}\left(2l'+|\sigma|\right)!\left(|\sigma|+1\right)!\\
            %&\quad\times\langle v\rangle^{m/2-1-|\alpha_{2}|}\int\langle v'\rangle^{-2l'+m/2-1-|\alpha_{2}|}dv'.
        %\end{split}
    %\end{equation*}
    %Hence, for any $n\in\mathbb R$, there exists a constant $C(l, l')$ depends on $l$ and $l'$
   %$\begin{equation*}%\label{T}
        %\begin{split}
            %\left|\partial_{v}^{\alpha_{2}}T(v, \xi')\right|&\le2^{|1-m/2|+|\alpha_{2}|}C(l, l', \alpha_{2})\langle v\rangle^{m/2-1-|\alpha_{2}|}\int\langle v'\rangle^{-2l'+|1-m/2|+|\alpha_{2}|}dv'.
        %\end{split}
    %\end{equation*}
    Since $2l'>\frac{|m|}{2}+|\alpha|+4$, %we have
    %$$\int\langle v'\rangle^{-2l'+1+\frac{|m|}{2}+|\alpha_{2}|}dv'\sim\left(l'-\frac{1+|m/2|+|\alpha_{2}|}{2}\right)^{-3/2},$$
    it follows that
    $$\left|\partial_{v}^{\alpha_{2}}T(v, \xi')\right|\le C(l, m, \alpha_{2}, \beta)\langle v\rangle^{\frac m2-1-|\alpha_{2}|}.$$
    Plugging it back into $\partial^{\alpha}_{v}\partial^{\beta}_{\xi}R_{\theta, t, j}(v, \eta, \xi)$, it follows that
    \begin{equation*}
        \begin{split}
             \left|\partial^{\alpha}_{v}\partial^{\beta}_{\xi}R_{\theta, t, j}(v, \eta, \xi)\right|&\le C(l, m, \alpha, \beta)\langle v\rangle^{-|\alpha|-1}\\
             &\quad\times\sum_{|\sigma|\le2l'}\int\langle\xi'\rangle^{-2l}\left|\partial^{\sigma}_{\xi'}\partial^{\beta}_{\xi}\partial_{\xi_{j}}M_{2\tilde s}(t, \eta, \xi+\theta\xi')\right|d\xi'.
            %&\le C(m, l)C(m)K^{l}\langle v\rangle^{m}\int\langle\xi'\rangle^{-2l}\left|\partial^{\alpha}_{\xi}M^{k}_{2\tilde s}(t, \eta, \xi+\theta\xi')\right|d\xi'\\
            %&\quad+C(m, l)C(m)K^{l}\langle v\rangle^{m}\sum_{2\le|\sigma'|\le2l'+1}\int\langle\xi'\rangle^{-2l}\left|\partial^{\sigma'}_{\xi}M^{k}_{2\tilde s}(t, \eta, \xi+\theta\xi')\right|d\xi'\\
            %&=S_{1}+S_{2}.
        \end{split}
    \end{equation*}
    From Lemma \ref{lemma  2.3}, it follows that $M_{2\tilde s}(t, D_{x}, D_{v})\sim t\left(1-\Delta_{v}-t^{2}\Delta_{x}\right)^{\tilde s}$,   then we have for any $\sigma\in\mathbb N^{3}$,
    $$\left|\partial^{\sigma}_{\xi}M_{2\tilde s}(t, \eta, \xi)\right|\le C(s, \sigma)t(1+|\xi|^{2}+t^{2}|\eta|^{2})^{\frac{2\tilde s-|\sigma|}{2}},$$
    %Since $M_{2\tilde s}(t, D_{x}, D_{v})$ is a pseud-differential operator of $2\tilde s$, we have
    %$$\left|\partial^{\sigma}_{\xi}M_{2\tilde s}(t, \eta, \xi)\right|\le C_{\sigma, s}M^{\frac{2\tilde s-|\sigma|}{2\tilde s}}_{2\tilde s}(t, \eta, \xi), \quad \forall\sigma\in\mathbb N^{3},$$
    so that using the fact $\theta\in]0,1[$, one has for any $\alpha, \beta\in\mathbb N^{3}$
    \begin{equation*}
    \begin{split}
        \left|\partial^{\alpha}_{v}\partial^{\beta}_{\xi}R_{\theta, t, j}(v, \eta, \xi)\right|&\le C(l', l, m, \alpha, \beta)t\int\langle\xi'\rangle^{-2l}(1+|\xi+\theta\xi'|^{2}+t^{2}|\eta|^{2})^{\frac{2\tilde s-1-|\beta|}{2}}d\xi'\\
        &\le C(s, m, \alpha, \beta)t\int\langle\xi'\rangle^{-2l+1+|\beta|-2\tilde s}d\xi'(1+|\xi|^{2}+t^{2}|\eta|^{2})^{\frac{2\tilde s-1-|\beta|}{2}}.
        %&\le\tilde C_{m, s, \alpha, \beta}t(1+|\xi|^{2}+t^{2}|\eta|^{2})^{\frac{2\tilde s-1-|\beta|}{2}}.
    \end{split}
    \end{equation*}
    Since $2l>4+|\beta|-2\tilde s$, one has for any $\alpha, \beta\in\mathbb R^{3}$
    \begin{equation*}%\label{r-j}
    \begin{split}
        &\left|\partial^{\alpha}_{v}\partial^{\beta}_{\xi}R_{t, j}(v, \eta, \xi)\right|
        %&\le c(m, \alpha, \beta)\sum_{|\sigma|\le2l'}\left(|\sigma|+|\beta|-1\right)!\int\langle\xi'\rangle^{-2l+1+|\beta|-2\tilde s}d\xi'M_{2\tilde s-1-|\beta|}(t, \eta, \xi)\\
        \le C(s, m, \alpha, \beta)t(1+|\xi|^{2}+t^{2}|\eta|^{2})^{\frac{-|\beta|}{2}}.
    \end{split}
    \end{equation*}
    Similarly, we can prove that
    \begin{equation*}%\label{r-j}
    \begin{split}
        &\left|\partial^{\alpha}_{v}\partial^{\beta}_{\xi}\tilde R_{t, j}(v, \eta, \xi)\right|
        %&\le c(m, \alpha, \beta)\sum_{|\sigma|\le2l'}\left(|\sigma|+|\beta|-1\right)!\int\langle\xi'\rangle^{-2l+1+|\beta|-2\tilde s}d\xi'M_{2\tilde s-1-|\beta|}(t, \eta, \xi)\\
        \le C(s, m, \alpha, \beta)t(1+|\xi|^{2}+t^{2}|\eta|^{2})^{\frac{-|\beta|}{2}}.
    \end{split}
    \end{equation*}
    Hence $R_{t, j}(v, D_{x}, D_{v})$ and $\tilde R_{t, j}(v, D_{x}, D_{v})$ are the pseudo-differential operators of order $0$. So that for all $g\in L^{2}(\mathbb R^{6}_{x, v})$,  we have
    $$\left\|R_{t, j}(v, D_{x}, D_{v})g\right\|_{L^{2}(\mathbb R^{6}_{x, v})}\lesssim t\left\|g\right\|_{L^{2}(\mathbb R^{6}_{x, v})},$$
    and
    $$\left\|\tilde R_{t, j}(v, D_{x}, D_{v})g\right\|_{L^{2}(\mathbb R^{6}_{x, v})}\lesssim t\left\|g\right\|_{L^{2}(\mathbb R^{6}_{x, v})}.$$
    Since $u\in H^{1}(\mathbb R^{3}_{x}; \mathcal S(\mathbb R^{3}_{v}))$, we have $\langle v\rangle^{\frac m2}u\in L^{2}(\mathbb R^{6}_{x, v})$, therefore
     \begin{equation*}
        \begin{split}
            &\left\|\langle v\rangle^{-\frac m2}\left[M_{2\tilde s}, \langle v\rangle^{m}\right]u\right\|_{L^{2}(\mathbb R^{6}_{x, v})}\\
            &\le\sum_{j=1}^{3}\left(\left\|R_{t, j}(v, D_{x}, D_{v})g\right\|_{L^{2}(\mathbb R^{6}_{x, v})}+\left\|\tilde R_{t, j}(v, D_{x}, D_{v})g\right\|_{L^{2}(\mathbb R^{6}_{x, v})}\right)\\
            %&\le C_{2}t\left\|\langle v\rangle^{m/2}u\right\|_{L^{2}(\mathbb R^{6}_{x, v})}\\
            &\le C_{2}(t+1)\left\|\langle v\rangle^{\frac m2}u\right\|_{L^{2}(\mathbb R^{6}_{x, v})},%=\frac12C_{2}(t+1)\left\|\langle v\rangle^{\frac m2}u\right\|_{L^{2}(\mathbb R^{6}_{x, v})},
        \end{split}
     \end{equation*}
     with $C_{2}$ depends on $m$ and $s$. Assume that $k\ge2$ and for any $1\le j\le k-1$
     \begin{equation}\label{M-j}
     \begin{split}
          &\left\|\langle v\rangle^{-\frac m2}\left[M^{j}_{2\tilde s}, \langle v\rangle^{m}\right]u\right\|_{L^{2}(\mathbb R^{6}_{x, v})}\\
          &\le\sum_{p=0}^{j-1}\left(C_{2}(t+1)\right)^{j-p}C_{j}^{p}\left\|\langle v\rangle^{\frac m2}M^{p}_{2\tilde s}u\right\|_{L^{2}(\mathbb R^{6}_{x, v})}, \quad\forall u\in H^{j}(\mathbb R^{3}_{x}; \mathcal S(\mathbb R^{3}_{v})).
     \end{split}
     \end{equation}
     We would show \eqref{M-j} is true for $j=k$. Noting that
     \begin{equation}\label{M-k-s}
        \begin{split}
            &\langle v\rangle^{-\frac{m}{2}}\left[M^{k}_{2\tilde s}, \langle v\rangle^{m}\right]u=\langle v\rangle^{-\frac{m}{2}}M^{k}_{2\tilde s}\left(\langle v\rangle^{m}u\right)-\langle v\rangle^{\frac{m}{2}}M^{k}_{2\tilde s}u\\
            %&=M^{k-1}_{2\tilde s}M_{2\tilde s}\left(\langle v\rangle^{m}u\right)-M^{k-1}_{2\tilde s}\left(\langle v\rangle^{m}M_{2\tilde s}u\right)\\
            %&\quad+M^{k-1}_{2\tilde s}\left(\langle v\rangle^{m}M_{2\tilde s}u\right)-\langle v\rangle^{m}M^{k-1}_{2\tilde s}M_{2\tilde s}u\\
            &=\langle v\rangle^{-\frac{m}{2}}\left[M^{k-1}_{2\tilde s}, \langle v\rangle^{m}\right]M_{2\tilde s}u+\langle v\rangle^{-\frac{m}{2}}M^{k-1}_{2\tilde s}\left(\langle v\rangle^{m}\langle v\rangle^{-m}\left[M_{2\tilde s}, \langle v\rangle^{m}\right]u\right)\\
            &=\langle v\rangle^{-\frac{m}{2}}\left[M^{k-1}_{2\tilde s}, \langle v\rangle^{m}\right]M_{2\tilde s}u+\langle v\rangle^{\frac{m}{2}}M^{k-1}_{2\tilde s}\left(\langle v\rangle^{-m}\left[M_{2\tilde s}, \langle v\rangle^{m}\right]u\right)\\
            &\quad+\langle v\rangle^{-\frac{m}{2}}\left[M^{k-1}_{2\tilde s}, \langle v\rangle^{m}\right]\left(\langle v\rangle^{-m}\left[M_{2\tilde s}, \langle v\rangle^{m}\right]u\right)\\
            &=\langle v\rangle^{-\frac{m}{2}}\left[M^{k-1}_{2\tilde s}, \langle v\rangle^{m}\right]M_{2\tilde s}u+\langle v\rangle^{\frac{m}{2}}M^{k-1}_{2\tilde s}Au+\langle v\rangle^{-\frac{m}{2}}\left[M^{k-1}_{2\tilde s}, \langle v\rangle^{m}\right]Au,
        \end{split}
     \end{equation}
     here $A(v, D_{x}, D_{v})=\langle v\rangle^{-m}\left[M_{2\tilde s}, \langle v\rangle^{m}\right]$. Then we have
     \begin{equation*}
        \begin{split}
            &\left\|\langle v\rangle^{-\frac{m}{2}}\left[M^{k}_{2\tilde s}, \langle v\rangle^{m}\right]u\right\|_{L^{2}(\mathbb R^{6}_{x, v})}\\
            &\le\left\|\langle v\rangle^{-\frac{m}{2}}\left[M^{k-1}_{2\tilde s}, \langle v\rangle^{m}\right]M_{2\tilde s}u\right\|_{L^{2}(\mathbb R^{6}_{x, v})}+\left\|\langle v\rangle^{\frac{m}{2}}M^{k-1}_{2\tilde s}Au\right\|_{L^{2}(\mathbb R^{6}_{x, v})}\\
            &\quad+\left\|\langle v\rangle^{-\frac{m}{2}}\left[M^{k-1}_{2\tilde s}, \langle v\rangle^{m}\right]Au\right\|_{L^{2}(\mathbb R^{6}_{x, v})}=J_{1}+J_{2}+J_{3}.
        \end{split}
     \end{equation*}
     Now, we consider the term $J_{1}$. Since $u\in H^{k}(\mathbb R^{3}_{x}; \mathcal S(\mathbb R^{3}_{v}))$, it follows that $M_{2\tilde s}u\in H^{k-2\tilde s}(\mathbb R^{3}_{x}; \mathcal S(\mathbb R^{3}_{v}))\subset H^{k-1}(\mathbb R^{3}_{x}; \mathcal S(\mathbb R^{3}_{v}))$, then by using \eqref{M-j},
     \begin{equation*}
     \begin{split}
        J_{1}&\le\sum_{p=0}^{k-2}\left(C_{2}(t+1)\right)^{k-1-p}C_{k-1}^{p}\left\|\langle v\rangle^{\frac{m}{2}}M^{p+1}_{2\tilde s}u\right\|_{L^{2}(\mathbb R^{6}_{x, v})}\\
        &=\sum_{p=1}^{k-1}\left(C_{2}(t+1)\right)^{k-p}\frac{(k-1)!}{(p-1)!(k-p)!}\left\|\langle v\rangle^{\frac{m}{2}}M^{p}_{2\tilde s}u\right\|_{L^{2}(\mathbb R^{6}_{x, v})}.
     \end{split}
     \end{equation*}
For the term $J_{2}$, noting that
     \begin{equation*}
        \begin{split}
            J_{2}&=\left\|\langle v\rangle^{\frac{m}{2}}M^{k-1}_{2\tilde s}AM^{-(k-1)}_{2\tilde s}\langle v\rangle^{\frac{-m}{2}}\langle v\rangle^{\frac{m}{2}}M^{k-1}_{2\tilde s}u\right\|_{L^{2}(\mathbb R^{6}_{x, v})},
            %&\le\left\|M^{k-1}_{2\tilde s}AM^{-(k-1)}_{2\tilde s}\left(\langle v\rangle^{\frac{m}{2}}M^{k-1}_{2\tilde s}u\right)\right\|_{L^{2}(\mathbb R^{6}_{x, v})}\\
            %&\quad+\left\|\langle v\rangle^{\frac{m}{2}}\left[M^{k-1}_{2\tilde s}AM^{-(k-1)}_{2\tilde s}, \langle v\rangle^{\frac{-m}{2}}\right]\left(\langle v\rangle^{\frac{m}{2}}M^{k-1}_{2\tilde s}u\right)\right\|_{L^{2}(\mathbb R^{6}_{x, v})}.
        \end{split}
     \end{equation*}
     as the argument in the operator $R_{t, j}(v, D_{x}, D_{v})$, we can obtain that $A(v, D_{x}, D_{v})$ is the pseudo-differential operator of order $2\tilde s-1$. Since $2\tilde s-1\le0$, it follows that $A(v, D_{x}, D_{v})$ is the pseudo-differential operator of order $0$. Since $M^{k-1}_{2\tilde s}(t, D_{x}, D_{v})$ is a pseudo-differential operator of order $2\tilde s(k-1)$ and $M^{-(k-1)}_{2\tilde s}(t, D_{x}, D_{v})$ is the pseudo-differential operator of order $-2\tilde s(k-1)$, it follows that $M^{k-1}_{2\tilde s}AM^{-(k-1)}_{2\tilde s}$ is the pseudo-differential operator of order $0$. %this implies
     %$$\langle v\rangle^{\frac{m}{2}}\left[M^{k-1}_{2\tilde s}AM^{-(k-1)}_{2\tilde s}, \langle v\rangle^{\frac{-m}{2}}\right]$$
     %is also the pseudo-differential operator of order $0$.
     Since $u\in H^{k}(\mathbb R^{3}_{x}; \mathcal S(\mathbb R^{3}_{v}))$, one has
$\langle v\rangle^{\frac{m}{2}}M^{k-1}_{2\tilde s}u\in H^{k-2\tilde s(k-1)}(\mathbb R^{3}_{x}; \mathcal S(\mathbb R^{3}_{x}))\subset L^{2}(\mathbb R^{6}_{x, v})$, then from Lemma 2.2 of~\cite{H-4}
     \begin{equation*}
        \begin{split}
            J_{2}%&=\left\|\langle v\rangle^{\frac{m}{2}}M^{k-1}_{2\tilde s}AM^{-(k-1)}_{2\tilde s}\langle v\rangle^{\frac{-m}{2}}\langle v\rangle^{\frac{m}{2}}M^{k-1}_{2\tilde s}u\right\|_{L^{2}(\mathbb R^{6}_{x, v})}\\
            &\le\left\|M^{k-1}_{2\tilde s}AM^{-(k-1)}_{2\tilde s}\left(\langle v\rangle^{\frac{m}{2}}M^{k-1}_{2\tilde s}u\right)\right\|_{L^{2}(\mathbb R^{6}_{x, v})}\\
            %&\quad+\left\|\langle v\rangle^{\frac{m}{2}}\left[M^{k-1}_{2\tilde s}AM^{-(k-1)}_{2\tilde s}, \langle v\rangle^{\frac{-m}{2}}\right]\left(\langle v\rangle^{\frac{m}{2}}M^{k-1}_{2\tilde s}u\right)\right\|_{L^{2}(\mathbb R^{6}_{x, v})}\\
            &\le C_{3}(t+1)\left\|\langle v\rangle^{\frac{m}{2}}M^{k-1}_{2\tilde s}u\right\|_{L^{2}(\mathbb R^{6}_{x, v})}.
        \end{split}
     \end{equation*}
     For the term $J_{3}$, since $u\in H^{k}(\mathbb R^{3}_{x}; \mathcal S(\mathbb R^{3}_{v}))$ and $A(v, D_{x}, D_{v})$ is the pseudo-differential operator of order $2\tilde s-1$, we have $Au\in H^{k}(\mathbb R^{3}_{x}; \mathcal S(\mathbb R^{3}_{v}))\subset H^{k-1}(\mathbb R^{3}_{x}; \mathcal S(\mathbb R^{3}_{v}))$, from \eqref{M-j}, we can obtain
     \begin{equation*}
     \begin{split}
        J_{3}&\le\sum_{p=0}^{k-2}\left(C_{2}(t+1)\right)^{k-1-p}C_{k-1}^{p}\left\|\langle v\rangle^{\frac{m}{2}}M^{p}_{2\tilde s}Au\right\|_{L^{2}(\mathbb R^{6}_{x, v})},
     \end{split}
     \end{equation*}
     then by using the result of $J_{2}$, we have
     \begin{equation*}
     \begin{split}
        J_{3}&\le C_{3}(t+1)\sum_{p=0}^{k-2}\left(C_{2}(t+1)\right)^{k-1-p}C_{k-1}^{p}\left\|\langle v\rangle^{\frac{m}{2}}M^{p}_{2\tilde s}u\right\|_{L^{2}(\mathbb R^{6}_{x, v})}.
     \end{split}
     \end{equation*}
     Taking $C_{2}\ge C_{3}$, we can get that for any $1\le p\le k-2$
     \begin{equation*}
        \begin{split}
             &\left(C_{2}(t+1)\right)^{k-p}\frac{(k-1)!}{(p-1)!(k-p)!}+C_{3}(t+1)\left(C_{2}(t+1)\right)^{k-1-p}\frac{(k-1)!}{p!(k-1-p)!}\\
             %&\le\left(C_{2}(t+1)\right)^{k-p}\frac{(k-1)!\sqrt{k}}{p!(k+1-p)!}\sqrt{(k-p)!}\left(\frac{p}{k-1}+\frac{C_{3}}{C_{2}}\frac{k+1-p}{k-1}\right)\\
             &\le\left(C_{2}(t+1)\right)^{k-p}\frac{k!}{p!(k-p)!}\left(\frac{p}{k}+\frac{C_{3}}{C_{2}}\frac{k-p}{k}\right)\le\left(C_{2}(t+1)\right)^{k-p}C_{k}^{p}.
        \end{split}
     \end{equation*}
     For $p=0$, taking $C_{2}\ge C_{3}$, one has
     \begin{equation*}
        \begin{split}
             &C_{3}(t+1)\left(C_{2}(t+1)\right)^{k-1}\le\left(C_{2}(t+1)\right)^{k}.
        \end{split}
     \end{equation*}
     For $p=k-1$, taking $C_{2}\ge C_{3}$,
     \begin{equation*}
     \begin{split}
          &(k-1)C_{2}(t+1)+C_{3}(t+1)\le kC_{2}(t+1).
     \end{split}
     \end{equation*}
     Combining these results of coefficients and taking $C_{2}\ge C_{3}$, we can obtain that
     \begin{equation*}
        \begin{split}
            &\left\|\langle v\rangle^{-\frac{m}{2}}\left[M^{k}_{2\tilde s}, \langle v\rangle^{m}\right]u\right\|_{L^{2}(\mathbb R^{6}_{x, v})}\le\sum_{p=0}^{k-1}\left(C_{2}(t+1)\right)^{k-p}C_{k}^{p}\left\|\langle v\rangle^{\frac{m}{2}}M^{p}_{2\tilde s}u\right\|_{L^{2}(\mathbb R^{6}_{x, v})}.
        \end{split}
     \end{equation*}
     here $C_{2}$ depends on $m$ and $s$.
\end{proof}
Let $0<s<1$ and $\gamma>-2s$, set $A_{k}=\langle v\rangle^{-\gamma/2}\left[M^{k}_{2\tilde s}, \langle v\rangle^{\gamma}\right]$, then
\begin{lemma}\label{A-k}
     Let $r\in\mathbb R$, then there exists $C_{4}>0$, depends on $r, s$ and $\gamma$, such that for all $k\in\mathbb N_{+}$
     \begin{equation*}%\label{A}
     \begin{split}
          &\left\|\langle D_{v}\rangle^{r}A_{k}u\right\|_{L^{2}(\mathbb R^{6}_{x, v})}=\left\|\langle D_{v}\rangle^{r}\left(\langle v\rangle^{-\gamma/2}\left[M^{k}_{2\tilde s}, \langle v\rangle^{\gamma}\right]u\right)\right\|_{L^{2}(\mathbb R^{6}_{x, v})}\\
          &\le\sum_{p=0}^{k-1}\left(C_{4}(t+1)\right)^{k-p}C_{k}^{p}\left\|\langle D_{v}\rangle^{r}\left(\langle v\rangle^{\frac\gamma2}M^{p}_{2\tilde s}u\right)\right\|_{L^{2}(\mathbb R^{6}_{x, v})}, \ \forall u\in H^{k}(\mathbb R^{3}_{x}; \mathcal S(\mathbb R^{3}_{v})).
     \end{split}
     \end{equation*}
\end{lemma}
\begin{proof}
    We would show it holds by induction on $k$. For $k=1$, as the argument in $R_{t, j}(v, D_{x}, D_{v})$, we have $\langle v\rangle^{-\frac\gamma2}A_{1}=\langle v\rangle^{-\gamma}\left[M_{2\tilde s}, \langle v\rangle^{\gamma}\right]$ is the pseudo-differential operator of order $0$. From $u\in H^{1}(\mathbb R^{3}_{x}; \mathcal S(\mathbb R^{3}_{v}))\in L^{2}(\mathbb R^{3}_{x}; \mathcal S(\mathbb R^{3}_{v}))$, it follows that $\langle v\rangle^{\frac\gamma2}\langle D_{v}\rangle^{r}\left(\langle v\rangle^{-\frac\gamma2}A_{1}u\right)\in L^{2}(\mathbb R^{6}_{x, v})$, so that from Lemma 2.2 of~\cite{H-4}, we have
     \begin{equation*}
     \begin{split}
          &\left\|\langle D_{v}\rangle^{r}A_{1}u\right\|_{L^{2}(\mathbb R^{6}_{x, v})}=\left\|\langle D_{v}\rangle^{r}\left(\langle v\rangle^{\frac\gamma2}\langle v\rangle^{-\frac\gamma2}A_{1}u\right)\right\|_{L^{2}(\mathbb R^{6}_{x, v})}\\
          &\lesssim\left\|\langle v\rangle^{\frac\gamma2}\langle D_{v}\rangle^{r}\left(\langle v\rangle^{-\frac\gamma2}A_{1}u\right)\right\|_{L^{2}(\mathbb R^{6}_{x, v})}\lesssim\left\|\left(\langle v\rangle^{-\frac\gamma2}A_{1}\right)\langle D_{v}\rangle^{r}\left(\langle v\rangle^{\frac\gamma2}u\right)\right\|_{L^{2}(\mathbb R^{6}_{x, v})}\\
          &\le C_{4}(t+1)\left\|\langle D_{v}\rangle^{r}\left(\langle v\rangle^{\frac\gamma2}u\right)\right\|_{L^{2}(\mathbb R^{6}_{x, v})},
          %&=\frac12C_{4}(t+1)\left\|\langle D_{v}\rangle^{r}\left(\langle v\rangle^{\frac\gamma2}u\right)\right\|_{L^{2}(\mathbb R^{6}_{x, v})},
     \end{split}
     \end{equation*}
    with $C_{4}$ depends on $r, s$ and $\gamma$. Assume that $k\ge2$ and for any $1\le j\le k-1$
    \begin{equation*}%\label{A-j}
     \begin{split}
          &\left\|\langle D_{v}\rangle^{r}A_{j}u\right\|_{L^{2}}\\
          &\le\sum_{p=0}^{j-1}\left(C_{4}(t+1)\right)^{j-p}C_{j}^{p}\left\|\langle D_{v}\rangle^{r}\left(\langle v\rangle^{\frac\gamma2}M^{p}_{2\tilde s}u\right)\right\|_{L^{2}}, \quad \forall u\in H^{j}(\mathbb R^{3}_{x}; \mathcal S(\mathbb R^{3}_{v})).
     \end{split}
     \end{equation*}
     From \eqref{M-k-s}, we have
     \begin{equation*}
     \begin{split}
         \langle D_{v}\rangle^{r}A_{k}u&=\langle D_{v}\rangle^{r}A_{k-1}M_{2\tilde s}u+\langle D_{v}\rangle^{r}\left(\langle v\rangle^{\frac\gamma2}M^{k-1}_{2\tilde s}Au\right)+\langle D_{v}\rangle^{r}A_{k-1}Au\\
         &=Q_{1}+Q_{2}+Q_{3}.
     \end{split}
     \end{equation*}
     with $A=\langle v\rangle^{-\gamma}\left[M_{2\tilde s}, \langle v\rangle^{\gamma}\right]$ is the pseudo-differential operator of order $2\tilde s-1$.

     Since $u\in H^{k}(\mathbb R^{3}_{x}; \mathcal S(\mathbb R^{3}_{v}))$, we have $M_{2\tilde s}u\in H^{k-2\tilde s}(\mathbb R^{3}_{x}; \mathcal S(\mathbb R^{3}_{v}))\subset H^{k-1}(\mathbb R^{3}_{x}; \mathcal S(\mathbb R^{3}_{v}))$, then using the induction hypothesis, it follows that
     \begin{equation*}
     \begin{split}
          \left\|Q_{1}\right\|_{L^{2}(\mathbb R^{6}_{x, v})}&\le\sum_{p=0}^{k-2}\left(C_{4}(t+1)\right)^{k-1-p}C_{k-1}^{p}\left\|\langle D_{v}\rangle^{r}\left(\langle v\rangle^{\frac\gamma2}M^{p+1}_{2\tilde s}u\right)\right\|_{L^{2}(\mathbb R^{6}_{x, v})}\\
          &=\sum_{p=1}^{k-1}\left(C_{4}(t+1)\right)^{k-p}\frac{(k-1)!}{(p-1)!(k-p)!}\left\|\langle D_{v}\rangle^{r}\left(\langle v\rangle^{\frac\gamma2}M^{p}_{2\tilde s}u\right)\right\|_{L^{2}(\mathbb R^{6}_{x, v})}.
     \end{split}
     \end{equation*}
     For $Q_{2}$, since $M^{k-1}_{2\tilde s}AM^{-(k-1)}_{2\tilde s}$ is the pseudo-differential operator of order $0$, from $u\in H^{k}(\mathbb R^{3}_{x}; \mathcal S(\mathbb R^{3}_{v}))$, one has
$\langle D_{v}\rangle^{r}(\langle v\rangle^{\frac{m}{2}}M^{k-1}_{2\tilde s}u)\in L^{2}(\mathbb R^{6}_{x, v})$, then using Lemma 2.2 of~\cite{H-4}, we can get that
     \begin{equation*}
     \begin{split}
          \left\|Q_{2}\right\|_{L^{2}(\mathbb R^{6}_{x, v})}&\lesssim\left\|\langle v\rangle^{\frac\gamma2}\langle D_{v}\rangle^{r}M^{k-1}_{2\tilde s}AM^{-(k-1)}_{2\tilde s}M^{k-1}_{2\tilde s}u\right\|_{L^{2}(\mathbb R^{6}_{x, v})}\\
          &\le C_{5}\left\|M^{k-1}_{2\tilde s}AM^{-(k-1)}_{2\tilde s}\langle D_{v}\rangle^{r}\left(\langle v\rangle^{\frac\gamma2}M^{k-1}_{2\tilde s}u\right)\right\|_{L^{2}(\mathbb R^{6}_{x, v})}\\
          &\le\tilde C_{5}(t+1)\left\|\langle D_{v}\rangle^{r}\left(\langle v\rangle^{\frac\gamma2}M^{k-1}_{2\tilde s}u\right)\right\|_{L^{2}(\mathbb R^{6}_{x, v})}.
     \end{split}
     \end{equation*}
     For $Q_{3}$, since $u\in H^{k}(\mathbb R^{3}_{x}; \mathcal S(\mathbb R^{3}_{v}))$, one has $Au\in H^{k}(\mathbb R^{3}_{x}; \mathcal S(\mathbb R^{3}_{v}))\subset H^{k-1}(\mathbb R^{3}_{x}; \mathcal S(\mathbb R^{3}_{v}))$, from the induction hypothesis and the result of $Q_{2}$, it follows that
     \begin{equation*}
     \begin{split}
          \left\|Q_{3}\right\|_{L^{2}(\mathbb R^{6}_{x, v})}&\le\tilde C_{5}(t+1)\sum_{p=0}^{k-2}\left(C_{4}(t+1)\right)^{k-1-p}C_{k-1}^{p}\left\|\langle D_{v}\rangle^{r}\left(\langle v\rangle^{\frac\gamma2}M^{p}_{2\tilde s}u\right)\right\|_{L^{2}(\mathbb R^{6}_{x, v})}.
     \end{split}
     \end{equation*}
     Combining these results and taking $C_{4}\ge\tilde C_{5}$, we can get that
     \begin{equation*}
     \begin{split}
          &\left\|\langle D_{v}\rangle^{r}A_{k}u\right\|_{L^{2}(\mathbb R^{6}_{x, v})}\le\sum_{p=0}^{k-1}\left(C_{4}(t+1)\right)^{k-p}C_{k}^{p}\left\|\langle D_{v}\rangle^{r}\left(\langle v\rangle^{\frac\gamma2}M^{p}_{2\tilde s}u\right)\right\|_{L^{2}(\mathbb R^{6}_{x, v})}.
     \end{split}
     \end{equation*}
\end{proof}

We also need the following interpolations, which have been proved in(~\cite{X-1, H-4}).
\begin{lemma}(~\cite{X-1})\label{interpolation}
     Let $k,l>0$ and $\delta>0$, then for any $\epsilon>0$ there exists a constant $C_{\epsilon}>0$ such that for any $u\in\mathcal S(\mathbb R^{d})$
     $$\|u\|_{H^{k}_{l}(\mathbb R^{d})}\le\epsilon\|u\|_{H^{k+\delta}_{l}(\mathbb R^{d})}+C_{\epsilon}\|u\|_{L^{2}(\mathbb R^{d})}.$$
\end{lemma}
\begin{lemma}(~\cite{H-4})\label{interpolation1}
    Let $k,l\in\mathbb R$ and $\delta>0$, then there exists a constant $C(k,l,\delta)$ such that for any $u\in\mathcal S(\mathbb R^{d})$,
    $$\|u\|^{2}_{H^{k}_{l}(\mathbb R^{d})}\le C(k,l,\delta)\|u\|_{H^{k+\delta}_{2l}(\mathbb R^{d})}\|u\|_{H^{k-\delta}(\mathbb R^{d})}.$$
\end{lemma}

\section{Energy Estimates and Smoothing Solution}
In this section, we study the energy estimates of the solution to the Cauchy problem \eqref{1-1}, and we would prove $u\in C^{\infty}([0, T]; H^{+\infty}(\mathbb R^{3}_{x}; \mathcal S(\mathbb R^{3}_{v})))$ under the suitable initial datum. For $r\ge0$ and $m\in\mathbb N$, set
$$\left\|h\right\|^{2}_{L^{2}_{r}(\mathbb R^{n})}=\int_{\mathbb R^{n}}\left|\langle v\rangle^{r}h(v)\right|^{2}dv,\quad \left\|h\right\|^{2}_{H^{m}_{r}(\mathbb R^{n})}=\sum_{|\alpha|\le m}\int_{\mathbb R^{n}}\left|\langle v\rangle^{r}\partial^{\alpha}_{v}h(v)\right|^{2}dv.$$
\begin{lemma}\label{lemma3.1}
    Let $0<s<1$, $\gamma+2s>0$ and $T>0$. Assume $u_0\in H^{\infty}(\mathbb R^{3}_{x}; \mathcal S(\mathbb R^{3}_{v}))$ and $f\in C^{\infty}([0,T]; H^{\infty}(\mathbb R^{3}_{x}; \mathcal S(\mathbb R^{3}_{v})))$, then the Cauchy problem \eqref{1-1} admits a solution $u\in C^{\infty}\left([0, T]; H^{\infty}(\mathbb R^{3}_{x}; \mathcal S(\mathbb R^{3}_{v}))\right)$. Moreover, if $u_0\in L^{2}(\mathbb R^{6}_{x, v})$ and $f\in L^{\infty}([0,T]; L^{2}(\mathbb R^{6}_{x, v}))$, then there exists a constant $C_{0}>0$, depends on $\gamma$ and $s$, such that for any $t\in[0, T]$,
    \begin{equation}\label{3-1}
    \begin{split}
        &\left\|u(t)\right\|^{2}_{L^{2}(\mathbb R^{6}_{x, v})}+\int_{0}^{t}\left\|\langle D_{v}\rangle^{s}\left(\langle v\rangle^{\frac\gamma2}u(\tau)\right)\right\|^{2}_{L^{2}(\mathbb R^{6}_{x, v})}d\tau\\
         &\quad+\int_{0}^{t}\left\|\langle v\rangle^{\frac\gamma2+s}u(\tau)\right\|^{2}_{L^{2}(\mathbb R^6_{x, v})}d\tau\\
         &\le\left(2C_{0}Te^{2TC_{0}}+1\right)\left(2\int_{0}^{t}\left\|f(\tau)\right\|^{2}_{L^{2}(\mathbb R^6_{x, v})}d\tau+\left\|u_{0}\right\|^{2}_{L^{2}(\mathbb R^6_{x, v})}\right).
    \end{split}
    \end{equation}
\end{lemma}
\begin{proof}
   For any $r\ge0$ and $n, m\in\mathbb N$, let
   $$P=-\partial_{t}+\left(v\cdot\nabla_{x}+\langle v\rangle^{\gamma}\langle D_{v}\rangle^{2s}+\langle v\rangle^{\gamma+2s}\right)^{*},$$
 where the adjoint operator $(\cdot)^{*}$ is taken with respect to the scalar product in $H^{m}(\mathbb R^{3}_{x}; H^{n}_{r}(\mathbb R^{3}_{v}))$. We first consider the case of $n=0$, then for all $0\le t\le T$ and $g\in C^{\infty}([0, T]; \mathcal S(\mathbb R^{6}_{x, v}))$ with $\partial^{\alpha}_{x}\partial^{\beta}_{v}g(T)=0$, $\forall\alpha, \beta\in\mathbb N^{3}$
    \begin{equation*}
        \begin{split}
             &\left(g, Pg\right)_{H^{m}(\mathbb R^{3}_{x}; L^{2}_{r}(\mathbb R^{3}_{v}))}=-\frac{1}{2}\frac{d}{dt}\left\|g(t)\right\|^{2}_{H^{m}(\mathbb R^{3}_{x}; L^{2}_{r}(\mathbb R^{3}_{v}))}+\left(v\cdot\nabla_{x}g, g\right)_{H^{m}(\mathbb R^{3}_{x}; L^{2}_{r}(\mathbb R^{3}_{v}))}\\
             &\quad+\left(\langle v\rangle^{\gamma}\langle D_{v}\rangle^{2s}g, g\right)_{H^{m}(\mathbb R^{3}_{x}; L^{2}_{r}(\mathbb R^{3}_{v}))}+\left\|\langle v\rangle^{\frac\gamma2+s}g(t)\right\|^{2}_{H^{m}(\mathbb R^{3}_{x}; L^{2}_{r}(\mathbb R^{3}_{v}))}, \quad\forall m\in\mathbb N, r\ge0.
        \end{split}
    \end{equation*}
Noting that
$$\left(v\cdot\nabla_{x}g, g\right)_{H^{m}(\mathbb R^{3}_{x}; L^{2}_{r}(\mathbb R^{3}_{v}))}=\sum_{|\alpha|\le m}\left(v\cdot\nabla_{x}\langle v\rangle^{r}\partial^{\alpha}_{x}g, \langle v\rangle^{r}\partial^{\alpha}_{x}g\right)_{L^{2}(\mathbb R^{6}_{x, v})}=0,$$
and
     \begin{equation*}%\label{I-1}
        \begin{split}
             &\left(\langle v\rangle^{\gamma}\langle D_{v}\rangle^{2s}g, g\right)_{H^{m}(\mathbb R^{3}_{x}; L^{2}_{r}(\mathbb R^{3}_{v}))}%=\sum_{|\alpha|\le m}\left(\langle v\rangle^{\gamma+r}\langle D_{v}\rangle^{2s}\partial^{\alpha}_{x}g, \langle v\rangle^{r}\partial^{\alpha}_{x}g\right)_{L^{2}(\mathbb R^{6}_{x, v})}\\
             %&=\left(\langle D_{v}\rangle^{2s}\left(\langle v\rangle^{\frac\gamma2}g\right), \langle v\rangle^{\frac\gamma2}g\right)_{L^{2}(\mathbb R^6_{x, v})}+\left([\langle v\rangle^{\frac\gamma2}, \langle D_{v}\rangle^{2s}]g, \langle v\rangle^{\frac\gamma2}g\right)_{L^{2}(\mathbb R^6_{x, v})}\\
             =\sum_{|\alpha|\le m}\left\|\langle D_{v}\rangle^{s}\left(\langle v\rangle^{\frac\gamma2+r}\partial^{\alpha}_{x}g\right)\right\|^{2}_{L^{2}(\mathbb R^{6}_{x, v})}\\
             &\qquad\qquad\quad+\sum_{|\alpha|\le m}\left([\langle v\rangle^{\frac\gamma2+r}, \langle D_{v}\rangle^{2s}]\partial^{\alpha}_{x}g, \langle v\rangle^{\frac\gamma2+r}\partial^{\alpha}_{x}g\right)_{L^{2}(\mathbb R^{6}_{x, v})},
        \end{split}
    \end{equation*}
it remains to consider the last term on the right-hand side of the above equality. For $0<s\le1/2$, using Cauchy-Schwarz inequality and \eqref{0-s-1/2}, one has for all $|\alpha|\le m$
     \begin{equation*}
        \begin{split}
             &\left|\left([\langle v\rangle^{\frac\gamma2+r}, \langle D_{v}\rangle^{2s}]\partial^{\alpha}_{x}g,\langle v\rangle^{\frac\gamma2+r}\partial^{\alpha}_{x}g\right)_{L^{2}(\mathbb R^{6}_{x, v})}\right|\\
             &\le\left\|\left[\langle D_{v}\rangle^{2s}, \langle v\rangle^{\frac\gamma2+r}\right]\partial^{\alpha}_{x}g\right\|_{L^{2}(\mathbb R^{6}_{x, v})}\left\|\langle v\rangle^{\frac\gamma2+r}\partial^{\alpha}_{x}g\right\|_{L^{2}(\mathbb R^{6}_{x, v})}\le C_{1}\left\|\langle v\rangle^{\frac\gamma2+r}\partial^{\alpha}_{x}g\right\|^{2}_{L^{2}(\mathbb R^{6}_{x, v})}.
        \end{split}
    \end{equation*}
    For $1/2<s<1$, we have $0<2s-1<2s-s=s$, then using \eqref{1/2-s-1} and Cauchy-Schwarz inequality, one has for all $|\alpha|\le m$
     $$\left|\left([\langle v\rangle^{\frac\gamma2+r}, \langle D_{v}\rangle^{2s}]\partial^{\alpha}_{x}g,\langle v\rangle^{\frac\gamma2+r}\partial^{\alpha}_{x}g\right)_{L^{2}(\mathbb R^{6}_{x, v})}\right|\le C_{1}\left\|\langle D_{v}\rangle^{2s-1}\left(\langle v\rangle^{\frac\gamma2+r}\partial^{\alpha}_{x}g\right)\right\|^{2}_{L^{2}(\mathbb R^{6}_{x ,v})}.$$
    Since $g\in\mathcal S(\mathbb R^{6}_{x, v})$, it follows that
    $$\left\|\langle v\rangle^{\frac\gamma2+r}\partial^{\alpha}_{x}g\right\|^{2}_{L^{2}(\mathbb R^{6}_{x, v})}=\int_{\mathbb R^{3}_{x}}\left\|\langle \cdot\rangle^{\frac\gamma2+r}\partial^{\alpha}_{x}g(t, x, \cdot)\right\|^{2}_{L^{2}(\mathbb R^{3}_{v})}dx,$$
    and
    $$\left\|\langle D_{v}\rangle^{2s-1}\left(\langle v\rangle^{\frac\gamma2+r}\partial^{\alpha}_{x}g\right)\right\|^{2}_{L^{2}(\mathbb R^{6}_{x ,v})}=\int_{\mathbb R^{3}_{x}}\left\|\langle D_{v}\rangle^{2s-1}\left(\langle\cdot\rangle^{\frac\gamma2+r}\partial^{\alpha}_{x}g(t, x, \cdot)\right)\right\|^{2}_{L^{2}(\mathbb R^{3}_{v})}dx,$$
    for $\gamma>0$, applying Lemma \ref{interpolation} with $\epsilon=\frac{1}{2\sqrt{C_{1}}}$, it follows that
    \begin{equation*}
     \begin{split}
         &C_{1}\left\|\langle \cdot\rangle^{\frac\gamma2+r}\partial^{\alpha}_{x}g(t, x, \cdot)\right\|^{2}_{L^{2}(\mathbb R^{3}_{v})}\le C_{1}\left\|\langle D_{v}\rangle^{s/2}\left(\langle \cdot\rangle^{\frac\gamma2+r}\partial^{\alpha}_{x}g(t, x, \cdot)\right)\right\|^{2}_{L^{2}(\mathbb R^{3}_{v})}\\
         &\le C_{1}\left(\frac{1}{2\sqrt{C_{1}}}\left\|\langle \cdot\rangle^{\frac\gamma2+r}\partial^{\alpha}_{x}g(t, x, \cdot)\right\|_{H^{s}(\mathbb R^3_{v})}+\tilde C_{1}\left\|\langle \cdot\rangle^{r}\partial^{\alpha}_{x}g(t, x, \cdot)\right\|_{L^{2}(\mathbb R^3_{v})}\right)^{2}\\
         &\le\frac12\left\|\langle \cdot\rangle^{\frac\gamma2+r}\partial^{\alpha}_{x}g(t, x, \cdot)\right\|^{2}_{H^{s}(\mathbb R^3_{v})}+2C_{1}(\tilde C_{1})^{2}\left\|\langle \cdot\rangle^{r}\partial^{\alpha}_{x}g(t, x, \cdot)\right\|^{2}_{L^{2}(\mathbb R^3_{v})},
      \end{split}
     \end{equation*}
     and
     \begin{equation*}
     \begin{split}
        &C_{1}\left\|\langle D_{v}\rangle^{2s-1}\left(\langle\cdot\rangle^{\frac\gamma2+r}\partial^{\alpha}_{x}g(t, x, \cdot)\right)\right\|^{2}_{L^{2}(\mathbb R^{3}_{v})}\\
        &\le\frac12\left\|\langle \cdot\rangle^{\frac\gamma2+r}\partial^{\alpha}_{x}g(t, x, \cdot)\right\|^{2}_{H^{s}(\mathbb R^3_{v})}+2C_{1}(\tilde C_{1})^{2}\left\|\langle \cdot\rangle^{r}\partial^{\alpha}_{x}g(t, x, \cdot)\right\|^{2}_{L^{2}(\mathbb R^3_{v})},
     \end{split}
     \end{equation*}
     here the constant $\tilde C_{1}$ depends on $C_{1}$. For the case of $\gamma\le0$, by using Lemma \ref{interpolation1} and Cauchy-Schwarz inequality, it follows that $C_{1}\left\|\langle \cdot\rangle^{\frac\gamma2+r}\partial^{\alpha}_{x}g(t, x, \cdot)\right\|^{2}_{L^{2}(\mathbb R^{3}_{v})}$ and $C_{1}\left\|\langle D_{v}\rangle^{2s-1}\left(\langle\cdot\rangle^{\frac\gamma2+r}\partial^{\alpha}_{x}g(t, x, \cdot)\right)\right\|^{2}_{L^{2}(\mathbb R^{3}_{v})}$ can be bounded by
      \begin{equation*}
         \frac12\left\|\langle \cdot\rangle^{\frac\gamma2+r}\partial^{\alpha}_{x}g(t, x, \cdot)\right\|^{2}_{H^{s}(\mathbb R^3_{v})}+C_{6}\left\|\langle \cdot\rangle^{r}\partial^{\alpha}_{x}g(t, x, \cdot)\right\|^{2}_{L^{2}(\mathbb R^3_{v})},
     \end{equation*}
     with $C_{6}$ depends on $C_{1}$ and $\gamma, s, r$. Thus for all $\gamma>-2s$, $0<s<1$
     \begin{equation}\label{I-2}
     \begin{split}
         &\left|\left([\langle v\rangle^{\frac\gamma2+r}, \langle D_{v}\rangle^{2s}]\partial^{\alpha}_{x}g,\langle v\rangle^{\frac\gamma2+r}\partial^{\alpha}_{x}g\right)_{L^{2}(\mathbb R^{6}_{x, v})}\right|\\
         &\le\frac12\left\|\langle D_{v}\rangle^{s}\left(\langle v\rangle^{\frac\gamma2+r}\partial^{\alpha}_{x}g\right)\right\|^{2}_{L^{2}(\mathbb R^6_{x, v})}+\tilde C_{6}\left\|\langle v\rangle^{r}\partial^{\alpha}_{x}g\right\|^{2}_{L^{2}(\mathbb R^6_{x, v})}, \ \forall \alpha, r,
     \end{split}
     \end{equation}
     here $\tilde C_{6}=\max\{C_{6}, 2C_{1}(\tilde C_{1})^{2}\}$. Combining these results and using Cauchy-Schwarz inequality, one has for all $m\in\mathbb N$
    \begin{equation*}
        \begin{split}
             &-\frac{1}{2}\frac{d}{dt}\left\|g(t)\right\|^{2}_{H^{m}(\mathbb R^{3}_{x}; L^{2}_{r}(\mathbb R^{3}_{v}))}+\frac12\sum_{|\alpha|\le m}\left\|\langle D_{v}\rangle^{s}\left(\langle v\rangle^{\frac\gamma2+r}\partial^{\alpha}_{x}g(t)\right)\right\|^{2}_{L^{2}(\mathbb R^6_{x, v})}\\
             &\quad+\left\|\langle v\rangle^{\frac\gamma2+s}g(t)\right\|^{2}_{H^{m}(\mathbb R^{3}_{x}; L^{2}_{r}(\mathbb R^{3}_{v}))}-\tilde C_{6}\left\|g(t)\right\|^{2}_{H^{m}(\mathbb R^{3}_{x}; L^{2}_{r}(\mathbb R^{3}_{v}))}\\
             &\le\left\|g(t)\right\|_{H^{m}(\mathbb R^{3}_{x}; L^{2}_{r}(\mathbb R^{3}_{v}))}\left\|Pg(t)\right\|_{H^{m}(\mathbb R^{3}_{x}; L^{2}_{r}(\mathbb R^{3}_{v}))},
        \end{split}
    \end{equation*}
    which implies that for all $m\in\mathbb N$ and $r\ge0$
    \begin{equation*}
        \begin{split}
             &-\frac{d}{dt}\left(e^{2\tilde C_{6}t}\left\|g(t)\right\|^{2}_{H^{m}(\mathbb R^{3}_{x}; L^{2}_{r}(\mathbb R^{3}_{v}))}\right)+2e^{2\tilde C_{6}t}\left\|\langle v\rangle^{\frac\gamma2+s}g(t)\right\|^{2}_{H^{m}(\mathbb R^{3}_{x}; L^{2}_{r}(\mathbb R^{3}_{v}))}\\
             &\quad+e^{2\tilde C_{6}t}\sum_{|\alpha|\le m}\left\|\langle D_{v}\rangle^{s}\left(\langle v\rangle^{\frac\gamma2+r}\partial^{\alpha}_{x}g(t)\right)\right\|^{2}_{L^{2}(\mathbb R^6_{x, v})}\\
             %&\quad+2e^{2\tilde C_{6}t}\left\|\langle v\rangle^{\frac\gamma2+s}g(t)\right\|^{2}_{H^{m}(\mathbb R^{3}_{x}; L^{2}_{r}(\mathbb R^{3}_{v}))}\\
             &\le2e^{2\tilde C_{6}t}\left\|g(t)\right\|_{H^{m}(\mathbb R^{3}_{x}; L^{2}_{r}(\mathbb R^{3}_{v}))}\left\|Pg(t)\right\|_{H^{m}(\mathbb R^{3}_{x}; L^{2}_{r}(\mathbb R^{3}_{v}))}.
        \end{split}
    \end{equation*}
    Since $\partial^{\alpha}_{x}g(T)=0$, it follows that for all $t\in[0, T]$,  $m\in\mathbb N$ and $r\ge0$
    \begin{equation}\label{n-0}
        \begin{split}
             &\left\|g(t)\right\|^{2}_{H^{m}(\mathbb R^{3}_{x}; L^{2}_{r}(\mathbb R^{3}_{v}))}+2\int_{t}^{T}e^{2\tilde C_{6}(\tau-t)}\left\|\langle v\rangle^{\frac\gamma2+s}g(\tau)\right\|^{2}_{H^{m}(\mathbb R^{3}_{x}; L^{2}_{r}(\mathbb R^{3}_{v}))}d\tau\\
             &\quad+\sum_{|\alpha|\le m}\int_{t}^{T}e^{2\tilde C_{6}(\tau-t)}\left\|\langle D_{v}\rangle^{s}\left(\langle v\rangle^{\frac\gamma2+r}\partial^{\alpha}_{x}g(\tau)\right)\right\|^{2}_{L^{2}(\mathbb R^6_{x, v})}d\tau\\
             %&\quad+2\int_{t}^{T}e^{2\tilde C_{6}(\tau-t)}\left\|\langle v\rangle^{\frac\gamma2+s}g(\tau)\right\|^{2}_{H^{m}(\mathbb R^{3}_{x}; L^{2}_{r}(\mathbb R^{3}_{v}))}d\tau\\
            &\le2e^{2\tilde C_{6}T}\int_{t}^{T}\left\|g(\tau)\right\|_{H^{m}(\mathbb R^{3}_{x}; L^{2}_{r}(\mathbb R^{3}_{v}))}\left\|Pg(\tau)\right\|_{H^{m}(\mathbb R^{3}_{x}; L^{2}_{r}(\mathbb R^{3}_{v}))}d\tau\\
            &\le2e^{2\tilde C_{6}T}\int_{t}^{T}\left\|g(\tau)\right\|_{H^{\infty}(\mathbb R^{3}_{x}; L^{2}_{r}(\mathbb R^{3}_{v}))}\left\|Pg(\tau)\right\|_{H^{\infty}(\mathbb R^{3}_{x}; L^{2}_{r}(\mathbb R^{3}_{v}))}d\tau.
        \end{split}
    \end{equation}

    Next, for all $n\in\mathbb N$, we would show by induction that there exists $C_{7}>0$ such that for all $t\in[0, T]$
    \begin{equation*}
        \begin{split}
             &\left\|g(t)\right\|^{2}_{H^{m}(\mathbb R^{3}_{x}; H^{n}_{r}(\mathbb R^{3}_{v}))}+2\int_{t}^{T}e^{2C_{7}(\tau-t)}\left\|\langle v\rangle^{\frac\gamma2+s}g(\tau)\right\|^{2}_{H^{m}(\mathbb R^{3}_{x}; H^{n}_{r}(\mathbb R^{3}_{v}))}d\tau\\
             &\quad+\sum_{|\alpha|\le m}\sum_{|\beta|\le n}\int_{t}^{T}e^{2C_{7}(\tau-t)}\left\|\langle D_{v}\rangle^{s}\left(\langle v\rangle^{\frac\gamma2+r}\partial^{\alpha}_{x}\partial^{\beta}_{v}g(\tau)\right)\right\|^{2}_{L^{2}(\mathbb R^6_{x, v})}d\tau\\
             %&\quad+2\int_{t}^{T}e^{2C_{7}(\tau-t)}\left\|\langle v\rangle^{\frac\gamma2+s}g(\tau)\right\|^{2}_{H^{m}(\mathbb R^{3}_{x}; L^{2}_{r}(\mathbb R^{3}_{v}))}d\tau\\
            &\le2e^{2C_{7}T}\int_{t}^{T}\left\|g(\tau)\right\|_{H^{\infty}(\mathbb R^{3}_{x}; H^{n}_{r}(\mathbb R^{3}_{v}))}\left\|Pg(\tau)\right\|_{H^{\infty}(\mathbb R^{3}_{x}; H^{n}_{r}(\mathbb R^{3}_{v}))}d\tau, \ \forall m\in\mathbb N, r\ge0.
        \end{split}
    \end{equation*}
    For $n=0$, it is true from \eqref{n-0}. Assume $n\ge1$, for all $t\in[0, T]$ and $0\le j\le n-1$
    \begin{equation*}
        \begin{split}
             &\left\|g(t)\right\|^{2}_{H^{m}(\mathbb R^{3}_{x}; H^{j}_{r}(\mathbb R^{3}_{v}))}+2\int_{t}^{T}e^{2C_{7}(\tau-t)}\left\|\langle v\rangle^{\frac\gamma2+s}g(\tau)\right\|^{2}_{H^{m}(\mathbb R^{3}_{x}; H^{j}_{r}(\mathbb R^{3}_{v}))}d\tau\\
             &\quad+\sum_{|\alpha|\le m}\sum_{|\beta|\le j}\int_{t}^{T}e^{2C_{7}(\tau-t)}\left\|\langle D_{v}\rangle^{s}\left(\langle v\rangle^{\frac\gamma2+r}\partial^{\alpha}_{x}\partial^{\beta}_{v}g(\tau)\right)\right\|^{2}_{L^{2}(\mathbb R^6_{x, v})}d\tau\\
             %&\quad+2\int_{t}^{T}e^{2C_{7}(\tau-t)}\left\|\langle v\rangle^{\frac\gamma2+s}g(\tau)\right\|^{2}_{H^{m}(\mathbb R^{3}_{x}; L^{2}_{r}(\mathbb R^{3}_{v}))}d\tau\\
            &\le2e^{2C_{7}T}\int_{t}^{T}\left\|g(\tau)\right\|_{H^{\infty}(\mathbb R^{3}_{x}; H^{j}_{r}(\mathbb R^{3}_{v}))}\left\|Pg(\tau)\right\|_{H^{\infty}(\mathbb R^{3}_{x}; H^{j}_{r}(\mathbb R^{3}_{v}))}d\tau, \ \forall m\in\mathbb N, r\ge0.
        \end{split}
    \end{equation*}
    Then, for $n$, we have
    \begin{equation*}
        \begin{split}
             &\left(g, Pg\right)_{H^{m}(\mathbb R^{3}_{x}; H^{n}_{r}(\mathbb R^{3}_{v}))}=-\frac{1}{2}\frac{d}{dt}\left\|g(t)\right\|^{2}_{H^{m}(\mathbb R^{3}_{x}; H^{n}_{r}(\mathbb R^{3}_{v}))}+\left(v\cdot\nabla_{x}g, g\right)_{H^{m}(\mathbb R^{3}_{x}; H^{n}_{r}(\mathbb R^{3}_{v}))}\\
             &\quad+\left(\langle v\rangle^{\gamma}\langle D_{v}\rangle^{2s}g, g\right)_{H^{m}(\mathbb R^{3}_{x}; H^{n}_{r}(\mathbb R^{3}_{v}))}+\left\|\langle v\rangle^{\frac\gamma2+s}g(t)\right\|^{2}_{H^{m}(\mathbb R^{3}_{x}; H^{n}_{r}(\mathbb R^{3}_{v}))}, \  \forall m\in\mathbb N, r\ge0.
        \end{split}
    \end{equation*}
    From the Leibniz formula, it follows that
    \begin{equation*}
        \begin{split}
             &\left(v\cdot\nabla_{x}g, g\right)_{H^{m}(\mathbb R^{3}_{x}; H^{n}_{r}(\mathbb R^{3}_{v}))}=\sum_{|\alpha|\le m}\sum_{|\beta|\le n}\left(\langle v\rangle^{r}\partial^{\beta}_{v}\left(v\cdot\nabla_{x}\partial^{\alpha}_{x}g\right), \langle v\rangle^{r}\partial^{\alpha}_{x}\partial^{\beta}_{v}g\right)_{L^{2}}\\
             &=\left(v\cdot\nabla_{x}g, g\right)_{H^{m}(\mathbb R^{3}_{x}; L^{2}_{r}(\mathbb R^{3}_{v}))}+\sum_{|\alpha|\le m}\sum_{1\le|\beta|\le n}\sum_{j=1}^{3}\left(\langle v\rangle^{r}\partial^{\beta-e_{j}}_{v}\partial^{\alpha+e_{j}}_{x}g, \langle v\rangle^{r}\partial^{\alpha}_{x}\partial^{\beta}_{v}g\right)_{L^{2}}\\
             &=\sum_{|\alpha|\le m}\sum_{1\le|\beta|\le n}\sum_{j=1}^{3}\left(\langle v\rangle^{r}\partial^{\beta-e_{j}}_{v}\partial^{\alpha+e_{j}}_{x}g, \langle v\rangle^{r}\partial^{\alpha}_{x}\partial^{\beta}_{v}g\right)_{L^{2}},
        \end{split}
    \end{equation*}
    using Cauchy-Schwarz inequality, we have
    \begin{equation*}
        \begin{split}
             &\left|\left(v\cdot\nabla_{x}g, g\right)_{H^{m}(\mathbb R^{3}_{x}; H^{n}_{r}(\mathbb R^{3}_{v}))}\right|\le\sum_{|\alpha|\le m}\sum_{1\le|\beta|\le n}\left\|\langle v\rangle^{r}\partial^{\beta-1}_{v}\partial^{\alpha+1}_{x}g\right\|_{L^{2}}\left\|\langle v\rangle^{r}\partial^{\alpha}_{x}\partial^{\beta}_{v}g\right\|_{L^{2}}\\
             &\le\bigg(\sum_{|\alpha|\le m}\sum_{1\le|\beta|\le n}\left\|\langle v\rangle^{r}\partial^{\beta-1}_{v}\partial^{\alpha+1}_{x}g\right\|_{L^{2}}\bigg)^{\frac12}\bigg(\sum_{|\alpha|\le m}\sum_{1\le|\beta|\le n}\left\|\langle v\rangle^{r}\partial^{\alpha}_{x}\partial^{\beta}_{v}g\right\|^{2}_{L^{2}}\bigg)^{\frac12}\\
             &\le\left\|g\right\|_{H^{m+1}(\mathbb R^{3}_{x}; H^{n-1}_{r}(\mathbb R^{3}_{v}))}\left\|g\right\|_{H^{m}(\mathbb R^{3}_{x}; H^{n}_{r}(\mathbb R^{3}_{v}))}\\
             &\le\frac14\left\|g\right\|^{2}_{H^{m+1}(\mathbb R^{3}_{x}; H^{n-1}_{r}(\mathbb R^{3}_{v}))}+\left\|g\right\|^{2}_{H^{m}(\mathbb R^{3}_{x}; H^{n}_{r}(\mathbb R^{3}_{v}))}.
        \end{split}
    \end{equation*}
   By using the Leibniz formula, it follows that
    \begin{equation*}
        \begin{split}
             &\left(\langle v\rangle^{\gamma}\langle D_{v}\rangle^{2s}g, g\right)_{H^{m}(\mathbb R^{3}_{x}; H^{n}_{r}(\mathbb R^{3}_{v}))}=\sum_{|\alpha|\le m}\sum_{|\beta|\le n}\left(\langle v\rangle^{r}\partial^{\beta}_{v}\left(\langle v\rangle^{\gamma}\langle D_{v}\rangle^{2s}\partial^{\alpha}_{x}g\right), \langle v\rangle^{r}\partial^{\alpha}_{x}\partial^{\beta}_{v}g\right)_{L^{2}}\\
             %&=\left(\langle D_{v}\rangle^{2s}\left(\langle v\rangle^{\frac\gamma2}g\right), \langle v\rangle^{\frac\gamma2}g\right)_{L^{2}(\mathbb R^6_{x, v})}+\left([\langle v\rangle^{\frac\gamma2}, \langle D_{v}\rangle^{2s}]g, \langle v\rangle^{\frac\gamma2}g\right)_{L^{2}(\mathbb R^6_{x, v})}\\
             %&=\sum_{|\alpha|\le m}\sum_{|\beta|\le n}\sum_{\beta_{1}+\beta_{2}=\beta}C_{\beta}^{\beta_{1}}\left(\langle v\rangle^{r}\partial^{\beta_{1}}_{v}\langle v\rangle^{\gamma}\langle D_{v}\rangle^{2s}\partial^{\alpha}_{x}\partial^{\beta_{2}}_{v}g, \langle v\rangle^{r}\partial^{\alpha}_{x}\partial^{\beta}_{v}g\right)_{L^{2}}\\
             &=\sum_{|\alpha|\le m}\sum_{|\beta|\le n}\sum_{0<\beta_{1}\le\beta}C_{\beta}^{\beta_{1}}\left(\langle D_{v}\rangle^{2s}\left(\langle v\rangle^{r-\frac\gamma2}\partial^{\beta_{1}}_{v}\langle v\rangle^{\gamma}\partial^{\alpha}_{x}\partial^{\beta-\beta_{1}}_{v}g\right), \langle v\rangle^{\frac\gamma2+r}\partial^{\alpha}_{x}\partial^{\beta}_{v}g\right)_{L^{2}}\\
             &\quad+\sum_{|\alpha|\le m}\sum_{|\beta|\le n}\sum_{0<\beta_{1}\le\beta}C_{\beta}^{\beta_{1}}\left(\left[\langle v\rangle^{r-\frac\gamma2}\partial^{\beta_{1}}_{v}\langle v\rangle^{\gamma}, \langle D_{v}\rangle^{2s}\right]\partial^{\alpha}_{x}\partial^{\beta-\beta_{1}}_{v}g, \langle v\rangle^{\frac\gamma2+r}\partial^{\alpha}_{x}\partial^{\beta}_{v}g\right)_{L^{2}}\\
             &\quad+\sum_{|\alpha|\le m}\sum_{|\beta|\le n}\left\|\langle D_{v}\rangle^{s}\left(\langle v\rangle^{\frac\gamma2+r}\partial^{\alpha}_{x}\partial^{\beta}_{v}g\right)\right\|^{2}_{L^{2}}\\
             &\quad+\sum_{|\alpha|\le m}\sum_{|\beta|\le n}\left(\left[\langle v\rangle^{\frac\gamma2+r}, \langle D_{v}\rangle^{2s}\right]\partial^{\alpha}_{x}\partial^{\beta}_{v}g, \langle v\rangle^{\frac\gamma2+r}\partial^{\alpha}_{x}\partial^{\beta}_{v}g\right)_{L^{2}}.
        \end{split}
    \end{equation*}
    For $0<s\le1/2$, by using Cauchy-Schwarz inequality and \eqref{0-s-1/2}, one has for all $\alpha, \beta$
    $$\left|\left(\left[\langle v\rangle^{\frac\gamma2+r}, \langle D_{v}\rangle^{2s}\right]\partial^{\alpha}_{x}\partial^{\beta}_{v}g, \langle v\rangle^{\frac\gamma2+r}\partial^{\alpha}_{x}\partial^{\beta}_{v}g\right)_{L^{2}}\right|\le C_{1}\left\|\langle v\rangle^{\frac\gamma2+r}\partial^{\alpha}_{x}\partial^{\beta}_{v}g\right\|^{2}_{L^{2}},$$
    and for $1/2<s<1$, we have $0<2s-1<s$, then using \eqref{1/2-s-1} and Cauchy-Schwarz inequality, one has for all $\alpha, \beta$
    $$\left|\left(\left[\langle v\rangle^{\frac\gamma2+r}, \langle D_{v}\rangle^{2s}\right]\partial^{\alpha}_{x}\partial^{\beta}_{v}g, \langle v\rangle^{\frac\gamma2+r}\partial^{\alpha}_{x}\partial^{\beta}_{v}g\right)_{L^{2}}\right|\le C_{1}\left\|\langle D_{v}\rangle^{2s-1}\left(\langle v\rangle^{\frac\gamma2+r}\partial^{\alpha}_{x}\partial^{\beta}_{v}g\right)\right\|^{2}_{L^{2}}.$$
    Then for $\gamma>0$, applying Lemma \ref{interpolation} with $\epsilon=\frac{1}{2\sqrt{2C_{1}}}$, for $\gamma+2r\le0$, using Lemma \ref{interpolation1} and Cauchy-Schwarz inequality, it follows that for all $|\alpha|\le m, |\beta|\le n, r\ge0$
    $$ C_{1}\left\|\langle \cdot\rangle^{\frac\gamma2+r}\partial^{\alpha}_{x}\partial^{\beta}_{v}g\right\|^{2}_{L^{2}(\mathbb R^{3}_{v})}\le\frac14\left\|\langle \cdot\rangle^{\frac\gamma2+r}\partial^{\alpha}_{x}\partial^{\beta}_{v}g\right\|^{2}_{H^{s}(\mathbb R^3_{v})}+C_{8}\left\|\partial^{\alpha}_{x}\partial^{\beta}_{v}g\right\|^{2}_{L^{2}_{r}(\mathbb R^3_{v})},$$
    and
    $$C_{1}\left\|\langle D_{v}\rangle^{2s-1}\left(\langle \cdot\rangle^{\frac\gamma2+r}\partial^{\alpha}_{x}\partial^{\beta}_{v}g\right)\right\|^{2}_{L^{2}(\mathbb R^{3}_{v})}\le\frac14\left\|\langle \cdot\rangle^{\frac\gamma2+r}\partial^{\alpha}_{x}\partial^{\beta}_{v}g\right\|^{2}_{H^{s}(\mathbb R^3_{v})}+C_{8}\left\|\partial^{\alpha}_{x}\partial^{\beta}_{v}g\right\|^{2}_{L^{2}_{r}(\mathbb R^3_{v})},$$
    so that for all $0<s<1$, $\gamma>-2s$ and $|\alpha|\le m, |\beta|\le n, r\ge0$
     \begin{equation*}
     \begin{split}
         &\left|\left([\langle v\rangle^{\frac\gamma2+r}, \langle D_{v}\rangle^{2s}]\partial^{\alpha}_{x}\partial^{\beta}_{v}g,\langle v\rangle^{\frac\gamma2+r}\partial^{\alpha}_{x}\partial^{\beta}_{v}g\right)_{L^{2}(\mathbb R^{6}_{x, v})}\right|\\
         &\le\frac14\left\|\langle D_{v}\rangle^{s}\left(\langle v\rangle^{\frac\gamma2+r}\partial^{\alpha}_{x}\partial^{\beta}_{v}g\right)\right\|^{2}_{L^{2}(\mathbb R^6_{x, v})}+C_{8}\left\|\langle v\rangle^{r}\partial^{\alpha}_{x}\partial^{\beta}_{v}g\right\|^{2}_{L^{2}(\mathbb R^6_{x, v})}.
     \end{split}
     \end{equation*}
    From Cauchy-Schwarz inequality, it follows that
    \begin{equation*}
    \begin{split}
         &\left|\left(\langle D_{v}\rangle^{2s}\left(\langle v\rangle^{r-\frac\gamma2}\partial^{\beta_{1}}_{v}\langle v\rangle^{\gamma}\partial^{\alpha}_{x}\partial^{\beta-\beta_{1}}_{v}g\right), \langle v\rangle^{\frac\gamma2+r}\partial^{\alpha}_{x}\partial^{\beta}_{v}g\right)_{L^{2}}\right|\\
         %&=\left|\left(\langle D_{v}\rangle^{s}\left(\langle v\rangle^{r-\frac\gamma2}\partial^{\beta_{1}}_{v}\langle v\rangle^{\gamma}\partial^{\alpha}_{x}\partial^{\beta_{2}}_{v}g\right), \langle D_{v}\rangle^{s}\left(\langle v\rangle^{\frac\gamma2+r}\partial^{\alpha}_{x}\partial^{\beta}_{v}g\right)\right)_{L^{2}}\right|\\
         &\le\left\|\langle D_{v}\rangle^{s}\left(\langle v\rangle^{r-\frac\gamma2}\partial^{\beta_{1}}_{v}\langle v\rangle^{\gamma}\partial^{\alpha}_{x}\partial^{\beta-\beta_{1}}_{v}g\right)\right\|_{L^{2}(\mathbb R^6_{x, v})}\left\|\langle D_{v}\rangle^{s}\left(\langle v\rangle^{\frac\gamma2+r}\partial^{\alpha}_{x}\partial^{\beta}_{v}g\right)\right\|_{L^{2}(\mathbb R^6_{x, v})}.
         %&\le C_{\gamma, s, r, \beta}\bigg(\sum_{0<\beta_{1}\le\beta}\left\|\langle D_{v}\rangle^{s}\left(\langle v\rangle^{\frac\gamma2+r}\partial^{\alpha}_{x}\partial^{\beta-\beta_{1}}_{v}g\right)\right\|_{L^{2}}\bigg)^{2}+\frac18\left\|\langle D_{v}\rangle^{s}\left(\langle v\rangle^{\frac\gamma2+r}\partial^{\alpha}_{x}\partial^{\beta}_{v}g\right)\right\|^{2}_{L^{2}}.
    \end{split}
    \end{equation*}
    Since $0<s<1$ and $\left|\langle v\rangle^{r-\frac\gamma2}\partial^{\beta_{1}}_{v}\langle v\rangle^{\gamma}\right|\le C(\gamma, \beta_{1})\langle v\rangle^{\frac\gamma2+r}$, we have
    \begin{equation*}
    \begin{split}
         &\left\|\langle D_{v}\rangle^{s}\left(\langle v\rangle^{r-\frac\gamma2}\partial^{\beta_{1}}_{v}\langle v\rangle^{\gamma}\partial^{\alpha}_{x}\partial^{\beta-\beta_{1}}_{v}g\right)\right\|_{L^{2}(\mathbb R^6_{x, v})}\\
         &\le\sum_{|\sigma|\le1}\left\|\partial^{\sigma}_{v}\left(\langle v\rangle^{r-\frac\gamma2}\partial^{\beta_{1}}_{v}\langle v\rangle^{\gamma}\partial^{\alpha}_{x}\partial^{\beta-\beta_{1}}_{v}g\right)\right\|_{L^{2}(\mathbb R^6_{x, v})}\\
         &\le C(\gamma, r, \beta_{1})\bigg(\left\|\langle v\rangle^{\frac\gamma2+r}\partial^{\alpha}_{x}\partial^{\beta-\beta_{1}}_{v}g\right\|_{L^{2}(\mathbb R^6_{x, v})}+\sum_{j=1}^{3}\left\|\langle v\rangle^{\frac\gamma2+r}\partial^{\alpha}_{x}\partial^{\beta-\beta_{1}+e_{j}}_{v}g\right\|_{L^{2}(\mathbb R^6_{x, v})}\bigg),
    \end{split}
    \end{equation*}
    then from Cauchy-Schwarz inequality
    \begin{equation*}
    \begin{split}
         &\left|\sum_{|\alpha|\le m}\sum_{|\beta|\le n}\sum_{0<\beta_{1}\le\beta}C_{\beta}^{\beta_{1}}\left(\langle D_{v}\rangle^{2s}\left(\langle v\rangle^{r-\frac\gamma2}\partial^{\beta_{1}}_{v}\langle v\rangle^{\gamma}\partial^{\alpha}_{x}\partial^{\beta-\beta_{1}}_{v}g\right), \langle v\rangle^{\frac\gamma2+r}\partial^{\alpha}_{x}\partial^{\beta}_{v}g\right)_{L^{2}}\right|\\
         %&=\left|\left(\langle D_{v}\rangle^{s}\left(\langle v\rangle^{r-\frac\gamma2}\partial^{\beta_{1}}_{v}\langle v\rangle^{\gamma}\partial^{\alpha}_{x}\partial^{\beta_{2}}_{v}g\right), \langle D_{v}\rangle^{s}\left(\langle v\rangle^{\frac\gamma2+r}\partial^{\alpha}_{x}\partial^{\beta}_{v}g\right)\right)_{L^{2}}\right|\\
         %&\le\sum_{0<\beta_{1}\le\beta}C_{\beta}^{\beta_{1}}\left\|\langle D_{v}\rangle^{s}\left(\langle v\rangle^{r-\frac\gamma2}\partial^{\beta_{1}}_{v}\langle v\rangle^{\gamma}\partial^{\alpha}_{x}\partial^{\beta-\beta_{1}}_{v}g\right)\right\|_{L^{2}}\left\|\langle D_{v}\rangle^{s}\left(\langle v\rangle^{\frac\gamma2+r}\partial^{\alpha}_{x}\partial^{\beta}_{v}g\right)\right\|_{L^{2}}\\
         &\le C_{9}\sum_{|\alpha|\le m}\sum_{|\beta|\le n}\left\|\langle v\rangle^{\frac\gamma2+r}\partial^{\alpha}_{x}\partial^{\beta}_{v}g\right\|_{L^{2}}^{2}+\frac{1}{16}\sum_{|\alpha|\le m}\sum_{|\beta|\le n}\left\|\langle D_{v}\rangle^{s}\left(\langle v\rangle^{\frac\gamma2+r}\partial^{\alpha}_{x}\partial^{\beta}_{v}g\right)\right\|^{2}_{L^{2}}.
    \end{split}
    \end{equation*}
    For the first term on the right-hand side of the above inequality, if $\gamma>0$, applying Lemma \ref{interpolation} with $\epsilon=\frac{1}{4\sqrt{2C_{9}}}$, it follows that
    \begin{equation*}
    \begin{split}
         &C_{9}\left\|\langle v\rangle^{\frac\gamma2+r}\partial^{\alpha}_{x}\partial^{\beta}_{v}g\right\|_{L^{2}(\mathbb R^{6}_{x, v})}^{2}=C_{9}\int_{\mathbb R^{3}_{x}}\left\|\langle \cdot\rangle^{\frac\gamma2+r}\partial^{\alpha}_{x}\partial^{\beta}_{v}g(t, x, \cdot)\right\|^{2}_{L^{2}(\mathbb R^{3}_{v})}dx\\
         &\le C_{9}\int_{\mathbb R^{3}_{x}}\left\|\langle D_{v}\rangle^{s/2}\left(\langle \cdot\rangle^{\frac\gamma2+r}\partial^{\alpha}_{x}\partial^{\beta}_{v}g(t, x, \cdot)\right)\right\|^{2}_{L^{2}(\mathbb R^{3}_{v})}dx\\
         &\le C_{9}\int_{\mathbb R^{3}_{x}}\left(\frac{1}{4\sqrt{2C_{9}}}\left\|\langle \cdot\rangle^{\frac\gamma2+r}\partial^{\alpha}_{x}\partial^{\beta}_{v}g(t, x, \cdot)\right\|_{H^{s}(\mathbb R^3_{v})}+\tilde C_{9}\left\|\partial^{\alpha}_{x}\partial^{\beta}_{v}g(t, x, \cdot)\right\|_{L^{2}_{r}(\mathbb R^3_{v})}\right)^{2}dx\\
         &\le\frac{1}{16}\left\|\langle D_{v}\rangle^{s}\left(\langle v\rangle^{\frac\gamma2+r}\partial^{\alpha}_{x}\partial^{\beta}_{v}g\right)\right\|^{2}_{L^{2}(\mathbb R^6_{x, v})}+2C_{9}(\tilde C_{9})^{2}\left\|\langle v\rangle^{r}\partial^{\alpha}_{x}\partial^{\beta}_{v}g\right\|^{2}_{L^{2}(\mathbb R^6_{x, v})},
    \end{split}
    \end{equation*}
    if $\gamma\le0$, by using Lemma \ref{interpolation1} and Cauchy-Schwarz inequality, it follows that
    \begin{equation*}
    \begin{split}
         &C_{9}\left\|\langle v\rangle^{\frac\gamma2+r}\partial^{\alpha}_{x}\partial^{\beta}_{v}g\right\|_{L^{2}(\mathbb R^6_{x, v})}^{2}\\
         &\le\frac{1}{16}\left\|\langle D_{v}\rangle^{s}\left(\langle v\rangle^{\frac\gamma2+r}\partial^{\alpha}_{x}\partial^{\beta}_{v}g\right)\right\|^{2}_{L^{2}(\mathbb R^6_{x, v})}+C_{10}\left\|\langle v\rangle^{r}\partial^{\alpha}_{x}\partial^{\beta}_{v}g\right\|^{2}_{L^{2}(\mathbb R^6_{x, v})}.
    \end{split}
    \end{equation*}
    Hence, for all $0<s<1$ and $\gamma>-2s$, we have
    \begin{equation*}
    \begin{split}
         &\left|\sum_{|\alpha|\le m}\sum_{|\beta|\le n}\sum_{0<\beta_{1}\le\beta}C_{\beta}^{\beta_{1}}\left(\langle D_{v}\rangle^{2s}\left(\langle v\rangle^{r-\frac\gamma2}\partial^{\beta_{1}}_{v}\langle v\rangle^{\gamma}\partial^{\alpha}_{x}\partial^{\beta-\beta_{1}}_{v}g\right), \langle v\rangle^{\frac\gamma2+r}\partial^{\alpha}_{x}\partial^{\beta}_{v}g\right)_{L^{2}(\mathbb R^6_{x, v})}\right|\\
         %&\le C_{9}\sum_{|\alpha|\le m}\sum_{|\beta|\le n}\left\|\langle v\rangle^{\frac\gamma2+r}\partial^{\alpha}_{x}\partial^{\beta}_{v}g\right\|_{L^{2}}^{2}+\frac{1}{16}\sum_{|\alpha|\le m}\sum_{|\beta|\le n}\left\|\langle D_{v}\rangle^{s}\left(\langle v\rangle^{\frac\gamma2+r}\partial^{\alpha}_{x}\partial^{\beta}_{v}g\right)\right\|^{2}_{L^{2}}\\
         &\le\tilde C_{10}\left\|g(t)\right\|^{2}_{H^{m}(\mathbb R^{3}_{x}; H^{n}_{r}(\mathbb R^{3}_{x}))}+\frac{1}{8}\sum_{|\alpha|\le m}\sum_{|\beta|\le n}\left\|\langle D_{v}\rangle^{s}\left(\langle v\rangle^{\frac\gamma2+r}\partial^{\alpha}_{x}\partial^{\beta}_{v}g\right)\right\|^{2}_{L^{2}(\mathbb R^6_{x, v})},
    \end{split}
    \end{equation*}
    here $\tilde C_{10}=\max\{2C_{9}(\tilde C_{9})^{2}, C_{10}\}$. As the arguments in Lemma \ref{lemma2.2}, there exists a constant $C_{11}>0$, depends on $\gamma, s, r$ and $\beta_{1}$, such that
    \begin{equation*}
    \begin{split}
         &\left\|\left[\langle v\rangle^{r-\frac\gamma2}\partial^{\beta_{1}}_{v}\langle v\rangle^{\gamma}, \langle D_{v}\rangle^{2s}\right]\partial^{\alpha}_{x}\partial^{\beta-\beta_{1}}_{v}g\right\|_{L^{2}(\mathbb R^6_{x, v})}\\
         &\le C_{11}\left\|\langle D_{v}\rangle^{2s-1}\left(\langle v\rangle^{\frac\gamma2+r}\partial^{\alpha}_{x}\partial^{\beta-\beta_{1}}_{v}g\right)\right\|_{L^{2}(\mathbb R^6_{x, v})}.
    \end{split}
    \end{equation*}
    %then by using Lemma 2.2 of~\cite{H-4}, it follows that
    %$$\left\|\langle D_{v}\rangle^{2s-1}\left(\langle v\rangle^{\frac\gamma2+r-|\beta_{1}|}\partial^{\alpha}_{x}\partial^{\beta-\beta_{1}}_{v}g\right)\right\|_{L^{2}}\le C_{\gamma, s, r, \beta_{1}}\left\|\langle D_{v}\rangle^{2s-1}\left(\langle v\rangle^{\frac\gamma2+r}\partial^{\alpha}_{x}\partial^{\beta-\beta_{1}}_{v}g\right)\right\|_{L^{2}},$$
    Since $2s-1<1$, it follows that
    \begin{equation*}
    \begin{split}
         &\left\|\langle D_{v}\rangle^{2s-1}\left(\langle v\rangle^{\frac\gamma2+r}\partial^{\alpha}_{x}\partial^{\beta-\beta_{1}}_{v}g\right)\right\|_{L^{2}(\mathbb R^6_{x, v})}\le\sum_{|\sigma|\le1}\left\|\partial^{\sigma}_{v}\left(\langle v\rangle^{\frac\gamma2+r}\partial^{\alpha}_{x}\partial^{\beta-\beta_{1}}_{v}g\right)\right\|_{L^{2}(\mathbb R^6_{x, v})}\\
         &\le C(\gamma, r, \beta_{1})\bigg(\left\|\langle v\rangle^{\frac\gamma2+r}\partial^{\alpha}_{x}\partial^{\beta-\beta_{1}}_{v}g\right\|_{L^{2}(\mathbb R^6_{x, v})}+\sum_{j=1}^{3}\left\|\langle v\rangle^{\frac\gamma2+r}\partial^{\alpha}_{x}\partial^{\beta-\beta_{1}+e_{j}}_{v}g\right\|_{L^{2}(\mathbb R^6_{x, v})}\bigg).
    \end{split}
    \end{equation*}
    Hence, from Cauchy-Schwarz inequality, we have
    \begin{equation*}
    \begin{split}
         &\left|\sum_{|\alpha|\le m}\sum_{|\beta|\le n}\sum_{0<\beta_{1}\le\beta}C_{\beta}^{\beta_{1}}\left(\left[\langle v\rangle^{r-\frac\gamma2}\partial^{\beta_{1}}_{v}\langle v\rangle^{\gamma}, \langle D_{v}\rangle^{2s}\right]\partial^{\alpha}_{x}\partial^{\beta-\beta_{1}}_{v}g, \langle v\rangle^{\frac\gamma2+r}\partial^{\alpha}_{x}\partial^{\beta}_{v}g\right)_{L^{2}(\mathbb R^6_{x, v})}\right|\\
         %&\le\sum_{|\alpha|\le m}\sum_{|\beta|\le n}\sum_{0<\beta_{1}\le\beta}C_{\beta}^{\beta_{1}}\left\|\left[\langle v\rangle^{r-\frac\gamma2}\partial^{\beta_{1}}_{v}\langle v\rangle^{\gamma}, \langle D_{v}\rangle^{2s}\right]\partial^{\alpha}_{x}\partial^{\beta-\beta_{1}}_{v}g\right\|_{L^{2}}\left\|\langle v\rangle^{\frac\gamma2+r}\partial^{\alpha}_{x}\partial^{\beta}_{v}g\right\|_{L^{2}}\\
         &\le\tilde C_{11}\sum_{|\alpha|\le m}\sum_{|\beta|\le n}\left\|\langle v\rangle^{\frac\gamma2+r}\partial^{\alpha}_{x}\partial^{\beta}_{v}g\right\|_{L^{2}(\mathbb R^6_{x, v})}^{2}+\frac{1}{16}\left\|\langle D_{v}\rangle^{s}\left(\langle v\rangle^{\frac\gamma2+r}\partial^{\alpha}_{x}\partial^{\beta}_{v}g\right)\right\|^{2}_{L^{2}(\mathbb R^6_{x, v})}.
    \end{split}
    \end{equation*}
    For $\gamma>0$, applying Lemma \ref{interpolation} with $\epsilon=\frac{1}{4\sqrt{2\tilde C_{11}}}$, and for $\gamma\le0$, applying Lemma \ref{interpolation1} and Cauchy-Schwarz inequality, it follows that
    \begin{equation*}
    \begin{split}
         &\tilde C_{11}\left\|\langle v\rangle^{\frac\gamma2+r}\partial^{\alpha}_{x}\partial^{\beta}_{v}g\right\|_{L^{2}(\mathbb R^6_{x, v})}^{2}\\
         &\le\frac{1}{16}\left\|\langle D_{v}\rangle^{s}\left(\langle v\rangle^{\frac\gamma2+r}\partial^{\alpha}_{x}\partial^{\beta}_{v}g\right)\right\|^{2}_{L^{2}(\mathbb R^6_{x, v})}+C_{12}\left\|\langle v\rangle^{r}\partial^{\alpha}_{x}\partial^{\beta}_{v}g\right\|^{2}_{L^{2}(\mathbb R^6_{x, v})}.
    \end{split}
    \end{equation*}
    So we can get that
    \begin{equation*}
    \begin{split}
         &\left|\sum_{|\alpha|\le m}\sum_{|\beta|\le n}\sum_{0<\beta_{1}\le\beta}C_{\beta}^{\beta_{1}}\left(\left[\langle v\rangle^{r-\frac\gamma2}\partial^{\beta_{1}}_{v}\langle v\rangle^{\gamma}, \langle D_{v}\rangle^{2s}\right]\partial^{\alpha}_{x}\partial^{\beta-\beta_{1}}_{v}g, \langle v\rangle^{\frac\gamma2+r}\partial^{\alpha}_{x}\partial^{\beta}_{v}g\right)_{L^{2}(\mathbb R^6_{x, v})}\right|\\
         &\le C_{12}\left\|g(t)\right\|^{2}_{H^{m}(\mathbb R^{x}; H^{n}_{r}(\mathbb R^{3}_{v}))}+\frac{1}{8}\left\|\langle D_{v}\rangle^{s}\left(\langle v\rangle^{\frac\gamma2+r}\partial^{\alpha}_{x}\partial^{\beta}_{v}g\right)\right\|^{2}_{L^{2}(\mathbb R^6_{x, v})}.
    \end{split}
    \end{equation*}
    Combining these results, it follows that for all $m\in\mathbb N$ and $r\ge0$
    \begin{equation*}
        \begin{split}
             &-\frac{1}{2}\frac{d}{dt}\left\|g(t)\right\|^{2}_{H^{m}(\mathbb R^{3}_{x}; H^{n}_{r}(\mathbb R^{3}_{v}))}+\frac12\sum_{|\alpha|\le m}\sum_{|\beta|\le n}\left\|\langle D_{v}\rangle^{s}\left(\langle v\rangle^{\frac\gamma2+r}\partial^{\alpha}_{x}\partial^{\beta}_{v}g(t)\right)\right\|^{2}_{L^{2}(\mathbb R^{6}_{x, v})}\\
             &\quad+\left\|\langle v\rangle^{\frac\gamma2+s}g(t)\right\|^{2}_{H^{m}(\mathbb R^{3}_{x}; H^{n}_{r}(\mathbb R^{3}_{v}))}-(C_{8}+\tilde C_{10}+C_{12}+1)\left\|g(t)\right\|^{2}_{H^{m}(\mathbb R^{3}_{x}; H^{n}_{r}(\mathbb R^{3}_{v}))}\\
             &\le\left\|g(t)\right\|_{H^{m}(\mathbb R^{3}_{x}; H^{n}_{r}(\mathbb R^{3}_{v}))}\left\|Pg(t)\right\|_{H^{m}(\mathbb R^{3}_{x}; H^{n}_{r}(\mathbb R^{3}_{v}))}+\frac14\left\|g(t)\right\|^{2}_{H^{m+1}(\mathbb R^{3}_{x}; H^{n-1}_{r}(\mathbb R^{3}_{v}))},
        \end{split}
    \end{equation*}
    which implies that for all $m\in\mathbb N$ and $r\ge0$
    \begin{equation*}
        \begin{split}
             &-\frac{d}{dt}\left(e^{2\tilde C_{12}t}\left\|g\right\|^{2}_{H^{m}(\mathbb R^{3}_{x}; H^{n}_{r}(\mathbb R^{3}_{v}))}\right)+2e^{2\tilde C_{12}t}\left\|\langle v\rangle^{\frac\gamma2+s}g\right\|^{2}_{H^{m}(\mathbb R^{3}_{x}; H^{n}_{r}(\mathbb R^{3}_{v}))}\\
             &\quad+e^{2\tilde C_{12}t}\sum_{|\alpha|\le m}\sum_{|\beta|\le n}\left\|\langle D_{v}\rangle^{s}\left(\langle v\rangle^{\frac\gamma2+r}\partial^{\alpha}_{x}\partial^{\beta}_{v}g\right)\right\|^{2}_{L^{2}(\mathbb R^{6}_{x, v})}\\
             %&\quad+2e^{2C_{13}t}\left\|\langle v\rangle^{\frac\gamma2+s}g(t)\right\|^{2}_{H^{m}(\mathbb R^{3}_{x}; H^{n}_{r}(\mathbb R^{3}_{v}))}\\
             &\le2e^{2\tilde C_{12}t}\left\|g\right\|_{H^{m}(\mathbb R^{3}_{x}; H^{n}_{r}(\mathbb R^{3}_{v}))}\left\|Pg\right\|_{H^{m}(\mathbb R^{3}_{x}; H^{n}_{r}(\mathbb R^{3}_{v}))}+\frac12e^{2\tilde C_{12}t}\left\|g\right\|^{2}_{H^{m+1}(\mathbb R^{3}_{x}; H^{n-1}_{r}(\mathbb R^{3}_{v}))},
        \end{split}
    \end{equation*}
    here $\tilde C_{12}=C_{8}+\tilde C_{10}+C_{12}+1$.
    Since $\partial^{\alpha}_{x}\partial^{\beta}_{v}g(T)=0$ and $\gamma+2s>0$, it follows that for all $t\in[0, T]$, $m\in\mathbb N$ and $r\ge0$
    \begin{equation*}
        \begin{split}
             &\left\|g(t)\right\|^{2}_{H^{m}(\mathbb R^{3}_{x}; H^{n}_{r}(\mathbb R^{3}_{v}))}+2\int_{t}^{T}e^{2\tilde C_{12}(\tau-t)}\left\|\langle v\rangle^{\frac\gamma2+s}g(\tau)\right\|^{2}_{H^{m}(\mathbb R^{3}_{x}; H^{n}_{r}(\mathbb R^{3}_{v}))}d\tau\\
             &\quad+\sum_{|\alpha|\le m}\sum_{|\beta|\le n}\int_{t}^{T}e^{2\tilde C_{12}(\tau-t)}\left\|\langle D_{v}\rangle^{s}\left(\langle v\rangle^{\frac\gamma2+r}\partial^{\alpha}_{x}\partial^{\beta}_{v}g(\tau)\right)\right\|^{2}_{L^{2}(\mathbb R^{6}_{x, v})}d\tau\\
             %&\quad+2\int_{t}^{T}e^{2C_{13}(\tau-t)}\left\|\langle v\rangle^{\frac\gamma2+s}g\right\|^{2}_{H^{m}_{x}(H^{n}_{r})}d\tau\\
             &\le2\int_{t}^{T}e^{2\tilde C_{12}(\tau-t)}\left\|g(\tau)\right\|_{H^{m}(\mathbb R^{3}_{x}; H^{n}_{r}(\mathbb R^{3}_{v}))}\left\|Pg(\tau)\right\|_{H^{m}(\mathbb R^{3}_{x}; H^{n}_{r}(\mathbb R^{3}_{v}))}d\tau\\
             &\quad+\frac12\int_{t}^{T}e^{2(\tilde C_{12}-C_{7})(\tau-t)}e^{2C_{7}(\tau-t)}\left\|\langle v\rangle^{\frac\gamma2+s}g(\tau)\right\|^{2}_{H^{m+1}(\mathbb R^{3}_{x}; H^{n-1}_{r}(\mathbb R^{3}_{v}))}d\tau.
        \end{split}
    \end{equation*}
    Then, taking $C_{7}\ge\max\{2\tilde C_{12}, \tilde C_{12}+\frac{1}{2T}\ln\frac{4}{3}\}$, from the induction hypothesis, we can obtain that for all $t\in[0, T]$, $m\in\mathbb N$ and $r\ge0$
    \begin{equation*}
        \begin{split}
             &\left\|g(t)\right\|^{2}_{H^{m}(\mathbb R^{3}_{x}; H^{n}_{r}(\mathbb R^{3}_{v}))}+2\int_{t}^{T}e^{2\tilde C_{12}(\tau-t)}\left\|\langle v\rangle^{\frac\gamma2+s}g(\tau)\right\|^{2}_{H^{m}(\mathbb R^{3}_{x}; H^{n}_{r}(\mathbb R^{3}_{v}))}d\tau\\
             &\quad+\sum_{|\alpha|\le m}\sum_{|\beta|\le n}\int_{t}^{T}e^{2\tilde C_{12}(\tau-t)}\left\|\langle D_{v}\rangle^{s}\left(\langle v\rangle^{\frac\gamma2+r}\partial^{\alpha}_{x}\partial^{\beta}_{v}g(\tau)\right)\right\|^{2}_{L^{2}(\mathbb R^{6}_{x, v})}d\tau\\
             &\le2\int_{t}^{T}e^{\tilde C_{12}(\tau-t)}\left\|g(\tau)\right\|_{H^{m}(\mathbb R^{3}_{x}; H^{n}_{r}(\mathbb R^{3}_{v}))}\left\|Pg(\tau)\right\|_{H^{m}(\mathbb R^{3}_{x}; H^{n}_{r}(\mathbb R^{3}_{v}))}d\tau\\
             &\quad+\frac12\int_{t}^{T}e^{2C_{7}(\tau-t)}\left\|g(\tau)\right\|_{H^{\infty}(\mathbb R^{3}_{x}; H^{n-1}_{r}(\mathbb R^{3}_{v}))}\left\|Pg(\tau)\right\|_{H^{\infty}(\mathbb R^{3}_{x}; H^{n-1}_{r}(\mathbb R^{3}_{v}))}d\tau\\
             &\le\left(2e^{2\tilde C_{12}T}+\frac12e^{2C_{7}T}\right)\int_{t}^{T}\left\|g(\tau)\right\|_{H^{\infty}(\mathbb R^{3}_{x}; H^{n}_{r}(\mathbb R^{3}_{v}))}\left\|Pg(\tau)\right\|_{H^{\infty}(\mathbb R^{3}_{x}; H^{n}_{r}(\mathbb R^{3}_{v}))}d\tau\\
             &\le\left(2\cdot\frac{3}{4}e^{2C_{7}T}+\frac12e^{2C_{7}T}\right)\int_{t}^{T}\left\|g(\tau)\right\|_{H^{\infty}(\mathbb R^{3}_{x}; H^{n}_{r}(\mathbb R^{3}_{v}))}\left\|Pg(\tau)\right\|_{H^{\infty}(\mathbb R^{3}_{x}; H^{n}_{r}(\mathbb R^{3}_{v}))}d\tau\\
             &\le2e^{2C_{7}T}\int_{t}^{T}\left\|g(\tau)\right\|_{H^{\infty}(\mathbb R^{3}_{x}; H^{n}_{r}(\mathbb R^{3}_{v}))}\left\|Pg(\tau)\right\|_{H^{\infty}(\mathbb R^{3}_{x}; H^{n}_{r}(\mathbb R^{3}_{v}))}d\tau.
        \end{split}
    \end{equation*}
    And Therefore, we can get that for all $t\in[0, T]$, $m\in\mathbb N$
    \begin{equation*}
        \begin{split}
             &\left\|g(t)\right\|^{2}_{H^{m}(\mathbb R^{3}_{x}; H^{n}_{r}(\mathbb R^{3}_{v}))}\\
             &\le2e^{2C_{7}T}\int_{t}^{T}\left\|g(\tau)\right\|_{H^{\infty}(\mathbb R^{3}_{x}; H^{n}_{r}(\mathbb R^{3}_{v}))}\left\|Pg(\tau)\right\|_{H^{\infty}(\mathbb R^{3}_{x}; H^{n}_{r}(\mathbb R^{3}_{v}))}d\tau\\
             &\le2e^{2C_{7}T}\left\|g\right\|_{L^{\infty}([0, T]; H^{\infty}(\mathbb R^{3}_{x}; H^{n}_{r}(\mathbb R^{3}_{v})))}\int_{t}^{T}\left\|Pg(\tau)\right\|_{H^{\infty}(\mathbb R^{3}_{x}; H^{n}_{r}(\mathbb R^{3}_{v}))}d\tau,
        \end{split}
    \end{equation*}
    so that we have for all $n\in\mathbb N$ and $r\ge0$
    \begin{equation}\label{sup-g}
        \begin{split}
             \left\|g\right\|_{L^{\infty}([0, T]; H^{\infty}(\mathbb R^{3}_{x}; H^{n}_{r}(\mathbb R^{3}_{v})))}&\le2e^{2C_{7}T}\int_{0}^{T}\left\|Pg(\tau)\right\|_{H^{\infty}(\mathbb R^{3}_{x}; H^{n}_{r}(\mathbb R^{3}_{v}))}d\tau.
        \end{split}
    \end{equation}
    For all $n\in\mathbb N$ and $r\ge0$, considering the subspace
    $$\Omega=\left\{Pg: g\in C^{\infty}([0, T]; S(\mathbb R^{6}_{x, v})), g(T)=0\right\}\subset L^{1}([0, T]; H^{\infty}(\mathbb R^{3}; H^{n}_{r}(\mathbb R^{3}_{v})),$$
   define the linear functional
    \begin{equation*}
        \begin{split}
             F: \Omega&\ \ \ \to\ \ \ \mathbb R\ \ \\
              w=Pg&\ \ \ \mapsto\ \ \ (u_{0}, g(0))_{H^{\infty}(\mathbb R^{3}_{x}, H^{n}_{r}(\mathbb R^{3}_{v}))}+\int_{0}^{T}(f(t), g(t))_{H^{\infty}(\mathbb R^{3}_{x}, H^{n}_{r}(\mathbb R^{3}_{v}))}dt.
        \end{split}
    \end{equation*}
    Let $Pg_{1}=Pg_{2}$ with $g_{1}, g_{2}\in C^{\infty}([0, T]; S(\mathbb R^{6}_{x, v}))$ and $g_{1}(T)=g_{2}(T)=0$, then from \eqref{sup-g}, it follows that
    \begin{equation*}
        \begin{split}
             &\left\|g_{1}-g_{2}\right\|_{L^{\infty}([0, T]; H^{\infty}(\mathbb R^{3}_{x}, H^{n}_{r}(\mathbb R^{3}_{v})))}\\
             &\le2e^{2C_{7}T}\int_{0}^{T}\left\|P\left(g_{1}(\tau)-g_{2}(\tau)\right)\right\|_{H^{\infty}(\mathbb R^{3}_{x}, H^{n}_{r}(\mathbb R^{3}_{v}))}d\tau=0,
        \end{split}
    \end{equation*}
   hence $g_{1}=g_{2}$, so the operator $P$ is injective. From \eqref{sup-g}, the linear functional $F$ is well-defined and continuous
    \begin{equation*}
        \begin{split}
             \left|F(Pg)\right|%&\le\left\|u_{0}\right\|_{H^{\infty}(\mathbb R^{3}_{x}, H^{n}_{r}(\mathbb R^{3}_{v}))}\left\|g(0)\right\|_{H^{\infty}(\mathbb R^{3}_{x}, H^{n}_{r}(\mathbb R^{3}_{v}))}\\
             %&\quad+\int_{0}^{T}\left\|f(t)\right\|_{H^{\infty}(\mathbb R^{3}_{x}, H^{n}_{r}(\mathbb R^{3}_{v}))}\left\|g(t)\right\|_{H^{\infty}(\mathbb R^{3}_{x}, H^{n}_{r}(\mathbb R^{3}_{v}))}dt\\
             &\le\left(\left\|u_{0}\right\|_{H^{\infty}(\mathbb R^{3}_{x}, H^{n}_{r}(\mathbb R^{3}_{v}))}+\left\|f\right\|_{L^{1}([0, T]; H^{\infty}(\mathbb R^{3}_{x}, H^{n}_{r}(\mathbb R^{3}_{v})))}\right)\\
             &\quad\times\left\|g\right\|_{L^{\infty}\left([0,T]; H^{\infty}(\mathbb R^{3}_{x}, H^{n}_{r}(\mathbb R^{3}_{v}))\right)}\\
             &\le2e^{2C_{7}T}\left(\left\|u_{0}\right\|_{H^{\infty}(\mathbb R^{3}_{x}, H^{n}_{r}(\mathbb R^{3}_{v}))}+\left\|f\right\|_{L^{1}([0, T]; H^{\infty}(\mathbb R^{3}_{x}, H^{n}_{r}(\mathbb R^{3}_{v})))}\right)\\
             &\quad\times\int_{0}^{T}\left\|Pg(\tau)\right\|_{H^{\infty}(\mathbb R^{3}_{x}, H^{n}_{r}(\mathbb R^{3}_{v}))}d\tau.
        \end{split}
    \end{equation*}
    By using the Hahn-Banach theorem, the linear functional $F$ can be extended as a continuous linear functional on $L^{1}\left([0, T]; H^{\infty}(\mathbb R^{3}_{x}; H^{n}_{r}(\mathbb R^{3}_{v}))\right)$ with
$$\left\|F\right\|\le2e^{2C_{7}T}\left(\left\|u_{0}\right\|_{H^{\infty}(\mathbb R^{3}_{x}, H^{n}_{r}(\mathbb R^{3}_{v}))}+\left\|f\right\|_{L^{1}([0, T]; H^{\infty}(\mathbb R^{3}_{x}, H^{n}_{r}(\mathbb R^{3}_{v})))}\right).$$
Hence there exists $u\in L^{\infty}([0, T]; H^{\infty}(\mathbb R^{3}_{x}; H^{n}_{r}(\mathbb R^{3}_{v})))$ satisfying
$$\int_{0}^{T}\left(u, w\right)_{H^{\infty}(\mathbb R^{3}_{x}, H^{n}_{r}(\mathbb R^{3}_{v}))}dt=F(w), \quad \forall w\in L^{1}\left([0, T]; H^{\infty}(\mathbb R^{3}_{x}, H^{n}_{r}(\mathbb R^{3}_{v}))\right),$$
and
   \begin{equation*}
        \begin{split}
             &\left\|u\right\|_{L^{\infty}([0, T]; H^{\infty}(\mathbb R^{3}_{x}, H^{n}_{r}(\mathbb R^{3}_{v})))}\\
             &\le2e^{2C_{7}T}\left(\left\|u_{0}\right\|_{H^{\infty}(\mathbb R^{3}_{x}, H^{n}_{r}(\mathbb R^{3}_{v}))}+\left\|f\right\|_{L^{1}([0, T]; H^{\infty}(\mathbb R^{3}_{x}, H^{n}_{r}(\mathbb R^{3}_{v})))}\right).
            %&\quad+2e^{2C_{5}T}\int_{0}^{T}\left\|f(\tau)\right\|_{L^{2}(\mathbb R^6_{x, v})}d\tau.
        \end{split}
    \end{equation*}
Since $\Omega\subset L^{1}([0, T]; H^{\infty}_{x}(H^{n}_{r}))$, it follows that for all $g\in C^{\infty}([0, T]; S(\mathbb R^{6}_{x, v}))$
\begin{equation*}
\begin{split}
     F(Pg)&=\int_{0}^{T}\left(u, Pg\right)_{H^{\infty}(\mathbb R^{3}_{x}, H^{n}_{r}(\mathbb R^{3}_{v}))}dt\\
     &=(u_{0}, g(0))_{H^{\infty}(\mathbb R^{3}_{x}, H^{n}_{r}(\mathbb R^{3}_{v}))}+\int_{0}^{T}(f(t), g(t))_{H^{\infty}(\mathbb R^{3}_{x}, H^{n}_{r}(\mathbb R^{3}_{v}))}dt.
\end{split}
\end{equation*}
Thus for all $n\in\mathbb N$ and $r\ge0$, $u\in L^{\infty}([0, T]; H^{\infty}(\mathbb R^{3}_{x}, H^{n}_{r}(\mathbb R^{3}_{v})))$ is a weak solution of the Cauchy problem \eqref{1-1}.

Next, we would show by induction that for all integers $k$,
\begin{equation}\label{S-0}
     \frac{d^{k}}{dt^{k}}u(t)\in L^\infty ([0, T]; H^{\infty}(\mathbb R^{3}_{x}; \mathcal S(\mathbb R^{3}_{v}))).
\end{equation}
For $k=0$, we have already proved. Assume that \eqref{S-0} holds for $k$, then we try to prove
$$\frac{d^{k+1}}{dt^{k+1}}u(t)\in L^\infty ([0, T]; H^{\infty}(\mathbb R^{3}_{x}; \mathcal S(\mathbb R^{3}_{v})) ).$$
Since $u$ is the solution of the Cauchy problem \eqref{1-1}, we have
$$\frac{d^{k+1}}{dt^{k+1}}u(t)=\partial^{k}_{t}f-v\cdot\nabla_{x}\partial^{k}_{t}u-\langle v\rangle^{\gamma}\langle D_{v}\rangle^{2s}\partial^{k}_{t}u-\langle v\rangle^{\gamma+2s}\partial^{k}_{t}u.$$
By using the induction hypothesis, it follows that
$$v\cdot\nabla_{x}\partial^{k}_{t}u,\ \langle v\rangle^{\gamma}\langle D_{v}\rangle^{2s}\partial^{k}_{t}u,\ \langle v\rangle^{\gamma+2s}\partial^{k}_{t}u\in L^\infty ([0, T]; H^{\infty}(\mathbb R^{3}_{x}; \mathcal S(\mathbb R^{3}_{v})) ),$$
since $f\in C^{\infty}([0,T]; H^{\infty}(\mathbb R^{3}_{x}; \mathcal S(\mathbb R^{3}_{v})))$, we have
$$\frac{d^{k+1}}{dt^{k+1}}u(t)\in L^\infty ([0, T]; H^{\infty}(\mathbb R^{3}_{x}; \mathcal S(\mathbb R^{3}_{v})) ).$$
Hence, we can obtain that $u\in C^{\infty}([0, T]; H^{\infty}(\mathbb R^{3}_{x}; \mathcal S(\mathbb R^{3}_{v})))$.

It remains to show \eqref{3-1}. Since $u$ is the solution of the Cauchy problem \eqref{1-1}, we have
     \begin{equation}\label{3-1-1}
     \begin{split}
         &\left(\left(\partial_{t}+v\cdot\nabla_{x}\right)u, u\right)_{L^{2}(\mathbb R^6_{x, v})}+\left(\langle v\rangle^{\gamma}\langle D_{v}\rangle^{2s}u, u\right)_{L^{2}(\mathbb R^6_{x, v})}\\
         &\quad+\left\|\langle v\rangle^{\frac\gamma2+s}u(t)\right\|^{2}_{L^{2}(\mathbb R^6_{x, v})}=(f, u)_{L^{2}(\mathbb R^6_{x, v})}.
     \end{split}
     \end{equation}
     Noting that $\left(\left(\partial_{t}+v\cdot\nabla_{x}\right)u, u\right)_{L^{2}(\mathbb R^6_{x, v})}=\frac{d}{dt}\left\|u(t)\right\|^{2}_{L^{2}(\mathbb R^6_{x, v})}$ and
     $$\left(\langle v\rangle^{\gamma}\langle D_{v}\rangle^{2s}u, u\right)=\left\|\langle D_{v}\rangle^{s}\left(\langle v\rangle^{\gamma/2}u\right)\right\|_{L^{2}(\mathbb R^6_{x, v})}+\left(\left[\langle v\rangle^{\gamma/2}, \langle D_{v}\rangle^{2s}\right], \langle v\rangle^{\gamma/2}u\right),$$
     using \eqref{I-2} with $g=u$ and $|\alpha|=r=0$, it follows that for all $\gamma>-2s$ and $0<s<1$
     \begin{equation}\label{dt-0-1}
     \begin{split}
         &\frac{d}{dt}\left\|u(t)\right\|^{2}_{L^{2}(\mathbb R^6_{x, v})}+\left\|\langle D_{v}\rangle^{s}\left(\langle v\rangle^{\frac\gamma2}u(t)\right)\right\|^{2}_{L^{2}(\mathbb R^{6}_{x, v})}+\left\|\langle v\rangle^{\frac\gamma2+s}u(t)\right\|^{2}_{L^{2}(\mathbb R^6_{x, v})}\\
         &\le2C_{0}\left\|u(t)\right\|^{2}_{L^{2}(\mathbb R^6_{x, v})}+2\left\|f(t)\right\|^{2}_{L^{2}(\mathbb R^6_{x, v})},
     \end{split}
     \end{equation}
     with $C_{0}$ depends on $\gamma$ and $s$. For all $t\in[0,T]$, integrating from $0$ to $t$, one has
     \begin{equation*}
     \begin{split}
         &\left\|u(t)\right\|^{2}_{L^{2}(\mathbb R^6_{x, v})}+\int_{0}^{t}\left\|\langle D_{v}\rangle^{s}\left(\langle v\rangle^{\frac\gamma2}u(\tau)\right)\right\|^{2}_{L^{2}(\mathbb R^{6}_{x, v})}d\tau+\int_{0}^{t}\left\|\langle v\rangle^{\frac\gamma2+s}u(\tau)\right\|^{2}_{L^{2}(\mathbb R^6_{x, v})}d\tau\\
         &\le2C_{0}\int_{0}^{t}\left\|u(\tau)\right\|^{2}_{L^{2}(\mathbb R^6_{x, v})}d\tau+2\int_{0}^{t}\left\|f(\tau)\right\|^{2}_{L^{2}(\mathbb R^6_{x, v})}d\tau+\left\|u_{0}\right\|^{2}_{L^{2}(\mathbb R^6_{x, v})},
     \end{split}
     \end{equation*}
by using Gronwall inequality, we have
$$\left\|u(t)\right\|^{2}_{L^{2}(\mathbb R^6_{x, v})}\le e^{2C_{0}T}\left(2\int_{0}^{t}\left\|f(\tau)\right\|^{2}_{L^{2}(\mathbb R^6_{x, v})}d\tau+\left\|u_{0}\right\|^{2}_{L^{2}(\mathbb R^6_{x, v})}\right).$$
And therefore, it follows that for any $t\in[0,T]$,
     %$$B_{0}\ge\left(2\tilde C_{6}Te^{2T\tilde C_{6}}+1\right)\left(2T\left\|f\right\|^{2}_{L^{\infty}\left([0, T]; L^{2}(\mathbb R^6_{x, v})\right)}+\left\|u_{0}\right\|^{2}_{L^{2}(\mathbb R^6_{x, v})}\right),$$
     %it follows that for any $t\in[0,T]$,
     \begin{equation*}
     \begin{split}
         &\left\|u(t)\right\|^{2}_{L^{2}(\mathbb R^{6}_{x, v})}+\int_{0}^{t}\left\|\langle D_{v}\rangle^{s}\left(\langle v\rangle^{\frac\gamma2}u(\tau)\right)\right\|^{2}_{L^{2}(\mathbb R^{6}_{x, v})}d\tau\\
         &\quad+\int_{0}^{t}\left\|\langle v\rangle^{\frac\gamma2+s}u(\tau)\right\|^{2}_{L^{2}(\mathbb R^6_{x, v})}d\tau\\
         &\le\left(2C_{0}Te^{2TC_{0}}+1\right)\left(2\int_{0}^{t}\left\|f(\tau)\right\|^{2}_{L^{2}(\mathbb R^6_{x, v})}d\tau+\left\|u_{0}\right\|^{2}_{L^{2}(\mathbb R^6_{x, v})}\right).
     \end{split}
     \end{equation*}
\end{proof}

\section{Proof of Theorem \ref{thm1}}

In this section, we give the proof of Theorem \ref{thm1}. Let $\mu\in\mathbb R$, set $\Phi_{\mu}(t, v)=t\langle v\rangle^{\mu}$. 
\begin{lemma}\label{lemma 4.1}
Let $f\in C^{\infty}([0, T]; \mathcal G^{\frac{1}{2\tilde s}}(\mathbb R_{x}^{3}; S^{\frac{1}{2\tilde s}}_{\frac{1}{\gamma/2+s}}(\mathbb R^{3}_{v})))$, then there exist constants $A, \tilde A$ such that for all $t\in[0, T]$
$$\left\|M^{k}_{2\tilde s}f(t)\right\|_{L^{2}(\mathbb R^{6}_{x, v})}\le (T+1)^{2k}A^{k+1}k!,$$
and
$$\left\|\Phi_{\gamma/2+s}^{k}f(t)\right\|_{L^{2}(\mathbb R^{6}_{x, v})}\le T^{k}\tilde A^{k+1}k!,$$
where $A, \tilde A$ is independent of $k$.
\end{lemma}
\begin{proof}
     From Lemma \ref{lemma  2.3}, it follows that 
     $$M_{2\tilde s}(t, D_{x}, D_{v})\sim t\left(1-\Delta_{v}-t^{2}\Delta_{x}\right)^{\tilde s},$$ 
     so that there exists a constant $C_{13}$, depends on $s$, such that  
     $$\left\|M^{k}_{2\tilde s}f(t)\right\|_{L^{2}(\mathbb R^{6})}\le(C_{13}t)^{k}\left\|\left(1-\Delta_{v}-t^{2}\Delta_{x}\right)^{\tilde sk}f(t)\right\|_{L^{2}(\mathbb R^{6})}.$$
     Since $f\in C^{\infty}([0, T]; \mathcal G^{\frac{1}{2\tilde s}}(\mathbb R_{x}^{3}; S^{\frac{1}{2\tilde s}}_{\frac{1}{\gamma/2+s}}(\mathbb R^{3}_{v})))$, we have $f(t)\in\mathcal G^{\frac{1}{2\tilde s}}(\mathbb R_{x, v}^{6})$, then there exist constants $C_{14}, \tilde C_{14}$ such that 
     $$\left\|\left(1-\Delta_{v}-t^{2}\Delta_{x}\right)^{\tilde sk}f(t)\right\|_{L^{2}(\mathbb R^{6}_{x, v})}\le(T+1)^{2\tilde sk}(C_{14})^{k+1}k!,$$
     and
     $$\left\|\left(1+|\cdot|^{2}\right)^{\tilde sk}f(t, x, \cdot)\right\|_{L^{2}(\mathbb R^{3}_{v})}\le(\tilde C_{14})^{k+1}k!.$$
     And therefore, let $A=C_{14}(C_{13}+1)$ and $\tilde A=\tilde C_{14}$, then since $\tilde s\le1/2$, we can obtain that for all $t\in[0, T]$
     $$\left\|M^{k}_{2\tilde s}f(t)\right\|_{L^{2}(\mathbb R^{6})}\le (C_{13}t)^{k}(T+1)^{2\tilde sk}(C_{14})^{k+1}k!\le (T+1)^{2k}A^{k+1}k!,$$
     and 
     $$\left\|\Phi_{\gamma/2+s}^{k}f(t)\right\|_{L^{2}(\mathbb R^{6}_{x, v})}\le T^{k}\tilde A^{k+1}k!.$$
\end{proof}

\begin{prop}\label{prop 4.1}
    For $0<s<1$, $T>0$ and $\gamma+2s>0$. Assume $u_0\in L^{2}(\mathbb R^6_{x, v})$ and $f\in C^{\infty}([0, T]; \mathcal G^{\frac{1}{2\tilde s}}(\mathbb R_{x}^{3}; S^{\frac{1}{2\tilde s}}_{\frac{1}{\gamma/2+s}}(\mathbb R^{3}_{v})))$. Let $u$ be the solution of Cauchy problem \eqref{1-1},  then there exists $B_{1}>0, C_T>0$ such that for any $t\in[0,T]$ and $k\in\mathbb N$,
    \begin{equation}\label{4-1}
    \begin{split}
        &\left\|M^{k}_{2\tilde s}u(t)\right\|^{2}_{L^{2}(\mathbb R^{6}_{x, v})}+\int_{0}^{t}\left\|\langle v\rangle^{\gamma/2+s}M^{k}_{2\tilde s}u(t)\right\|^{2}_{L^{2}(\mathbb R^{6}_{x, v})}d\tau\\
        &\quad+\int_{0}^{t}\left\|\langle D_{v}\rangle^{s}\left(\langle v\rangle^{\gamma/2}M^{k}_{2\tilde s}u(t)\right)\right\|^{2}_{L^{2}(\mathbb R^{6}_{x, v})}d\tau\le\left(B_{1}^{k+1}k!\right)^{2}.
    \end{split}
    \end{equation}
\end{prop}
\begin{proof}
    Fixing $\varphi\ge0$, a $C^{\infty}$ function of compact support, defined in $\mathbb R^{3}_{x}\times\mathbb R^{3}_{v}$, with the properties that $\varphi_{\delta}(x, v)=\delta^{-6}\varphi(\frac{x}{\delta}, \frac{v}{\delta})$ and $\hat\varphi(0)=1$.
%$$\varphi_{\epsilon}(x, v)=\epsilon^{-6}\varphi\left(\frac{x}{\epsilon}, \frac{v}{\epsilon}\right) \quad {\rm and} \quad \int_{\mathbb R^{6}_{x, v}}\varphi(x, v)dxdv=1.$$
Let  $u_{0, \delta}=(u_{0}*\varphi_{\delta})e^{-\delta|v|^{2}}$ with $0<\delta<1$. Since $u_{0}\in L^{2}(\mathbb R^{6}_{x, v})$, from the Leibniz formula, one has for all $m, n\in\mathbb N$ and $r\ge0$
$$\langle v\rangle^{r}\partial^{\alpha}_{x}\partial^{\beta}_{v}u_{0, \delta}=\sum_{\beta'\le\beta}C_{\beta}^{\beta'}\langle v\rangle^{r}\left(u_{0}*\partial^{\alpha}_{x}\partial^{\beta'}_{v}\varphi_{\delta}\right)\partial^{\beta-\beta'}_{v}e^{-\delta|v|^{2}},\quad \forall |\alpha|\le m, |\beta|\le n,$$
then from Minkowski inequality and Young inequality, we can obtain that for all $m, n\in\mathbb N$ and $r\ge0$
%$$\left\|u_{0, \epsilon}\right\|^{2}_{H^{m}_{x}(H^{n}_{r})}\le\sum_{|\alpha|\le m}\sum_{|\beta|\le n}\bigg(\sum_{\beta_{1}+\beta_{2}=\beta}C_{\beta}^{\beta_{1}}\left(\int_{\mathbb R^{6}_{x, v}}\left|\langle v\rangle^{r}\left(u_{0}*\partial^{\alpha}_{x}\partial^{\beta_{1}}_{v}\varphi_{\epsilon}\right)\partial^{\beta_{2}}_{v}e^{-\epsilon|v|^{2}}\right|^{2}dxdv\right)^{\frac12}\bigg)^{2}$$
\begin{equation*}\label{u-delta}
\begin{split}
     &\left\|u_{0, \delta}\right\|^{2}_{H^{m}_{x}(H^{n}_{r})}\\
     &\le\sum_{|\alpha|\le m}\sum_{|\beta|\le n}\bigg(\sum_{\beta'\le\beta}C_{\beta}^{\beta'}\left\|\langle v\rangle^{r}\left(u_{0}*\partial^{\alpha}_{x}\partial^{\beta'}_{v}\varphi_{\delta}\right)\partial^{\beta-\beta'}_{v}e^{-\delta|v|^{2}}\right\|_{L^{2}(\mathbb R^{6}_{x, v})}\bigg)^{2}\\
     &\le\sum_{|\alpha|\le m}\sum_{|\beta|\le n}\bigg(\sum_{\beta'\le\beta}C_{\beta}^{\beta'}\left\|\langle v\rangle^{r}\partial^{\beta-\beta'}_{v}e^{-\delta|v|^{2}}\right\|_{L^{\infty}(\mathbb R^{6}_{x, v})}\\
     &\quad\times\left\|\partial^{\alpha}_{x}\partial^{\beta'}_{v}\varphi_{\delta}\right\|_{L^{1}(\mathbb R^{6}_{x, v})}\left\|u_{0}\right\|_{L^{2}(\mathbb R^{6}_{x, v})}\bigg)^{2}\le C_{\delta}\left\|u_{0}\right\|^{2}_{L^{2}(\mathbb R^{6}_{x, v})},
\end{split}
\end{equation*}
hence $u_{0, \delta}\in H^{\infty}(\mathbb R^{3}_{x}; \mathcal S(\mathbb R^{3}_{v}))$. In particular, for $m=n=r=0$, by using Minkowski inequality, we have 
\begin{equation}\label{u-delta}
\begin{split}
    \left\|u_{0, \delta}\right\|_{L^{2}(\mathbb R^{6}_{x, v})}&=\left(\int_{\mathbb R^{6}_{x, v}}\left|\int_{\mathbb R^{6}}u_{0}(x-x', v-v')\varphi_{\delta}(x ', v ')dx'dv'\right|^{2}\right)^{\frac12}\\
    &\le\int_{\mathbb R^{6}}\left(\int_{\mathbb R^{6}}\left|u_{0}(x-x', v-v')\right|^{2}dxdv\right)^{\frac12}\varphi_{\delta}(x ', v ')dx ' d v '\\
    &=\left\|u_{0}\right\|_{L^{2}(\mathbb R^{6}_{x, v})}\int_{\mathbb R^{6}}\varphi_{\delta}(x ', v ')dx ' d v '=\left\|u_{0}\right\|_{L^{2}(\mathbb R^{6}_{x, v})}.
\end{split}
\end{equation}
Let $u_{\delta}=(u*\varphi_{\delta})e^{-\delta|v|^{2}}$ satisfies
\begin{equation}\label{1-1B}
\left\{
\begin{aligned}
 &\partial_t u_{\delta}+v\cdot\nabla_x u_{\delta}+\langle v\rangle^{\gamma}\langle D_{v}\rangle^{2s}u_{\delta}+\langle v\rangle^{\gamma+2s}u_{\delta}=f,\\
 &u_{\delta}|_{t=0}=u_{0, \delta},
\end{aligned}
\right.
\end{equation}
then by applying Lemma \ref{lemma3.1} and \eqref{u-delta} with $m=n=r=0$, it follows that $u_{\delta}\in C^{\infty}\left([0, T]; H^{\infty}(\mathbb R^{3}_{x}; \mathcal S(\mathbb R^{3}_{v}))\right)$ and for all $0<\delta<1$
    \begin{equation}\label{u-0}
    \begin{split}
        &\left\|u_{\delta}(t)\right\|^{2}_{L^{2}(\mathbb R^{6}_{x, v})}+\int_{0}^{t}\left\|\langle D_{v}\rangle^{s}\left(\langle v\rangle^{\frac\gamma2}u_{\delta}(\tau)\right)\right\|^{2}_{L^{2}(\mathbb R^{6}_{x, v})}d\tau\\
         &\quad+\int_{0}^{t}\left\|\langle v\rangle^{\frac\gamma2+s}u_{\delta}(\tau)\right\|^{2}_{L^{2}(\mathbb R^6_{x, v})}d\tau\\
         &\le\left(2C_{0}Te^{2TC_{0}}+1\right)\left(2\int_{0}^{t}\left\|f(\tau)\right\|^{2}_{L^{2}(\mathbb R^6_{x, v})}d\tau+\left\|u_{0}\right\|^{2}_{L^{2}(\mathbb R^6_{x, v})}\right)=B_{0},
    \end{split}
    \end{equation}
    where $B_0$ is independent of $\delta$.

    We now prove that for all $t\in[0, T]$ and $0<\delta<1$
    \begin{equation}\label{4-1-2}
    \begin{split}
        &\left\|M^{k}_{2\tilde s}u_{\delta}(t)\right\|^{2}_{L^{2}(\mathbb R^{6}_{x, v})}+\int_{0}^{t}\left\|\langle v\rangle^{\gamma/2+s}M^{k}_{2\tilde s}u_{\delta}(t)\right\|^{2}_{L^{2}(\mathbb R^{6}_{x, v})}d\tau\\
        &\quad+\int_{0}^{t}\left\|\langle D_{v}\rangle^{s}\left(\langle v\rangle^{\gamma/2}M^{k}_{2\tilde s}u_{\delta}(t)\right)\right\|^{2}_{L^{2}(\mathbb R^{6}_{x, v})}d\tau\le\left(B_{1}^{k+1}k!\right)^{2},
    \end{split}
    \end{equation}
    with $B_1$ independent of $\delta$ and $k$.

    Since $u_{\delta}$ satisfies \eqref{1-1B}, then
    \begin{equation*}
    \begin{split}
        &\left(\partial_{t}+v\cdot\nabla_{x}\right)M^{k}_{2\tilde s}u_{\delta}+\left[M^{k}_{2\tilde s}, \partial_{t}+v\cdot\nabla_{x}\right]u_{\delta}+\langle v\rangle^{\gamma}\left((1-\Delta_{v})^{s}+\langle v\rangle^{2s}\right)M^{k}_{2\tilde s}u_{\delta}\\
        &\qquad+\left[M^{k}_{2\tilde s}, \langle v\rangle^{\gamma}\left((1-\Delta_{v})^{s}+\langle v\rangle^{2s}\right)\right]u_{\delta}=M^{k}_{2\tilde s}f.
    \end{split}
    \end{equation*}
    By using the Fourier transform, we have
    \begin{equation*}
    \begin{split}
         &\left[M^{k}_{2\tilde s}(t, D_{x}, D_{v}), \partial_{t}+v\cdot\nabla_{x}\right]=\mathcal F^{-1}\left(\left[M^{k}_{2\tilde s}(t, \eta, \xi), \partial_{t}-\eta\cdot\nabla_{\xi}\right]\right)\\
         &=-k\mathcal F^{-1}\left(\langle\xi\rangle^{2\tilde s}M^{k-1}_{2\tilde s}(t, \eta, \xi)\right)=-k(1-\Delta_{v})^{\tilde s}M^{k-1}_{2\tilde s}(t, D_{x}, D_{v}),
    \end{split}
    \end{equation*}
    taking the scalar product with respect to $M^{k}_{2\tilde s}(t, D_{x}, D_{v})u_{\delta}$ in $L^{2}(\mathbb R^{6}_{x, v})$, one has
    \begin{equation}\label{dt-0}
    \begin{split}
        &\frac12\frac{d}{dt}\left\|M^{k}_{2\tilde s}u_{\delta}\right\|^{2}_{L^{2}}+\left\|\langle v\rangle^{\gamma/2+s}M^{k}_{2\tilde s}u_{\delta}\right\|^{2}_{L^{2}}+\left(\langle v\rangle^{\gamma}\langle D_{v}\rangle^{2s}M^{k}_{2\tilde s}u_{\delta}, M^{k}_{2\tilde s}u_{\delta}\right)_{L^{2}}\\
        &=k\left((1-\Delta_{v})^{\tilde s}M^{k-1}_{2\tilde s}u_{\delta}, M^{k}_{2\tilde s}u_{\delta}\right)+\left(M^{k}_{2\tilde s}f, M^{k}_{2\tilde s}u_{\delta}\right)\\
        &\quad+\left(\left[\langle v\rangle^{\gamma}\left((1-\Delta_{v})^{s}+\langle v\rangle^{2s}\right), M^{k}_{2\tilde s}\right]u_{\delta}, M^{k}_{2\tilde s}u_{\delta}\right).
    \end{split}
    \end{equation}
    Noting that
    \begin{equation*}
    \begin{split}
         &\left(\langle v\rangle^{\gamma}\langle D_{v}\rangle^{2s}M^{k}_{2\tilde s}u_{\delta}, M^{k}_{2\tilde s}u_{\delta}\right)\\
         &=\left(\left[\langle v\rangle^{\gamma/2}, \langle D_{v}\rangle^{2s}\right]M^{k}_{2\tilde s}u_{\delta}, \langle v\rangle^{\gamma/2}M^{k}_{2\tilde s}u_{\delta}\right)+\left\|\langle D_{v}\rangle^{s}\left(\langle v\rangle^{\gamma/2}M^{k}_{2\tilde s}u_{\delta}\right)\right\|^{2}_{L^{2}},
    \end{split}
    \end{equation*}
    substituting it into \eqref{dt-0}, and by using the Cauchy-Schwarz inequality, one gets
    \begin{equation}\label{dt}
    \begin{split}
        &\frac12\frac{d}{dt}\left\|M^{k}_{2\tilde s}u_{\delta}(t)\right\|^{2}_{L^{2}(\mathbb R^{6}_{x, v})}+\left\|\langle v\rangle^{\gamma/2+s}M^{k}_{2\tilde s}u_{\delta}(t)\right\|^{2}_{L^{2}(\mathbb R^{6}_{x, v})}\\
        &\quad+\left\|\langle D_{v}\rangle^{s}\left(\langle v\rangle^{\gamma/2}M^{k}_{2\tilde s}u_{\delta}(t)\right)\right\|^{2}_{L^{2}(\mathbb R^{6}_{x, v})}\\
        &\le\left\|M^{k}_{2\tilde s}f(t)\right\|_{L^{2}(\mathbb R^{6}_{x, v})}\left\|M^{k}_{2\tilde s}u_{\delta}(t)\right\|_{L^{2}(\mathbb R^{6}_{x, v})}\\
        &\quad+k\left\|\langle D_{v}\rangle^{\tilde s}M^{k-1}_{2\tilde s}u_{\delta}(t)\right\|_{L^{2}(\mathbb R^{6}_{x, v})}\left\|\langle D_{v}\rangle^{\tilde s}M^{k}_{2\tilde s}u_{\delta}(t)\right\|_{L^{2}(\mathbb R^{6}_{x, v})}\\
        &\quad+\left(\left[\langle D_{v}\rangle^{2s}, \langle v\rangle^{\gamma/2}\right]M^{k}_{2\tilde s}u_{\delta}(t), \langle v\rangle^{\gamma/2}M^{k}_{2\tilde s}u_{\delta}(t)\right)_{L^{2}(\mathbb R^{6}_{x, v})}\\
        &\quad+\left(\left[\langle v\rangle^{\gamma}\langle D_{v}\rangle^{2s}, M^{k}_{2\tilde s}\right]u_{\delta}(t), M^{k}_{2\tilde s}u_{\delta}(t)\right)_{L^{2}(\mathbb R^{6}_{x, v})}\\
        &\quad+\left(\left[\langle v\rangle^{\gamma+2s}, M^{k}_{2\tilde s}\right]u_{\delta}(t), M^{k}_{2\tilde s}u_{\delta}(t)\right)_{L^{2}(\mathbb R^{6}_{x, v})}=I_{1}+I_{2}+I_{3}+I_{4}+I_{5}.
    \end{split}
    \end{equation}

    We show \eqref{4-1-2} holds by induction on the index $k$. For $k=0$, taking $B_{1}\ge B_{0}$, then it is enough to take in \eqref{u-0}. Assume $k\ge1$ and \eqref{4-1-2} is true for $0\le m\le k-1$,
    \begin{equation}\label{m}
    \begin{split}
        &\left\|M^{m}_{2\tilde s}u_{\delta}(t)\right\|^{2}_{L^{2}(\mathbb R^{6}_{x, v})}+\int_{0}^{t}\left\|\langle v\rangle^{\gamma/2+s}M^{m}_{2\tilde s}u_{\delta}(\tau)\right\|^{2}_{L^{2}(\mathbb R^{6}_{x, v})}d\tau\\
        &\quad+\int_{0}^{t}\left\|\langle D_{v}\rangle^{s}\left(\langle v\rangle^{\gamma/2}M^{m}_{2\tilde s}u_{\delta}(\tau)\right)\right\|^{2}_{L^{2}(\mathbb R^{6}_{x, v})}d\tau\le\left(B_{1}^{m+1}m!\right)^{2}.
    \end{split}
    \end{equation}
    Now, we prove that \eqref{m} is true for $m=k$. For the term $I_{1}$, since $\gamma+2s>0$, by using Cauchy-Schwarz inequality, we have
    $$I_{1}\le2\left\|M^{k}_{2\tilde s}f(t)\right\|^{2}_{L^{2}(\mathbb R^{6}_{x, v})}+\frac18\left\|\langle v\rangle^{\gamma/2+s}M^{k}_{2\tilde s}u_{\delta}(t)\right\|^{2}_{L^{2}(\mathbb R^{6}_{x, v})}.$$
    For the term $I_{2}$. By using the Taylor formula, we have
     \begin{equation*}
     \begin{split}
          \left|\langle v\rangle^{\gamma/2}-1\right|&\le\sum_{j=1}^{3}\int_{0}^{1}\left|\partial_{j}\langle\theta v\rangle^{\gamma/2}\right|d\theta|v_{j}|\le C_{\gamma}\int_{0}^{1}\langle\theta v\rangle^{\gamma/2-1}|\theta v|d\theta,%\le C_{\gamma}\int_{0}^{1}\langle\theta v\rangle^{\gamma/2}d\theta,
     \end{split}
    \end{equation*}
    if $\gamma\ge0$, then $\langle\theta v\rangle^{\gamma/2}\le\langle v\rangle^{\gamma/2}$; if $-2s<\gamma<0$, then $\langle\theta v\rangle^{\gamma/2}\le\langle v\rangle^{\gamma/2}\theta^{\gamma/2}$, so that
    \begin{equation*}%\label{gamma}
     \begin{split}
          \left|\langle v\rangle^{\gamma/2}-1\right|\le C_{\gamma}\langle v\rangle^{\gamma/2}\max\left\{1, \int_{0}^{1}\theta^{\gamma/2}d\theta\right\}\le\tilde C_{\gamma}\langle v\rangle^{\gamma/2},
     \end{split}
    \end{equation*}
    here $\tilde C_{\gamma}$ is the constant depends on $\gamma$. Since $0<\tilde s\le1/2$, using \eqref{0-s-1/2}, it follows that for any $p\in\mathbb N$
    \begin{equation*}
    \begin{split}
         &\left\|\langle D_{v}\rangle^{\tilde s}M^{p}_{2\tilde s}u_{\delta}\right\|_{L^{2}(\mathbb R^{6}_{x, v})}\\
         &\le\left\|\langle v\rangle^{\frac\gamma2}\langle D_{v}\rangle^{\tilde s}M^{p}_{2\tilde s}u_{\delta}\right\|_{L^{2}(\mathbb R^{6}_{x, v})}+\left\|\left(1-\langle v\rangle^{\frac\gamma2}\right)\langle D_{v}\rangle^{\tilde s}M^{p}_{2\tilde s}u_{\delta}\right\|_{L^{2}(\mathbb R^{6}_{x, v})}\\
         %&\le\left(\tilde C_{\gamma}+1\right)\left\|\langle v\rangle^{\gamma/2}\langle D_{v}\rangle^{\tilde s}M^{j}_{2\tilde s}u\right\|_{L^{2}(\mathbb R^{6}_{x, v})}\\
         &\le\left(\tilde C_{\gamma}+1\right)\left(\left\|\langle D_{v}\rangle^{\tilde s}\left(\langle v\rangle^{\frac\gamma2}M^{p}_{2\tilde s}u_{\delta}\right)\right\|_{L^{2}(\mathbb R^{6}_{x, v})}+\left\|\left[\langle v\rangle^{\frac\gamma2}, \langle D_{v}\rangle^{\tilde s}\right]M^{p}_{2\tilde s}u_{\delta}\right\|_{L^{2}(\mathbb R^{6}_{x, v})}\right)\\
         &\le\left(\tilde C_{\gamma}+1\right)\left(\left\|\langle D_{v}\rangle^{\tilde s}\left(\langle v\rangle^{\frac\gamma2}M^{p}_{2\tilde s}u_{\delta}\right)\right\|_{L^{2}(\mathbb R^{6}_{x, v})}+C_{1}\left\|\langle v\rangle^{\frac\gamma2}M^{p}_{2\tilde s}u_{\delta}\right\|_{L^{2}(\mathbb R^{6}_{x, v})}\right)\\
         &\le\left(\tilde C_{\gamma}+1\right)\left(C_{1}+1\right)\left\|\langle D_{v}\rangle^{\tilde s}\left(\langle v\rangle^{\frac\gamma2}M^{p}_{2\tilde s}u_{\delta}\right)\right\|_{L^{2}(\mathbb R^{6}_{x, v})}.
    \end{split}
    \end{equation*}
    Thus, by using Cauchy-Schwarz inequality, one has
    \begin{equation*}
    \begin{split}
         I_{2}&\le k\left\|\langle D_{v}\rangle^{\tilde s}M^{k-1}_{2\tilde s}u_{\delta}(t)\right\|_{L^{2}(\mathbb R^{6}_{x, v})}\left\|\langle D_{v}\rangle^{\tilde s}M^{k}_{2\tilde s}u_{\delta}(t)\right\|_{L^{2}(\mathbb R^{6}_{x, v})}\\
         &\le k\left(\tilde C_{\gamma}+1\right)^{2}\left(C_{1}+1\right)^{2}\left\|\langle D_{v}\rangle^{s}\left(\langle v\rangle^{\frac\gamma2}M^{k-1}_{2\tilde s}u_{\delta}(t)\right)\right\|_{L^{2}(\mathbb R^{6}_{x, v})}\\
         &\quad\times\left\|\langle D_{v}\rangle^{s}\left(\langle v\rangle^{\frac\gamma2}M^{k}_{2\tilde s}u_{\delta}(t)\right)\right\|_{L^{2}(\mathbb R^{6}_{x, v})}\\
         &\le2k^{2}\tilde C_{1}\left\|\langle D_{v}\rangle^{s}\left(\langle v\rangle^{\frac\gamma2}M^{k-1}_{2\tilde s}u_{\delta}(t)\right)\right\|^{2}_{L^{2}(\mathbb R^{6}_{x, v})}+\frac18\left\|\langle D_{v}\rangle^{s}\left(\langle v\rangle^{\frac\gamma2}M^{k}_{2\tilde s}u_{\delta}(t)\right)\right\|^{2}_{L^{2}(\mathbb R^{6}_{x, v})},
    \end{split}
    \end{equation*}
    here we use the fact $0<\tilde s\le s$, and the constant $\tilde C_{1}$ depends on $C_{1}$ and $\gamma$.

    For $I_{3}$, since $u_{\delta}\in H^{\infty}(\mathbb R^{3}_{x}; \mathcal S(\mathbb R^{3}_{v}))$, we have $M^{k}_{2\tilde s}u_{\delta}\in H^{\infty}(\mathbb R^{3}_{x}; \mathcal S(\mathbb R^{3}_{v}))$.
    we first restrict to the case $0<s\le1/2$,  by using \eqref{0-s-1/2}, it follows that
    \begin{equation*}
    \begin{split}
         &\left|I_{3}\right|\le C_{1}\left\|\langle v\rangle^{\gamma/2}M^{k}_{2\tilde s}u_{\delta}(t)\right\|^{2}_{L^{2}(\mathbb R^{6}_{x, v})}.
    \end{split}
    \end{equation*}
    Then we consider the case $1/2<s<1$. By using \eqref{1/2-s-1} and the Cauchy- Schwarz inequality, it follows that
    \begin{equation*}
    \begin{split}
         \left|I_{3}\right|%&\le C_{1}\left\|\langle D_{v}\rangle^{2s-1}\left(\langle v\rangle^{\gamma/2}M^{k}_{2\tilde s}u\right)\right\|_{L^{2}(\mathbb R^{6}_{x, v})}\left\|\langle v\rangle^{\gamma/2}M^{k}_{2\tilde s}u\right\|_{L^{2}(\mathbb R^{6}_{x, v})}\\
         &\le C_{1}\left\|\langle D_{v}\rangle^{2s-1}\left(\langle v\rangle^{\gamma/2}M^{k}_{2\tilde s}u_{\delta}(t)\right)\right\|^{2}_{L^{2}(\mathbb R^{6}_{x, v})},
    \end{split}
    \end{equation*}
    here we use the fact $2s-1>0$ if $1/2<s<1$. For the above two cases, since 
    $$M^{k}_{2\tilde s}(t, D_{x}, D_{v})\sim \left(t\left(1-\Delta_{v}-t^{2}\Delta_{x}\right)^{\tilde s}\right)^{k}$$ 
    and $u_{\delta}\in H^{\infty}(\mathbb R^{3}_{x}; \mathcal S(\mathbb R^{3}_{v}))$, it follows that
    $$\left\|\langle v\rangle^{\gamma/2}M^{k}_{2\tilde s}u_{\delta}\right\|^{2}_{L^{2}(\mathbb R^{6}_{x, v})}=\int_{\mathbb R^{3}_{x}}\left\|\langle\cdot\rangle^{\gamma/2}M^{k}_{2\tilde s}u_{\delta}(t, x, \cdot)\right\|^{2}_{L^{2}(\mathbb R^{3}_{v})}dx,$$
    and
    $$\left\|\langle D_{v}\rangle^{2s-1}\left(\langle v\rangle^{\gamma/2}M^{k}_{2\tilde s}u_{\delta}\right)\right\|^{2}_{L^{2}(\mathbb R^{6}_{x, v})}=\int_{\mathbb R^{3}_{x}}\left\|\langle D_{v}\rangle^{2s-1}\left(\langle v\rangle^{\gamma/2}M^{k}_{2\tilde s}u_{\delta}(t, x, \cdot)\right)\right\|^{2}_{L^{2}(\mathbb R^{3}_{v})}dx,$$
    as the argument in Lemma \ref{lemma3.1}, if $\gamma>0$, then applying Lemma \ref{interpolation} with $\epsilon=\frac{1}{2\sqrt{2C_{1}}}$,
    \begin{equation*}
    \begin{split}
        &C_{1}\left\|\langle\cdot\rangle^{\gamma/2}M^{k}_{2\tilde s}u_{\delta}(t, x, \cdot)\right\|^{2}_{L^{2}(\mathbb R^{3}_{v})}\\
        &\le C_{1}\left\|\langle D_{v}\rangle^{s/2}\left(\langle\cdot\rangle^{\gamma/2}M^{k}_{2\tilde s}u_{\delta}(t, x, \cdot)\right)\right\|^{2}_{L^{2}(\mathbb R^{3}_{v})}\\
        &\le C_{1}\left(\frac{1}{2\sqrt{2C_{1}}}\left\|\langle D_{v}\rangle^{s}\left(\langle\cdot\rangle^{\gamma/2}M^{k}_{2\tilde s}u_{\delta}(t, x, \cdot)\right)\right\|_{L^{2}(\mathbb R^{3}_{v})}+\tilde C_{1}\left\|M^{k}_{2\tilde s}u_{\delta}(t, x, \cdot)\right\|_{L^{2}(\mathbb R^{3}_{v})}\right)^{2}\\
        &\le\frac14\left\|\langle D_{v}\rangle^{s}\left(\langle\cdot\rangle^{\gamma/2}M^{k}_{2\tilde s}u_{\delta}(t, x, \cdot)\right)\right\|^{2}_{L^{2}(\mathbb R^{3}_{v})}+2C_{1}\left(\tilde C_{1}\right)^{2}\left\|M^{k}_{2\tilde s}u_{\delta}(t, x, \cdot)\right\|^{2}_{L^{2}(\mathbb R^{3}_{v})},
    \end{split}
    \end{equation*}
    and
    \begin{equation*}
    \begin{split}
        &C_{1}\left\|\langle D_{v}\rangle^{2s-1}\left(\langle\cdot\rangle^{\gamma/2}M^{k}_{2\tilde s}u_{\delta}(t, x, \cdot)\right)\right\|^{2}_{L^{2}(\mathbb R^{3}_{v})}\\
        %&\le C_{1}\left(\frac{1}{2\sqrt{2C_{1}}}\left\|\langle D_{v}\rangle^{s}\left(\langle\cdot\rangle^{\gamma/2}M^{k}_{2\tilde s}u(t, x, \cdot)\right)\right\|_{L^{2}(\mathbb R^{3}_{v})}+\tilde C_{1}\left\|M^{k}_{2\tilde s}u(t, x, \cdot)\right\|_{L^{2}(\mathbb R^{3}_{v})}\right)^{2}\\
        &\le\frac14\left\|\langle D_{v}\rangle^{s}\left(\langle\cdot\rangle^{\gamma/2}M^{k}_{2\tilde s}u_{\delta}(t, x, \cdot)\right)\right\|^{2}_{L^{2}(\mathbb R^{3}_{v})}+2C_{1}\left(\tilde C_{1}\right)^{2}\left\|M^{k}_{2\tilde s}u_{\delta}(t, x, \cdot)\right\|^{2}_{L^{2}(\mathbb R^{3}_{v})},
    \end{split}
    \end{equation*}
    here $\tilde C_{1}$ are the constant that depend on $C_{1}$. If $-2s<\gamma\le0$, using Lemma \ref{interpolation1} and Cauchy-Schwarz inequality, then it follows that
    \begin{equation*}
    \begin{split}
         &C_{1}\left\|\langle\cdot\rangle^{\gamma/2}M^{k}_{2\tilde s}u_{\delta}(t, x, \cdot)\right\|^{2}_{L^{2}(\mathbb R^{3}_{v})}\\
         &\le\frac14\left\|\langle D_{v}\rangle^{s}\left(\langle\cdot\rangle^{\gamma/2}M^{k}_{2\tilde s}u_{\delta}(t, x, \cdot)\right)\right\|^{2}_{L^{2}(\mathbb R^{3}_{v})}+C_{15}\left\|M^{k}_{2\tilde s}u_{\delta}(t, x, \cdot)\right\|^{2}_{L^{2}(\mathbb R^{3}_{v})},
     \end{split}
     \end{equation*}
     and
     \begin{equation*}
    \begin{split}
         &C_{1}\left\|\langle D_{v}\rangle^{2s-1}\left(\langle\cdot\rangle^{\gamma/2}M^{k}_{2\tilde s}u_{\delta}(t, x, \cdot)\right)\right\|^{2}_{L^{2}(\mathbb R^{3}_{v})}\\
         &\le\frac14\left\|\langle D_{v}\rangle^{s}\left(\langle\cdot\rangle^{\gamma/2}M^{k}_{2\tilde s}u_{\delta}(t, x, \cdot)\right)\right\|^{2}_{L^{2}(\mathbb R^{3}_{v})}+C_{15}\left\|M^{k}_{2\tilde s}u_{\delta}(t, x, \cdot)\right\|^{2}_{L^{2}(\mathbb R^{3}_{v})},
     \end{split}
     \end{equation*}
     here $C_{15}$ is the constant that depends on $s$ and $\gamma$. Plugging these back into $I_{3}$, we can obtain that for all $\gamma>-2s$ and $0<s<1$
    $$|I_{3}|\le\frac14\left\|\langle D_{v}\rangle^{s}\left(\langle v\rangle^{\gamma/2}M^{k}_{2\tilde s}u_{\delta}(t)\right)\right\|^{2}_{L^{2}(\mathbb R^{6}_{x, v})}+\tilde C_{15}\left\|M^{k}_{2\tilde s}u_{\delta}(t)\right\|^{2}_{L^{2}(\mathbb R^{6}_{x, v})},$$
    here $\tilde C_{15}=\max\{2C_{1}(\tilde C_{1})^{2}, C_{15}\}$. 
    
    Next, we consider the term $I_{4}$. Set $A_{k}=\langle v\rangle^{-\gamma/2}\left[\langle v\rangle^{\gamma}, M^{k}_{2\tilde s}\right]$, then
    \begin{equation*}
    \begin{split}
    I_{4}%&=\left(\left[\langle v\rangle^{\gamma}\langle D_{v}\rangle^{2s}, M^{k}_{2\tilde s}\right]u, M^{k}_{2\tilde s}u\right)\\
    &=\left(\langle v\rangle^{\gamma}M^{k}_{2\tilde s}\langle D_{v}\rangle^{2s}u_{\delta}, M^{k}_{2\tilde s}u_{\delta}\right)-\left(M^{k}_{2\tilde s}\left(\langle v\rangle^{\gamma}\langle D_{v}\rangle^{2s}u_{\delta}\right), M^{k}_{2\tilde s}u_{\delta}\right)\\
    &=\left(\left[\langle v\rangle^{\gamma}, M^{k}_{2\tilde s}\right]\langle D_{v}\rangle^{2s}u_{\delta}, M^{k}_{2\tilde s}u_{\delta}\right)\\
    %&=\left(\langle D_{v}\rangle^{-s}\left(\langle v\rangle^{-\gamma/2}\left[\langle v\rangle^{\gamma}, M^{k}_{2\tilde s}\right]\langle D_{v}\rangle^{2s}u\right), \langle D_{v}\rangle^{s}\left(\langle v\rangle^{\gamma/2}M^{k}_{2\tilde s}u\right)\right)\\
    &=\left(\langle D_{v}\rangle^{-s}A_{k}\langle D_{v}\rangle^{2s}u_{\delta}, \langle D_{v}\rangle^{s}\left(\langle v\rangle^{\gamma/2}M^{k}_{2\tilde s}u_{\delta}\right)\right).
    \end{split}
    \end{equation*}
    From Lemma \ref{A-k} and Cauchy-Schwarz inequality, one has
   \begin{equation*}
    \begin{split}
        \left|I_{4}\right|&\le 2\left(\sum_{p=0}^{k-1}\left(C_{4}(t+1)\right)^{k-p}C_{k}^{p}\left\|\langle D_{v}\rangle^{-s}\left(\langle v\rangle^{\frac\gamma2}M^{p}_{2\tilde s}\langle D_{v}\rangle^{2s}u_{\delta}(t)\right)\right\|_{L^{2}(\mathbb R^{6}_{x, v})}\right)^{2}\\
        &\quad+\frac18\left\|\langle D_{v}\rangle^{s}\left(\langle v\rangle^{\frac\gamma2}M^{k}_{2\tilde s}u_{\delta}(t)\right)\right\|^{2}_{L^{2}(\mathbb R^{6}_{x, v})},
        %&\le\sum_{p=0}^{k-1}\frac{2}{\left((k-p)!\right)^{2}}\sum_{p=0}^{k-1}\left(\frac{k!\left(C_{4}(t+1)\right)^{k-p}}{p!}\right)^{2}\left\|\langle D_{v}\rangle^{s}\left(\langle v\rangle^{\frac\gamma2}M^{p}_{2\tilde s}u\right)\right\|^{2}_{L^{2}(\mathbb R^{6}_{x, v})}\\
        %&\quad+\frac18\left\|\langle D_{v}\rangle^{s}\left(\langle v\rangle^{\frac\gamma2}M^{k}_{2\tilde s}u\right)\right\|^{2}_{L^{2}(\mathbb R^{6}_{x, v})},
   \end{split}
   \end{equation*}
   from Lemma 2.2 of~\cite{H-4}, it follows that
   $$\left\|\langle D_{v}\rangle^{-s}\left(\langle v\rangle^{\frac\gamma2}M^{p}_{2\tilde s}\langle D_{v}\rangle^{2s}u_{\delta}\right)\right\|_{L^{2}(\mathbb R^{6}_{x, v})}\lesssim\left\|\langle D_{v}\rangle^{s}\left(\langle v\rangle^{\frac\gamma2}M^{p}_{2\tilde s}u_{\delta}\right)\right\|^{2}_{L^{2}(\mathbb R^{6}_{x, v})},$$
   hence, by using Cauchy-Schwarz inequality, we can get
   \begin{equation*}
    \begin{split}
        \left|I_{4}\right|&\le 2\left(\sum_{p=0}^{k-1}(\tilde C_{4}(t+1))^{k-p}C_{k}^{p}\left\|\langle D_{v}\rangle^{s}(\langle v\rangle^{\frac\gamma2}M^{p}_{2\tilde s}u_{\delta}(t))\right\|_{L^{2}(\mathbb R^{6}_{x, v})}\right)^{2}\\
        &\quad+\frac18\left\|\langle D_{v}\rangle^{s}(\langle v\rangle^{\frac\gamma2}M^{k}_{2\tilde s}u_{\delta}(t))\right\|^{2}_{L^{2}(\mathbb R^{6}_{x, v})}.
   \end{split}
   \end{equation*}  
    Now, it remains to consider $I_{5}$. Noting that
    \begin{equation*}
    \begin{split}
         I_{5}%&=\left(\left[\langle v\rangle^{\gamma+2s}, M^{k}_{2\tilde s}\right]u, M^{k}_{2\tilde s}u\right)_{L^{2}(\mathbb R^{6}_{x, v})}\\
         &=\left(\langle v\rangle^{-\gamma/2-s}\left[\langle v\rangle^{\gamma/2+s}, M^{k}_{2\tilde s}\right]u_{\delta}, \langle v\rangle^{\gamma/2+s}M^{k}_{2\tilde s}u_{\delta}\right)_{L^{2}(\mathbb R^{6}_{x, v})},
    \end{split}
   \end{equation*}
   then by using \eqref{2-3-1} with $m=\gamma+2s$ and Cauchy-Schwarz inequality, we have
   \begin{equation*}
    \begin{split}
         \left|I_{5}\right|&\le2\left(\sum_{p=0}^{k-1}\left(C_{2}(t+1)\right)^{k-p}C_{k}^{p}\left\|\langle v\rangle^{\frac\gamma2+s}M^{p}_{2\tilde s}u_{\delta}(t)\right\|_{L^{2}(\mathbb R^{6}_{x, v})}\right)^{2}\\
         &\quad+\frac18\left\|\langle v\rangle^{\frac\gamma2+s}M^{k}_{2\tilde s}u_{\delta}(t)\right\|^{2}_{L^{2}(\mathbb R^{6}_{x, v})}.
         %&\le\sum_{p=0}^{k-1}\frac{2}{\left((k-p)!\right)^{2}}\sum_{p=0}^{k-1}\left(\frac{k!\left(C_{2}(t+1)\right)^{k-p}}{p!}\right)^{2}\left\|\langle v\rangle^{\frac{\gamma}{2}+s}M^{p}_{2\tilde s}u\right\|^{2}_{L^{2}(\mathbb R^{6}_{x, v})}\\
         %&\quad+\frac18\left\|\langle v\rangle^{\frac\gamma2+s}M^{k}_{2\tilde s}u\right\|^{2}_{L^{2}(\mathbb R^{6}_{x, v})}
   \end{split}
   \end{equation*} 
    Substituting $I_{1}-I_{5}$ into \eqref{dt}, we can obtain
    \begin{equation*}
    \begin{split}
        &\frac{d}{dt}\left\|M^{k}_{2\tilde s}u_{\delta}\right\|^{2}_{L^{2}(\mathbb R^{6}_{x, v})}+\left\|\langle v\rangle^{\frac\gamma2+s}M^{k}_{2\tilde s}u_{\delta}\right\|^{2}_{L^{2}(\mathbb R^{6}_{x, v})}+\left\|\langle D_{v}\rangle^{s}\left(\langle v\rangle^{\frac\gamma2}M^{k}_{2\tilde s}u_{\delta}\right)\right\|^{2}_{L^{2}(\mathbb R^{6}_{x, v})}\\
        &\le2\tilde C_{15}\left\|M^{k}_{2\tilde s}u_{\delta}\right\|^{2}_{L^{2}(\mathbb R^{6}_{x, v})}+4k^{2}\tilde C_{1}\left\|\langle D_{v}\rangle^{s}\left(\langle v\rangle^{\frac\gamma2}M^{k-1}_{2\tilde s}u_{\delta}\right)\right\|^{2}_{L^{2}(\mathbb R^{6}_{x, v})}\\
        &\quad+4\left\|M^{k}_{2\tilde s}f\right\|^{2}_{L^{2}(\mathbb R^{6}_{x, v})}+2\left(\sum_{p=0}^{k-1}\left(C_{2}(t+1)\right)^{k-p}C_{k}^{p}\left\|\langle v\rangle^{\frac\gamma2+s}M^{p}_{2\tilde s}u_{\delta}\right\|_{L^{2}(\mathbb R^{6}_{x, v})}\right)^{2}\\
         &\quad+2\left(\sum_{p=0}^{k-1}\left(\tilde C_{4}(t+1)\right)^{k-p}C_{k}^{p}\left\|\langle D_{v}\rangle^{s}\left(\langle v\rangle^{\frac\gamma2}M^{p}_{2\tilde s}u_{\delta}\right)\right\|_{L^{2}(\mathbb R^{6}_{x, v})}\right)^{2}.
    \end{split}
    \end{equation*}
    Integrating from~0~to $t$, one gets for all $0\le t\le T$,
    \begin{equation*}
    \begin{split}
        &\left\|M^{k}_{2\tilde s}u_{\delta}(t)\right\|^{2}_{L^{2}(\mathbb R^{6}_{x, v})}+\int_{0}^{t}\left\|\langle v\rangle^{\frac\gamma2+s}M^{k}_{2\tilde s}u_{\delta}(\tau)\right\|^{2}_{L^{2}(\mathbb R^{6}_{x, v})}d\tau\\
        &\quad+\int_{0}^{t}\left\|\langle D_{v}\rangle^{s}\left(\langle v\rangle^{\frac\gamma2}M^{k}_{2\tilde s}u_{\delta}(\tau)\right)\right\|^{2}_{L^{2}(\mathbb R^{6}_{x, v})}d\tau\\
        &\le2\tilde C_{15}\int_{0}^{t}\left\|M^{k}_{2\tilde s}u_{\delta}(\tau)\right\|^{2}_{L^{2}(\mathbb R^{6}_{x, v})}d\tau+4\int_{0}^{t}\left\|M^{k}_{2\tilde s}f(\tau)\right\|^{2}_{L^{2}(\mathbb R^{6}_{x, v})}d\tau\\
        &\quad+4k^{2}\tilde C_{1}\int_{0}^{t}\left\|\langle D_{v}\rangle^{s}\left(\langle v\rangle^{\frac\gamma2}M^{k-1}_{2\tilde s}u_{\delta}(\tau)\right)\right\|^{2}_{L^{2}(\mathbb R^{6}_{x, v})}d\tau\\
        &\quad+2\int_{0}^{t}\left(\sum_{p=0}^{k-1}\left(C_{2}(\tau+1)\right)^{k-p}C_{k}^{p}\left\|\langle v\rangle^{\frac\gamma2+s}M^{p}_{2\tilde s}u_{\delta}(\tau)\right\|_{L^{2}(\mathbb R^{6}_{x, v})}\right)^{2}d\tau\\
         &\quad+2\int_{0}^{t}\left(\sum_{p=0}^{k-1}\left(\tilde C_{4}(\tau+1)\right)^{k-p}C_{k}^{p}\left\|\langle D_{v}\rangle^{s}\left(\langle v\rangle^{\frac\gamma2}M^{p}_{2\tilde s}u_{\delta}(\tau)\right)\right\|_{L^{2}(\mathbb R^{6}_{x, v})}\right)^{2}d\tau,
    \end{split}
    \end{equation*}
    here we use the fact
    $$
    \left\|M^{k}_{2\tilde s}u_{\delta}(t)\right\|^{2}_{L^{2}(\mathbb R^{6}_{x, v})}\Big|_{t=0}=0,\ \ \forall\ k\ge 1.
    $$
    From Minkowski inequality and \eqref{m}, taking $B_{1}\ge\left(C_{2}(T+1)\right)^{2}+1$, it follows that for any $t\in[0, T]$
    \begin{equation*}
    \begin{split}
        &\int_{0}^{t}\left(\sum_{p=0}^{k-1}\left(C_{2}(\tau+1)\right)^{k-p}C_{k}^{p}\left\|\langle v\rangle^{\frac\gamma2+s}M^{p}_{2\tilde s}u_{\delta}(\tau)\right\|_{L^{2}(\mathbb R^{6}_{x, v})}\right)^{2}d\tau\\
        %&\le\left(\sum_{p=0}^{k-1}\left(C_{2}(t+1)\right)^{k-p}C_{k}^{p}\int_{0}^{t}\left\|\langle v\rangle^{\frac m2}M^{p}_{2\tilde s}u\right\|_{L^{2}(\mathbb R^{6}_{x, v})}d\tau\right)^{2}\\
        &\le \left(\sum_{p=0}^{k-1}\left(C_{2}(T+1)\right)^{k-p}C_{k}^{p}\left(\int_{0}^{t}\left\|\langle v\rangle^{\frac\gamma2+s}M^{p}_{2\tilde s}u_{\delta}(\tau)\right\|^{2}_{L^{2}(\mathbb R^{6}_{x, v})}d\tau\right)^{\frac12}\right)^{2}\\
        &\le\left(\sum_{p=0}^{k-1}\frac{\left(C_{2}(T+1)\right)^{k-p}k!}{(k-p)!}B_{1}^{p+1}+C_{2}(T+1)B_{1}^{k}k!\right)^{2}\\
        &\le\left((3+C_{2}(T+1))B_{1}^{k}k!\right)^{2}
    \end{split}
    \end{equation*}
    here we use the fact $\sum_{n=0}^{\infty}\frac{1}{n!}=e<3$.
    Similarly, taking $B_{1}\ge\left(\tilde C_{4}(T+1)\right)^{2}+1$, we can obtain that
    \begin{equation*}
    \begin{split}
         &\int_{0}^{t}\left(\sum_{p=0}^{k-1}\left(\tilde C_{4}(\tau+1)\right)^{k-p}C_{k}^{p}\left\|\langle D_{v}\rangle^{s}\left(\langle v\rangle^{\frac\gamma2}M^{p}_{2\tilde s}u_{\delta}(\tau)\right)\right\|_{L^{2}(\mathbb R^{6}_{x, v})}\right)^{2}d\tau\\
         &\le \left((3+\tilde C_{4}(T+1))B^{k}_{1}k!\right)^{2}.
    \end{split}
    \end{equation*}
    Combining these results, one gets for all $t\in[0, T]$
    \begin{equation}\label{Integration}
    \begin{split}
        &\left\|M^{k}_{2\tilde s}u_{\delta}(t)\right\|^{2}_{L^{2}(\mathbb R^{6}_{x, v})}+\int_{0}^{t}\left\|\langle v\rangle^{\gamma/2+s}M^{k}_{2\tilde s}u_{\delta}(\tau)\right\|^{2}_{L^{2}(\mathbb R^{6}_{x, v})}d\tau\\
        &\quad+\int_{0}^{t}\left\|\langle D_{v}\rangle^{s}\left(\langle v\rangle^{\gamma/2}M^{k}_{2\tilde s}u_{\delta}(\tau)\right)\right\|^{2}_{L^{2}(\mathbb R^{6}_{x, v})}d\tau\\
        &\le2\tilde C_{15}\int_{0}^{t}\left\|M^{k}_{2\tilde s}u_{\delta}(\tau)\right\|^{2}_{L^{2}(\mathbb R^{6}_{x, v})}d\tau+4\int_{0}^{t}\left\|M^{k}_{2\tilde s}f(\tau)\right\|^{2}_{L^{2}(\mathbb R^{6}_{x, v})}d\tau\\
        &\quad+4C_{16}(B^{k}_{1}k!)^{2},
    \end{split}
    \end{equation}
    here 
    $$C_{16}=\max\left\{\left(3+C_{2}(T+1)\right)^{2}+\tilde C_{1}, (3+\tilde C_{4}(T+1))^{2}+\tilde C_{1}\right\}.$$    
    Applying Gronwall inequality, one has for all $t\in[0, T]$
    \begin{equation*}
    \begin{split}
        &\left\|M^{k}_{2\tilde s}u_{\delta}(t)\right\|^{2}_{L^{2}(\mathbb R^{6}_{x, v})}\le8e^{2\tilde C_{15}T}\left(\int_{0}^{t}\left\|M^{k}_{2\tilde s}f(\tau)\right\|^{2}_{L^{2}(\mathbb R^{6}_{x, v})}d\tau+C_{16}(B^{k}_{1}k!)^{2}\right).
    \end{split}
    \end{equation*}
    Substituting these back into \eqref{Integration}, then using Lemma \ref{lemma 4.1}, one has for all $t\in[0, T]$
    \begin{equation*}
    \begin{split}
        &\left\|M^{k}_{2\tilde s}u_{\delta}(t)\right\|^{2}_{L^{2}(\mathbb R^{6}_{x, v})}+\int_{0}^{t}\left\|\langle v\rangle^{\gamma/2+s}M^{k}_{2\tilde s}u_{\delta}(\tau)\right\|^{2}_{L^{2}(\mathbb R^{6}_{x, v})}d\tau\\
        &\quad+\int_{0}^{t}\left\|\langle D_{v}\rangle^{s}\left(\langle v\rangle^{\gamma/2}M^{k}_{2\tilde s}u_{\delta}(\tau)\right)\right\|^{2}_{L^{2}(\mathbb R^{6}_{x, v})}d\tau\\
        %&\le(k!)^{2}\left(4\tilde C_{13}Te^{2\tilde C_{13}T}+1\right)\\
        &\le\left(16\tilde C_{15}Te^{2\tilde C_{15}T}+1\right)\left(\int_{0}^{t}\left\|M^{k}_{2\tilde s}f(\tau)\right\|^{2}_{L^{2}(\mathbb R^{6}_{x, v})}d\tau+C_{16}(B^{k}_{1}k!)^{2}\right)\\
        &\le\left(16\tilde C_{15}Te^{2\tilde C_{15}T}+1\right)\left(T((T+1)^{2k}A^{k+1}k!)^{2}+C_{16}(B^{k}_{1}k!)^{2}\right).
    \end{split}
    \end{equation*}
    Taking $B_{1}$ such that
    \begin{equation*}
    \begin{split}
         &B_{1}\ge\max\bigg\{B_{0}, A(T+1)^{2}, \sqrt{\left(A^{2}T+C_{16}\right)(16\tilde C_{15}Te^{2\tilde C_{15}T}+1)}\bigg\},
    \end{split}
    \end{equation*}
    it follows that for all $t\in[0, T]$ and $0<\delta<1$
    \begin{equation*}
    \begin{split}
        &\left\|M^{k}_{2\tilde s}u_{\delta}(t)\right\|^{2}_{L^{2}(\mathbb R^{6}_{x, v})}+\int_{0}^{t}\left\|\langle v\rangle^{\gamma/2+s}M^{k}_{2\tilde s}u_{\delta}(\tau)\right\|^{2}_{L^{2}(\mathbb R^{6}_{x, v})}d\tau\\
        &\quad+\int_{0}^{t}\left\|\langle D_{v}\rangle^{s}\left(\langle v\rangle^{\gamma/2}M^{k}_{2\tilde s}u_{\delta}(\tau)\right)\right\|^{2}_{L^{2}(\mathbb R^{6}_{x, v})}d\tau\le(B_{1}^{k+1}k!)^{2}.
    \end{split}
    \end{equation*}
    Since $B_1$ is independent of $0<\delta<1$, by the compactness and uniqueness of solution of \eqref{1-1}, we can obtain that for all $t\in[0, T]$
    \begin{equation*}
    \begin{split}
        &\left\|M^{k}_{2\tilde s}u(t)\right\|^{2}_{L^{2}(\mathbb R^{6}_{x, v})}+\int_{0}^{t}\left\|\langle v\rangle^{\gamma/2+s}M^{k}_{2\tilde s}u(\tau)\right\|^{2}_{L^{2}(\mathbb R^{6}_{x, v})}d\tau\\
        &\quad+\int_{0}^{t}\left\|\langle D_{v}\rangle^{s}\left(\langle v\rangle^{\gamma/2}M^{k}_{2\tilde s}u(\tau)\right)\right\|^{2}_{L^{2}(\mathbb R^{6}_{x, v})}d\tau\le(B_{1}^{k+1}k!)^{2}.
    \end{split}
    \end{equation*}
\end{proof}

Now, we consider the weighted inequality.
\begin{prop}\label{prop 5.1}
    For $0<s<1$, $\gamma+2s>0$ and $T>0$. Assume $u_0\in L^{2}(\mathbb R^6_{x, v})$ and $f\in C^{\infty}([0,T]; H^{\infty}(\mathbb R^{3}_{x}; \mathcal S(\mathbb R^{3}_{v})))$. Let $u$ be the solution of Cauchy problem \eqref{1-1}. Then there exists a constant $B_{2}>0, C_T$ such that for all $k\in\mathbb N$, $0\le t\le T$,
    \begin{equation}\label{5-1}
    \begin{split}
        &\left\|\Phi_{\frac\gamma2+s}^{k}u\right\|^{2}_{L^{2}(\mathbb R^{6}_{x, v})}+\int_{0}^{t}\left\|\langle v\rangle^{\frac\gamma2+s}\Phi_{\frac\gamma2+s}^{k}u\right\|^{2}_{L^{2}(\mathbb R^{6}_{x, v})}d\tau\\
         &\quad+\int_{0}^{t}\left\|\langle D_{v}\rangle^{s}\left(\langle v\rangle^{\frac\gamma2}\Phi_{\frac\gamma2+s}^{k}u\right)\right\|^{2}_{L^{2}(\mathbb R^{6}_{x, v})}d\tau\le\left(B_{2}^{k+1}k!\right)^{2}.
    \end{split}
    \end{equation}
\end{prop}
\begin{proof}
     We first show that
    \begin{equation}\label{5-1-2}
    \begin{split}
        &\left\|\Phi_{\frac\gamma2+s}^{k}u_{\delta}\right\|^{2}_{L^{2}(\mathbb R^{6}_{x, v})}+\int_{0}^{t}\left\|\langle v\rangle^{\frac\gamma2+s}\Phi_{\frac\gamma2+s}^{k}u_{\delta}\right\|^{2}_{L^{2}(\mathbb R^{6}_{x, v})}d\tau\\
         &\quad+\int_{0}^{t}\left\|\langle D_{v}\rangle^{s}\left(\langle v\rangle^{\frac\gamma2}\Phi_{\frac\gamma2+s}^{k}u_{\delta}\right)\right\|^{2}_{L^{2}(\mathbb R^{6}_{x, v})}d\tau\le\left(B_{2}^{k+1}k!\right)^{2},
    \end{split}
    \end{equation}
    with $u_{\delta}$ defined in Proposition \ref{prop 4.1}, and $B_{2}$ independent of $\delta, k$. 
    
    Since $u_{\delta} $ satisfies \eqref{1-1}, it follows that
    \begin{equation*}
    \begin{split}
         &\left(\partial_{t}+v\cdot\nabla_{x}\right)\Phi_{\frac\gamma2+s}^{k}u_{\delta}+\langle v\rangle^{\gamma+2s}\Phi_{\frac\gamma2+s}^{k}u_{\delta}+\langle v\rangle^{\gamma}\langle D_{v}\rangle^{2s}\Phi_{\frac\gamma2+s}^{k}u_{\delta}\\
         &=k\langle v\rangle^{\frac\gamma2+s}\Phi_{\frac\gamma2+s}^{k-1}u_{\delta}+\left[\langle v\rangle^{\gamma}\langle D_{v}\rangle^{2s}, \Phi_{\frac\gamma2+s}^{k}\right]u_{\delta}+\Phi_{\frac\gamma2+s}^{k}f,
    \end{split}
    \end{equation*}
    then taking the scalar product with respect to $\Phi_{\frac\gamma2+s}^{k}u_{\delta}$ in $L^{2}(\mathbb R^{6}_{x, v})$, and by using Cauchy-Schwarz inequality, one has
    \begin{equation}\label{5-1-1}
    \begin{split}
         &\frac12\frac{d}{dt}\left\|\Phi_{\frac\gamma2+s}^{k}u_{\delta}\right\|^{2}_{L^{2}}+\left\|\langle v\rangle^{\frac\gamma2+s}\Phi_{\frac\gamma2+s}^{k}u_{\delta}\right\|^{2}_{L^{2}}+\left(\langle v\rangle^{\gamma}\langle D_{v}\rangle^{2s}\Phi_{\frac\gamma2+s}^{k}u_{\delta}, \Phi_{\frac\gamma2+s}^{k}u_{\delta}\right)\\
         &\le k\left\|\langle v\rangle^{\frac\gamma2+s}\Phi_{\frac\gamma2+s}^{k-1}u_{\delta}\right\|_{L^{2}}\left\|\Phi_{\frac\gamma2+s}^{k}u_{\delta}\right\|_{L^{2}}+\left\|\Phi_{\frac\gamma2+s}^{k}f\right\|_{L^{2}}\left\|\Phi_{\frac\gamma2+s}^{k}u_{\delta}\right\|_{L^{2}}\\
         &\quad+\left(\left[\langle v\rangle^{\gamma}\langle D_{v}\rangle^{2s}, \Phi_{\frac\gamma2+s}^{k}\right]u_{\delta}, \Phi_{\frac\gamma2+s}^{k}u_{\delta}\right)=R_{1}+R_{2}+R_{3}.
    \end{split}
    \end{equation}

    We show \eqref{5-1-2} holds by induction on the index $k$. For $k=0$, it is enough to take in \eqref{u-0}. Assume $k\ge1$ and \eqref{5-1-2} is true for $0\le m\le k-1$,
    \begin{equation}\label{m1}
    \begin{split}
        &\left\|\Phi_{\frac\gamma2+s}^{m}u_{\delta}\right\|^{2}_{L^{2}(\mathbb R^{6}_{x, v})}+\int_{0}^{t}\left\|\langle v\rangle^{\frac\gamma2+s}\Phi_{\frac\gamma2+s}^{m}u_{\delta}\right\|^{2}_{L^{2}(\mathbb R^{6}_{x, v})}d\tau\\
         &\quad+\int_{0}^{t}\left\|\langle D_{v}\rangle^{s}\left(\langle v\rangle^{\frac\gamma2}\Phi_{\frac\gamma2+s}^{m}u_{\delta}\right)\right\|^{2}_{L^{2}(\mathbb R^{6}_{x, v})}d\tau\le\left(B_{2}^{m+1}m!\right)^2.
    \end{split}
    \end{equation}

    Now, we prove that \eqref{5-1-2} holds true for $m=k$. For $R_{1}$ and $R_{2}$, by using Cauchy-Schwarz inequality, it follows that
    \begin{equation*}
    \begin{split}
        |R_{1}|\le\frac14\left\|\Phi_{\frac\gamma2+s}^{k}u_{\delta}\right\|^{2}_{L^{2}(\mathbb R^{6}_{x, v})}+k^{2}\left\|\langle v\rangle^{\frac\gamma2+s}\Phi_{\frac\gamma2+s}^{k-1}u_{\delta}\right\|^{2}_{L^{2}(\mathbb R^{6}_{x, v})},
    \end{split}
    \end{equation*}
    and
    \begin{equation*}
    \begin{split}
        |R_{2}|\le\frac14\left\|\Phi_{\frac\gamma2+s}^{k}u_{\delta}\right\|^{2}_{L^{2}(\mathbb R^{6}_{x, v})}+\left\|\Phi_{\frac\gamma2+s}^{k}f\right\|^{2}_{L^{2}(\mathbb R^{6}_{x, v})}.
    \end{split}
    \end{equation*}
    For the term $R_{3}$, we can rewrite it as
    \begin{equation*}
    \begin{split}
        &R_{3}=\left(\langle v\rangle^{\frac\gamma2}\left[\langle D_{v}\rangle^{2s}, \Phi_{\frac\gamma2+s}^{k}\right]u_{\delta}, \langle v\rangle^{\frac\gamma2}\Phi_{\frac\gamma2+s}^{k}u_{\delta}\right).
        %&=\left(\langle v\rangle^{\frac\gamma2}\left[\langle D_{v}\rangle^{2s}, \Phi_{\frac\gamma2+s}^{k-1}\right]\Phi_{\frac\gamma2+s}u, \langle v\rangle^{\frac\gamma2}\Phi_{\frac\gamma2+s}^{k}u\right)+\left(\langle v\rangle^{\frac\gamma2}\Phi_{\frac\gamma2+s}^{k-1}\left[\langle D_{v}\rangle^{2s}, \Phi_{\frac\gamma2+s}\right]u, \langle v\rangle^{\frac\gamma2}\Phi_{\frac\gamma2+s}^{k}u\right)
    \end{split}
    \end{equation*}
    Noting that
    \begin{equation*}
    \begin{split}
        &\langle v\rangle^{\frac\gamma2}\left[\langle D_{v}\rangle^{2s}, \Phi_{\frac\gamma2+s}^{k}\right]u_{\delta}
        %&=\langle v\rangle^{\frac\gamma2}\left[\langle D_{v}\rangle^{2s}, \Phi_{\frac\gamma2+s}^{k-1}\right]\Phi_{\frac\gamma2+s}u+\langle v\rangle^{\frac\gamma2}\langle D_{v}\rangle^{2s}\Phi_{\frac\gamma2+s}^{k-1}\langle D_{v}\rangle^{-2s}\left[\langle D_{v}\rangle^{2s}, \Phi_{\frac\gamma2+s}\right]u\\
        %&\quad+\langle v\rangle^{\frac\gamma2}\left[\Phi_{\frac\gamma2+s}^{k-1}, \langle D_{v}\rangle^{2s}\right]\langle D_{v}\rangle^{-2s}\left[\langle D_{v}\rangle^{2s}, \Phi_{\frac\gamma2+s}\right]u\\
        =\langle v\rangle^{\frac\gamma2}\langle D_{v}\rangle^{2s}\Phi_{\frac\gamma2+s}^{k-1}A(v, D_{v})\Phi_{\frac\gamma2+s}^{-k}\Phi_{\frac\gamma2+s}^{k}u_{\delta}\\
        &\quad+\langle v\rangle^{\frac\gamma2}\left[\Phi_{\frac\gamma2+s}^{k-1}, \langle D_{v}\rangle^{2s}\right]A(v, D_{v})u_{\delta}+\langle v\rangle^{\frac\gamma2}\left[\langle D_{v}\rangle^{2s}, \Phi_{\frac\gamma2+s}^{k-1}\right]\Phi_{\frac\gamma2+s}u_{\delta},
    \end{split}
    \end{equation*}
    here $A(v, D_{v})=\langle D_{v}\rangle^{-2s}[\langle D_{v}\rangle^{2s}, \Phi_{\frac\gamma2+s}]$, as the arguments in Lemma \ref{lemma2.2}, we have $\Phi_{\frac\gamma2+s}^{k-1}A(v, D_{v})\Phi_{\frac\gamma2+s}^{-k}$ is a pseudo-differential operator of order $-1$,
    %$$\left|A(v, \xi)\right|\le C_{\gamma, s}t\langle v\rangle^{\frac\gamma2+s}\langle\xi\rangle^{-1}=C_{\gamma, s}\Phi_{\frac\gamma2+s}\langle\xi\rangle^{-1},$$
    hence from Lemma 2.2 of~\cite{H-4}, it follows that
    $$\left\|\langle v\rangle^{\frac\gamma2}\langle D_{v}\rangle^{2s}\Phi_{\frac\gamma2+s}^{k-1}A(v, D_{v})\Phi_{\frac\gamma2+s}^{-(k-1)}\Phi_{\frac\gamma2+s}^{k-1}u_{\delta}\right\|_{L^{2}}\le C_{\gamma, s}\left\|\langle v\rangle^{\frac\gamma2}\langle D_{v}\rangle^{2s-1}\Phi_{\frac\gamma2+s}^{k}u_{\delta}\right\|_{L^{2}}.$$
    Therefore as the arguments in Lemma \ref{commutator M}, by induction on $k$, we have
     \begin{equation*}
     \begin{split}
          &\left\|\langle v\rangle^{\frac\gamma2}\left[\langle D_{v}\rangle^{2s}, \Phi_{\frac\gamma2+s}^{k}\right]u_{\delta}\right\|_{L^{2}(\mathbb R^{6}_{x, v})}\\
          &\le\sum_{j=1}^{k}\left(C_{17}t\right)^{k-(j-1)}C_{k}^{j}\left\|\langle v\rangle^{\frac\gamma2}\langle D_{v}\rangle^{2s-1}\Phi_{\frac\gamma2+s}^{j}u_{\delta}\right\|_{L^{2}(\mathbb R^{6}_{x, v})}.
     \end{split}
     \end{equation*}
    Therefore, using the Cauchy-Schwarz inequality and Lemma 2.2 of~\cite{H-4}, it follows that for $0<s\le1/2$
    \begin{equation*}
     \begin{split}
          \left|R_{3}\right|%&\le\left\|\left[\langle D_{v}\rangle^{2s}, \langle v\rangle^{\frac\gamma2}\Phi_{\frac\gamma2+s}^{k}\right]u\right\|_{L^{2}(\mathbb R^{6}_{x, v})}\left\|\langle v\rangle^{\frac\gamma2}\Phi_{\frac\gamma2+s}^{k}u\right\|_{L^{2}(\mathbb R^{6}_{x, v})}\\
          &\le\tilde C_{17}\sum_{j=1}^{k}\left(C_{17}t\right)^{k-(j-1)}C_{k}^{j}\left\|\langle v\rangle^{\frac\gamma2}\Phi_{\frac\gamma2+s}^{j}u_{\delta}\right\|_{L^{2}(\mathbb R^{6}_{x, v})}\left\|\langle v\rangle^{\frac\gamma2}\Phi_{\frac\gamma2+s}^{k}u_{\delta}\right\|_{L^{2}(\mathbb R^{6}_{x, v})}\\
          &\le\left(\tilde C_{17}\sum_{j=1}^{k-1}\left(C_{17}t\right)^{k-(j-1)}C_{k}^{j}\left\|\langle D_{v}\rangle^{s}\left(\langle v\rangle^{\frac\gamma2}\Phi_{\frac\gamma2+s}^{j}u_{\delta}\right)\right\|_{L^{2}(\mathbb R^{6}_{x, v})}\right)^{2}\\
          &\quad+(C_{17}\tilde C_{17}+1)(t+1)\left\|\langle v\rangle^{\frac\gamma2}\Phi_{\frac\gamma2+s}^{k}u_{\delta}\right\|^{2}_{L^{2}(\mathbb R^{6}_{x, v})},
     \end{split}
    \end{equation*}
    and for $1/2<s<1$, we have $2s-1<2s-s=s$, then
    \begin{equation*}
     \begin{split}
          \left|R_{3}\right|%&\le\left\|\left[\langle D_{v}\rangle^{2s}, \langle v\rangle^{\frac\gamma2}\Phi_{\frac\gamma2+s}^{k}\right]u\right\|_{L^{2}(\mathbb R^{6}_{x, v})}\left\|\langle v\rangle^{\frac\gamma2}\Phi_{\frac\gamma2+s}^{k}u\right\|_{L^{2}(\mathbb R^{6}_{x, v})}\\
          &\le\left(\tilde C_{17}\sum_{j=1}^{k-1}\left(C_{17}t\right)^{k-(j-1)}C_{k}^{j}\left\|\langle D_{v}\rangle^{s}\left(\langle v\rangle^{\frac\gamma2}\Phi_{\frac\gamma2+s}^{j}u_{\delta}\right)\right\|_{L^{2}(\mathbb R^{6}_{x, v})}\right)^{2}\\
          &\quad+(C_{17}\tilde C_{17}+1)(t+1)\left\|\langle D_{v}\rangle^{2s-1}\left(\langle v\rangle^{\frac\gamma2}\Phi_{\frac\gamma2+s}^{k}u_{\delta}\right)\right\|^{2}_{L^{2}(\mathbb R^{6}_{x, v})},
     \end{split}
    \end{equation*}
    with $C_{17}$ and $\tilde C_{17}$ depend on $\gamma$, $s$. For both two cases, by using Lemma \ref{interpolation} for $\gamma>0$ and by using Lemma \ref{interpolation1} for $-2s<\gamma\le0$, we obtain that for all $0<s<1$
     \begin{equation*}
     \begin{split}
          |R_{3}|&\le\frac14\left\|\langle D_{v}\rangle^{s}\left(\langle v\rangle^{\frac\gamma2}\Phi_{\frac\gamma2+s}^{k}u_{\delta}\right)\right\|^{2}_{L^{2}(\mathbb R^{6}_{x, v})}+C_{18}(t+1)\left\|\Phi_{\frac\gamma2+s}^{k}u_{\delta}\right\|^{2}_{L^{2}(\mathbb R^{6}_{x, v})}\\
          &\quad+\left(\tilde C_{17}\sum_{j=1}^{k-1}\left(C_{17}t\right)^{k-(j-1)}C_{k}^{j}\left\|\langle D_{v}\rangle^{s}\left(\langle v\rangle^{\frac\gamma2}\Phi_{\frac\gamma2+s}^{j}u_{\delta}\right)\right\|_{L^{2}(\mathbb R^{6}_{x, v})}\right)^{2},
    \end{split}
    \end{equation*}
     with $C_{18}$ depends on $C_{17}$ and $\tilde C_{17}$.

     It remains to consider the third term on the left-hand side of \eqref{5-1-1}. From Cauchy-Schwarz inequality and Lemma 2.2 of~\cite{H-4},
     \begin{equation*}
     \begin{split}
          &\left(\langle v\rangle^{\gamma}\langle D_{v}\rangle^{2s}\Phi_{\frac\gamma2+s}^{k}u_{\delta}, \Phi_{\frac\gamma2+s}^{k}u_{\delta}\right)=\left(\langle D_{v}\rangle^{-s}\langle v\rangle^{\frac\gamma2}\langle D_{v}\rangle^{2s}\Phi_{\frac\gamma2+s}^{k}u_{\delta}, \langle D_{v}\rangle^{s}\langle v\rangle^{\frac\gamma2}\Phi_{\frac\gamma2+s}^{k}u_{\delta}\right)\\
          &=\left\|\langle D_{v}\rangle^{s}(\langle v\rangle^{\frac\gamma2}\Phi_{\frac\gamma2+s}^{k}u_{\delta})\right\|^{2}_{L^{2}}+\left(\langle D_{v}\rangle^{-s}[\langle v\rangle^{\frac\gamma2}, \langle D_{v}\rangle^{2s}]\Phi_{\frac\gamma2+s}^{k}u_{\delta}, \langle D_{v}\rangle^{s}\langle v\rangle^{\frac\gamma2}\Phi_{\frac\gamma2+s}^{k}u_{\delta}\right),
     \end{split}
    \end{equation*}
    as the arguments in Lemma \ref{lemma2.2}, there exists $C_{19}$ depends on $\gamma$, $s$ such that
    $$\left\|\langle D_{v}\rangle^{-s}[\langle v\rangle^{\frac\gamma2}, \langle D_{v}\rangle^{2s}]\Phi_{\frac\gamma2+s}^{k}u_{\delta}\right\|_{L^{2}(\mathbb R^{6}_{x, v})}\le C_{19}\left\|\langle v\rangle^{\frac\gamma2}\Phi_{\frac\gamma2+s}^{k}u_{\delta}\right\|_{L^{2}(\mathbb R^{6}_{x, v})},$$
    hence from Cauchy-Schwarz inequality, we have
    \begin{equation*}
     \begin{split}
          &\left|\left(\langle D_{v}\rangle^{-s}[\langle v\rangle^{\frac\gamma2}, \langle D_{v}\rangle^{2s}]\Phi_{\frac\gamma2+s}^{k}u_{\delta}, \langle D_{v}\rangle^{s}\langle v\rangle^{\frac\gamma2}\Phi_{\frac\gamma2+s}^{k}u_{\delta}\right)\right|\\
          &\le C_{19}\left\|\langle v\rangle^{\frac\gamma2}\Phi_{\frac\gamma2+s}^{k}u_{\delta}\right\|_{L^{2}(\mathbb R^{6}_{x, v})}\left\|\langle D_{v}\rangle^{s}\left(\langle v\rangle^{\frac\gamma2}\Phi_{\frac\gamma2+s}^{k}u_{\delta}\right)\right\|_{L^{2}(\mathbb R^{6}_{x, v})}\\
          &\le2(C_{19})^{2}\left\|\langle v\rangle^{\frac\gamma2}\Phi_{\frac\gamma2+s}^{k}u_{\delta}\right\|^{2}_{L^{2}(\mathbb R^{6}_{x, v})}+\frac18\left\|\langle D_{v}\rangle^{s}\left(\langle v\rangle^{\frac\gamma2}\Phi_{\frac\gamma2+s}^{k}u_{\delta}\right)\right\|^{2}_{L^{2}(\mathbb R^{6}_{x, v})},
     \end{split}
    \end{equation*}
    for $\gamma>0$, using Lemma \ref{interpolation}, and for $-2s<\gamma\le0$, using Lemma \ref{interpolation1}, one has
    \begin{equation*}
    \begin{split}
        &2(C_{19})^{2}\left\|\langle v\rangle^{\frac\gamma2}\Phi_{\frac\gamma2+s}^{k}u_{\delta}\right\|^{2}_{L^{2}}\le\frac18\left\|\langle D_{v}\rangle^{s}\left(\langle v\rangle^{\frac\gamma2}\Phi_{\frac\gamma2+s}^{k}u_{\delta}\right)\right\|^{2}_{L^{2}}+\tilde C_{19}\left\|\Phi_{\frac\gamma2+s}^{k}u_{\delta}\right\|^{2}_{L^{2}}.
    \end{split}
    \end{equation*}
    Plugging these results back into \eqref{5-1-1} and using the fact $\gamma+2s>0$, one gets
     \begin{equation*}
     \begin{split}
         &\frac{d}{dt}\left\|\Phi_{\frac\gamma2+s}^{k}u_{\delta}(t)\right\|^{2}_{L^{2}(\mathbb R^{6}_{x, v})}+\left\|\langle v\rangle^{\frac\gamma2+s}\Phi_{\frac\gamma2+s}^{k}u_{\delta}(t)\right\|^{2}_{L^{2}(\mathbb R^{6}_{x, v})}\\
         &\quad+\left\|\langle D_{v}\rangle^{s}\left(\Phi_{\frac\gamma2+s}^{k}\langle v\rangle^{\frac\gamma2}u_{\delta}(t)\right)\right\|^{2}_{L^{2}(\mathbb R^{6}_{x, v})}\\
         &\le2k^{2}\left\|\langle v\rangle^{\frac\gamma2+s}\Phi_{\frac\gamma2+s}^{k-1}u_{\delta}(t)\right\|^{2}_{L^{2}(\mathbb R^{6}_{x, v})}+2\left\|\Phi_{\frac\gamma2+s}^{k}f(t)\right\|^{2}_{L^{2}(\mathbb R^{6}_{x, v})}\\
         &\quad+2(C_{18}+\tilde C_{19})(t+1)\left\|\Phi_{\frac\gamma2+s}^{k}u_{\delta}(t)\right\|^{2}_{L^{2}(\mathbb R^{6}_{x, v})}\\
         &\quad+2\left(\tilde C_{17}\sum_{j=1}^{k-1}\left(C_{17}t\right)^{k-(j-1)}C_{k}^{j}\left\|\langle D_{v}\rangle^{s}\left(\langle v\rangle^{\frac\gamma2}\Phi_{\frac\gamma2+s}^{j}u_{\delta}(t)\right)\right\|_{L^{2}(\mathbb R^{6}_{x, v})}\right)^{2}.
    \end{split}
    \end{equation*}
    Integrating from 0 to $t$ and using \eqref{m1}, we obtain
     \begin{equation*}
     \begin{split}
         &\left\|\Phi_{\frac\gamma2+s}^{k}u_{\delta}(t)\right\|^{2}_{L^{2}(\mathbb R^{6}_{x, v})}+\int_{0}^{t}\left\|\langle v\rangle^{\frac\gamma2+s}\Phi_{\frac\gamma2+s}^{k}u_{\delta}(\tau)\right\|^{2}_{L^{2}(\mathbb R^{6}_{x, v})}d\tau\\
         &\quad+\int_{0}^{t}\left\|\langle D_{v}\rangle^{s}\left(\Phi_{\frac\gamma2+s}^{k}\langle v\rangle^{\frac\gamma2}u_{\delta}(\tau)\right)\right\|^{2}_{L^{2}(\mathbb R^{6}_{x, v})}d\tau\\
         &\le2\left(B_{2}^{k}k!\right)^{2}+2\int_{0}^{t}\left\|\Phi_{\frac\gamma2+s}^{k}
         f(\tau)\right\|^{2}_{L^{2}(\mathbb R^{6}_{x, v})}d\tau\\
         &\quad+2(C_{18}+\tilde C_{19})(t+1)\int_{0}^{t}\left\|\Phi_{\frac\gamma2+s}^{k}u_{\delta}(\tau)\right\|^{2}_{L^{2}(\mathbb R^{6}_{x, v})}d\tau\\
         &\quad+2\int_{0}^{t}\left(\tilde C_{17}\sum_{j=1}^{k-1}\left(C_{17}\tau\right)^{k-(j-1)}C_{k}^{j}\left\|\langle D_{v}\rangle^{s}\left(\langle v\rangle^{\frac\gamma2}\Phi_{\frac\gamma2+s}^{j}u_{\delta}(\tau)\right)\right\|_{L^{2}(\mathbb R^{6}_{x, v})}\right)^{2}d\tau.
    \end{split}
    \end{equation*}
    From the Minkowski inequality and \eqref{m1}, taking $B_{2}\ge(C_{16}T)^{3}+1$, it follows that for any $t\in[0, T]$%H${\rm\ddot o}$lder inequality,
    \begin{equation*}
     \begin{split}
          &\int_{0}^{t}\left(\tilde C_{17}\sum_{j=1}^{k-1}\left(C_{17}\tau\right)^{k-(j-1)}C_{k}^{j}\left\|\langle D_{v}\rangle^{s}\left(\langle v\rangle^{\frac\gamma2}\Phi_{\frac\gamma2+s}^{j}u_{\delta}(\tau)\right)\right\|_{L^{2}(\mathbb R^{6}_{x, v})}\right)^{2}d\tau\\
          &\le\left(\tilde C_{17}\sum_{j=1}^{k-1}\left(C_{17}T\right)^{k-(j-1)}C_{k}^{j}\left(\int_{0}^{t}\left\|\langle D_{v}\rangle^{s}\left(\langle v\rangle^{\frac\gamma2}\Phi_{\frac\gamma2+s}^{j}u_{\delta}(\tau)\right)\right\|^{2}_{L^{2}(\mathbb R^{6}_{x, v})}d\tau\right)^{\frac12}\right)^{2}\\
          &\le\left(\tilde C_{17}\sum_{j=1}^{k-3}\frac{k!\left(C_{17}T\right)^{k-(j-1)}B_{2}^{j+1}}{(k-j)!}+\tilde C_{17}(C_{17}T)^{3}B_{2}^{k-1}k!+\tilde C_{17}(C_{17}T)^{2}B_{2}^{k}k!\right)^{2}\\
          &\le\left(\tilde C_{17}(4+(C_{17}T)^{2})B_{2}^{k}k!\right)^{2}.
    \end{split}
    \end{equation*}
    Combining these inequalities, we can get for all $0\le t\le T$,
     \begin{equation*}
     \begin{split}
         &\left\|\Phi_{\frac\gamma2+s}^{k}u_{\delta}(t)\right\|^{2}_{L^{2}(\mathbb R^{6}_{x, v})}+\int_{0}^{t}\left\|\langle v\rangle^{\frac\gamma2+s}\Phi_{\frac\gamma2+s}^{k}u_{\delta}\right\|^{2}_{L^{2}(\mathbb R^{6}_{x, v})}d\tau\\
         &\quad+\int_{0}^{t}\left\|\langle D_{v}\rangle^{s}\left(\Phi_{\frac\gamma2+s}^{k}\langle v\rangle^{\frac\gamma2}u_{\delta}\right)\right\|^{2}_{L^{2}(\mathbb R^{6}_{x, v})}d\tau\\
         %&\le2\left(B_{2}^{k}k!\right)^{2}+2\int_{0}^{t}\left\|\Phi_{\frac\gamma2+s}^{k}f(\tau)\right\|^{2}_{L^{2}(\mathbb R^{6}_{x, v})}d\tau\\
         %&\quad+2(C_{18}+\tilde C_{19})(t+1)\int_{0}^{t}\left\|\Phi_{\frac\gamma2+s}^{k}u_{\delta}\right\|^{2}_{L^{2}(\mathbb R^{6}_{x, v})}d\tau+\left(\tilde C_{17}(4+(C_{17}T)^{2})B_{2}^{k}k!\right)^{2}\\
         &\le C_{20}\left(B_{2}^{k}k!\right)^{2}+2\int_{0}^{t}\left\|\Phi_{\frac\gamma2+s}^{k}f(\tau)\right\|^{2}_{L^{2}(\mathbb R^{6}_{x, v})}d\tau+\tilde C_{20}\int_{0}^{t}\left\|\Phi_{\frac\gamma2+s}^{k}u_{\delta}\right\|^{2}_{L^{2}(\mathbb R^{6}_{x, v})}d\tau,
    \end{split}
    \end{equation*}
    here $C_{20}=2+\left(\tilde C_{17}(4+(C_{17}T)^{2})\right)^{2}$ and $\tilde C_{20}=2(C_{18}+\tilde C_{19})(T+1)$. Then by applying Gronwall inequality, one has for all $t\in[0, T]$
    \begin{equation*}
    \begin{split}
         &\left\|\Phi_{\gamma/2+s}^{k}u_{\delta}\right\|^{2}_{L^{2}(\mathbb R^{6}_{x, v})}\le2e^{2\tilde C_{20}T}\left(C_{20}\left(B_{2}^{k}k!\right)^{2}+2\int_{0}^{t}\left\|\Phi_{\frac\gamma2+s}^{k}f(\tau)\right\|^{2}_{L^{2}(\mathbb R^{6}_{x, v})}d\tau\right),
    \end{split}
    \end{equation*}
    plugging it back into the above inequality, and using Lemma \ref{lemma 4.1}, one has
    \begin{equation*}
     \begin{split}
         &\left\|\Phi_{\frac\gamma2+s}^{k}u_{\delta}(t)\right\|^{2}_{L^{2}(\mathbb R^{6}_{x, v})}+\int_{0}^{t}\left\|\langle v\rangle^{\frac\gamma2+s}\Phi_{\frac\gamma2+s}^{k}u_{\delta}\right\|^{2}_{L^{2}(\mathbb R^{6}_{x, v})}d\tau\\
         &\quad+\int_{0}^{t}\left\|\langle D_{v}\rangle^{s}\left(\Phi_{\frac\gamma2+s}^{k}\langle v\rangle^{\frac\gamma2}u_{\delta}\right)\right\|^{2}_{L^{2}(\mathbb R^{6}_{x, v})}d\tau\\
         &\le\left(2T\tilde C_{20}e^{2\tilde C_{20}T}+1\right)\left(C_{20}\left(B_{2}^{k}k!\right)^{2}+2\int_{0}^{t}\left\|\Phi_{\frac\gamma2+s}^{k}f(\tau)\right\|^{2}_{L^{2}(\mathbb R^{6}_{x, v})}d\tau\right)\\
         &\le\left(2T\tilde C_{20}e^{2\tilde C_{20}T}+1\right)\left(C_{20}\left(B_{2}^{k}k!\right)^{2}+2T(T^{k}\tilde A^{k+1}k!)^{2}\right).
    \end{split}
    \end{equation*}
    Then taking $B_{2}$ such that
    \begin{equation*}
    \begin{split}
         &B_{2}\ge\max\bigg\{B_{0}, \tilde AT, \sqrt{\left(2\tilde C_{20}Te^{2\tilde C_{20}T}+1\right)\left(2\tilde A^{2}T+C_{20}\right)}\bigg\},
    \end{split}
    \end{equation*}
    it follows that for all $t\in[0, T]$ and $0<\delta<1$
    \begin{equation*}
     \begin{split}
         &\left\|\Phi_{\gamma/2+s}^{k}u_{\delta}\right\|^{2}_{L^{2}(\mathbb R^{6}_{x, v})}+\int_{0}^{t}\left\|\langle v\rangle^{\gamma/2+s}\Phi_{\gamma/2+s}^{k}u_{\delta}\right\|^{2}_{L^{2}(\mathbb R^{6}_{x, v})}d\tau\\
         &\quad+\int_{0}^{t}\left\|\langle D_{v}\rangle^{s}\left(\Phi_{\gamma/2+s}^{k}\langle v\rangle^{\gamma/2}u_{\delta}\right)\right\|^{2}_{L^{2}(\mathbb R^{6}_{x, v})}d\tau\le\left(B_{2}^{k+1}k!\right)^{2}.
    \end{split}
    \end{equation*}
    Since $B_2$ is independent of $0<\delta<1$, by the compactness and uniqueness of solution of \eqref{1-1}, we can obtain that for all $t\in[0, T]$
    \begin{equation*}
     \begin{split}
         &\left\|\Phi_{\gamma/2+s}^{k}u\right\|^{2}_{L^{2}(\mathbb R^{6}_{x, v})}+\int_{0}^{t}\left\|\langle v\rangle^{\gamma/2+s}\Phi_{\gamma/2+s}^{k}u\right\|^{2}_{L^{2}(\mathbb R^{6}_{x, v})}d\tau\\
         &\quad+\int_{0}^{t}\left\|\langle D_{v}\rangle^{s}\left(\Phi_{\gamma/2+s}^{k}\langle v\rangle^{\gamma/2}u\right)\right\|^{2}_{L^{2}(\mathbb R^{6}_{x, v})}d\tau\le\left(B_{2}^{k+1}k!\right)^{2}.
    \end{split}
    \end{equation*}
\end{proof}

\bigskip
\noindent
{\bf Proof of Theorem \ref{thm1} }

%\vskip 1cm
\bigskip

From \eqref{4-1}, we have
        $$\left\|M^{k}_{2\tilde s}u(t)\right\|_{L^{2}(\mathbb R^{6}_{x, v})}\le B_{1}^{k+1}k!, \quad \forall t\in[0, T].$$
        Using the Plancherel theorem and Lemma \ref{lemma  2.3}, it follows that
        \begin{equation*}
        \begin{split}
             \left\|M^{k}_{2\tilde s}u(t)\right\|^{2}_{L^{2}(\mathbb R^{6}_{x, v})}&=\int_{\mathbb R^{6}_{x, v}}\left|M^{k}_{2\tilde s}(t, \eta, \xi)\hat u(t, \eta, \xi)\right|^{2}\frac{d\eta}{(2\pi)^{3}}\frac{d\xi}{(2\pi)^{3}}\\
             &\ge\left(c_{s}t\right)^{2k}\int_{\mathbb R^{6}_{x, v}}\left|\left(1+|\xi|^{2}+t^{2}|\eta|^{2}\right)^{\tilde sk}\hat u(t, \eta, \xi)\right|^{2}\frac{d\eta}{(2\pi)^{3}}\frac{d\xi}{(2\pi)^{3}}.
        \end{split}
        \end{equation*}
        Hence, for any $t>0$, we have
        \begin{equation*}
        \begin{split}
             &\int_{\mathbb R^{6}_{x, v}}\left|\left(1+|\xi|^{2}+t^{2}|\eta|^{2}\right)^{\tilde sk}\hat u(t, \eta, \xi)\right|^{2}\frac{d\eta}{(2\pi)^{3}}\frac{d\xi}{(2\pi)^{3}}\le \left(\frac{1}{t^{k}}\left(\frac{B_{1}}{c_{s}}\right)^{k+1}k!\right)^{2},
        \end{split}
        \end{equation*}
        by the Plancherel theorem, it follows that  for all $t\in]0, T]$
        \begin{equation*}
        \begin{split}
             &\left\|(1-\Delta_{v}-t^{2}\Delta_{x})^{\tilde sk}u\right\|_{L^{2}(\mathbb R^{6}_{x, v})}\le\frac{1}{t^{k}}\left(\frac{B_{1}}{c_{s}}\right)^{k+1}k!, \quad\forall k\in\mathbb N.
        \end{split}
        \end{equation*}
        From \eqref{5-1}, we can get that for all $t\in]0, T]$
        $$\left\|\langle v\rangle^{(\gamma/2+s)k}u\right\|_{L^{2}(\mathbb R^{6}_{x, v})}\le\frac{B_{2}^{k+1}}{t^{k}}k!, \quad\forall k\in\mathbb N.$$
        And therefore, combing these results and taking $C=\frac{B_{1}}{c_{s}}+B_{2}$, it follows that for all $t\in]0, T]$
        $$
        \left\|(1-\Delta_{v}-t^{2}\Delta_{x})^{\tilde sk}u\right\|_{L^{2}(\mathbb R^{6}_{x, v})}+\left\|\langle v\rangle^{(\gamma/2+s)k}u\right\|_{L^{2}(\mathbb R^{6}_{x, v})}\le\frac{ C^{k+1}}{t^{k}}k!, \quad\forall k\in\mathbb N.
        $$

\bigskip
\noindent {\bf Acknowledgements.}
This work was supported by the NSFC (No.12031006) and the Fundamental
Research Funds for the Central Universities of China.

\end{document}